\documentclass[11pt, a4paper]{amsart}
\usepackage{cgeometry}
\usepackage{multirow}
\usepackage{manfnt}

\makeatletter
\def\hang{\hangindent\parindent}
\def\d@nger{\medbreak\begingroup\clubpenalty=10000
  \def\par{\endgraf\endgroup\medbreak} \noindent
  \hang\hangafter=-2
  \hbox to0pt{\hskip-\hangindent\dbend\hfill}}
\outer\def\danger{\d@nger}
\makeatother

\bibliography{ComplexGeometry}
\title{Transcendental b-divisors II --- Monotonicity theorem}
\author{Mingchen Xia}

\begin{document}

\begin{abstract}
In this paper, we develop the general intersection theory of nef b-divisors, extending the movable intersection theory. 
We define a notion of restricted volume of b-divisors and prove a quantitative version of the monotonicity of the intersection product. As a consequence, we prove a number of new volume inequalities of currents and cohomology classes.
\end{abstract}    

\maketitle

\tableofcontents

\section{Introduction}
The current paper is a sequel of the author's previous paper \cite{Xiabdiv}.
\subsection{Motivation}
Let $X$ be a compact Kähler manifold of dimension $n$.
This paper is motivated by the monotonicity theorem of Witt Nyström (see \cite{WN19, DDNL18mono}): Given closed positive $(1,1)$-currents
\[
T_1,\ldots,T_n,  S_1,\ldots,S_n
\]
on $X$ such that $T_i$ and $S_i$ lie in the same cohomology class and $T_i$ is more singular than $S_i$ for each $i$, then one has
\begin{equation}\label{eq:mono_intro}
    \vol \left( S_1,\dots, S_n \right) \geq \vol \left( T_1, \dots, T_n \right).
\end{equation}
Here the mixed volume is understood in the sense of Cao \cite{Cao14}. As proved in the previous paper \cite{Xiabdiv}, this product also coincides the mixed volume in the sense of Darvas--Xia.

Inequality~\eqref{eq:mono_intro} expresses the principle that increasing singularities of currents lead to a loss of mass. However, the monotonicity theorem itself provides no quantitative or qualitative control on how much mass is lost.

From an intuitive point of view, the more singular the currents $T_i$ are relative to $S_i$, the larger the difference in~\eqref{eq:mono_intro} should be. The main goal of this paper is to make this intuition precise.

To this end, one must first quantify the difference in singularities. A classical invariant serving this purpose is given by Lelong numbers. We aim to obtain a qualitative control of the difference in~\eqref{eq:mono_intro} in terms of the differences between the Lelong numbers of the currents $T_i$ and $S_i$.



Related problems have been previously investigated by Vu, Su, among others; see for instance \cite{SV25, Su25, Vu23, Suthesis}. Their approaches rely on techniques
involving relative non-pluripolar products and density currents, and lead only to non-optimal lower bounds for the difference in \eqref{eq:mono_intro}. In contrast, in this paper we approach the problem using the theory of nef b-divisors, developed in \cite{Xiabdiv} and \cite{DF20}. As we shall show, this method allows us to obtain an optimal control of the difference in \eqref{eq:mono_intro}. As a consequence, we get a systematic way of obtaining new inequalities involving volumes and movable intersection products.

\subsection{Main results}
We first extend the intersection theory of transcendental b-divisors over $X$ initiated in \cite{Xiabdiv}.
Recall that a nef b-divisor $\mathbb{D}$ over $X$ is a family of modified nef classes $\mathbb{D}_Y\in \mathrm{H}^{1,1}(Y,\mathbb{R})$, with $Y\rightarrow X$ running over all modifications of $X$. These classes are assumed to be compatible with each other under pushforward. 

In \cite{Xiabdiv}, we showed that nef b-divisors can be understood using the theory of currents, and defined the intersection number between $n=\dim X$ nef b-divisors. The first goal of this paper is to develop the full intersection theory. In particular, we establish the following theorem:
\begin{maintheorem}\label{thm:main}
\leavevmode
    \begin{enumerate}
        \item There is a map sending $p$ nef b-divisors $\mathbb{D}_1,\ldots,\mathbb{D}_p$ over $X$ to a nef b-class $\mathbb{D}_1\cap \cdots \cap \mathbb{D}_p$ over $X$. The map is multi-$\mathbb{R}_{\geq 0}$-linear\footnote{In the sense that the product is additive in each variable, while homogeneous only with respect to non-negative scalings.} and satisfies the expected properties of an intersection product.
        \item When $\mathbb{D}_i=\mathbb{D}(\alpha_i)$ is the nef b-divisor induced by a pseudoeffective class $\alpha_i\in \mathrm{H}^{1,1}(X,\mathbb{R})$ as in \cref{def:bdiv_psefclass} for each $i=1,\ldots,p$, the product
    \[
    (\mathbb{D}_1\cap \cdots \cap \mathbb{D}_p)_X
    \]
    is just the movable intersection $\langle \alpha_1\wedge\dots\wedge\alpha_p\rangle$ in the sense of \cite{Bou02, BDPP13}. 
    \item In general, the intersection product can be recovered from the movable intersection as a limit:
    \[
    (\mathbb{D}_1\cap \cdots \cap \mathbb{D}_p)_X=\lim_{\pi\colon Y\rightarrow X} \langle \mathbb{D}_{1,Y}\wedge \dots \wedge \mathbb{D}_{p,Y} \rangle,
    \]
    where $\pi\colon Y\rightarrow X$ runs over the modifications of $X$.
    \item Suppose that $\mathbb{D}_i$ is associated with an $\mathcal{I}$-good current $T_i$ for each $i=1,\ldots,p$, then
    \[
    (\mathbb{D}_1\cap \cdots \cap \mathbb{D}_p)_X=\{T_1\wedge \cdots \wedge T_p\},
    \]
    where the bracket $\{\bullet\}$ denotes the associated cohomology class. The Monge--Ampère type product is understood in the non-pluripolar sense of \cite{BT87, GZ07, BEGZ10}.
    \end{enumerate}
\end{maintheorem}
The notion of nef b-class is a generalization of nef b-divisors. A nef b-class $\mathbb{T}$ is a compatible choice of movable classes $\mathbb{T}_Y\in \mathrm{H}^{p,p}(Y,\mathbb{R})$. We refer to \cref{def:bclassnef} for the details.

Part~(2) shows that the intersection theory of b-divisors is a generalization of the more classical movable intersection theory. The latter has the drawback of being non-linear, making concrete applications difficult. By contrast, the intersection theory of b-divisors is linear. The non-linearity is completely encoded in the map from a pseudoeffective class $\alpha$ to the associated nef b-divisors $\mathbb{D}(\alpha)$, which is easier to understand in general.

Part~(4) shows that the intersection theory of b-divisors can also be regarded as a partial generalization of the non-pluripolar product. Although we can only recover the cohomology class of the non-pluripolar product, this is enough for most applications, since in practice, we usually only need the top intersection. 

The monotonicity theorem \eqref{eq:mono_intro} admits a generalization in the theory of b-divisors as follows: If we have another nef b-divisor $\mathbb{D}_1'$ which dominates $\mathbb{D}_1$ in the sense that $\mathbb{D}'_{1,Y}-\mathbb{D}_{1,Y}$ is pseudoeffective for each modification $\pi\colon Y\rightarrow X$, then
\[
\mathbb{D}'_1\cap \mathbb{D}_2\cap \cdots \cap \mathbb{D}_p\geq \mathbb{D}_1\cap \mathbb{D}_2\cap \cdots \cap \mathbb{D}_p,
\]
in the sense that on each modification of $X$ the difference is a positive cohomology class, as defined in \cref{def:poscoh}.

By the linearity of the intersection product, we could formally write the difference as
\[
\left(\mathbb{D}'_1-\mathbb{D}_1 \right)\cap \mathbb{D}_2\cap \cdots \cap \mathbb{D}_p.
\]
However, this term simply makes no sense, as $\mathbb{D}'_1-\mathbb{D}_1$ is not nef in general. However, it could be formally regarded as a linear combination of prime divisors over $X$ with non-negative coefficients, at least when $\mathbb{D}_{1,X}=\mathbb{D}'_{1,X}$. 

When trying to give a formal meaning to these objects, in \cref{sec:resvol}, we define the restricted volume of nef b-divisors, generalizing the notion of restricted volumes of cohomology classes \cite{Mat13, CT22}. More precisely, given nef b-divisors $\mathbb{D}_1,\ldots,\mathbb{D}_p$ over $X$ and a prime divisor $D$ over $X$, we shall define a class
\[
\vol_{X|D}\left( \mathbb{D}_1,\ldots,\mathbb{D}_p\right)\in \mathrm{H}^{p+1,p+1}(X,\mathbb{R}).
\]
We will show that these classes for various $X$ are compatible with each other and form a b-class, which we denote by 
\[
\vol_{|D}\left( \mathbb{D}_1,\ldots,\mathbb{D}_p\right).
\]
After studying the restricted volumes in detail,
in \cref{sec:qual_mono} we prove the following key inequality:
\begin{maintheorem}\label{thm:bdivint_diff_ineq_intro}
     Let $\mathbb{D}_1,\ldots,\mathbb{D}_p$ be nef b-divisors over $X$, and $T,T'$ be closed positive $(1,1)$-currents on $X$ in the same cohomology class such that $T\preceq_{\mathcal{I}}T'$. Then
    \begin{equation}\label{eq:Ddiffineq1_intro_weak}
        \mathbb{D}_1\cap \dots \cap \mathbb{D}_p \cap \mathbb{D}(T')-\mathbb{D}_1\cap \dots \cap \mathbb{D}_p \cap \mathbb{D}(T)
        \geq \sum_{D/X} \Bigl(\nu(T,D)-\nu(T',D) \Bigr) \vol_{|D}\left(\mathbb{D}_1,\dots, \mathbb{D}_p \right).
    \end{equation}
    The sum is over all prime divisors over $X$.

    Equality holds in \eqref{eq:Ddiffineq1_intro_weak} when the $\mathbb{D}_i$'s are Cartier.
\end{maintheorem}
Here $\mathbb{D}(T)$ is a nef b-divisor canonically constructed from $T$, as we studied in \cite{XiaPPT, Xiabdiv}. The notation $T\preceq_{\mathcal{I}} T'$ means that all Lelong numbers of $T$ (on all modifications) dominate those of $T'$. Equivalently, in the setup of the theorem, it says that $\mathbb{D}(T)\leq \mathbb{D}(T')$.

The main inequality \eqref{eq:Ddiffineq1_intro_weak} gives a quantitative version of the monotonicity theorem. 
When $\mathbb{D}_1,\ldots,\mathbb{D}_p$ are nef b-divisors associated with currents, \cref{thm:bdivint_diff_ineq_intro} gives a solution to our initial question in view of \cref{thm:main}(4). 
We can for example formulate the following consequence:
\begin{corollary}\label{cor:main_eq_analy_intro}
    Let $T_1,\ldots,T_{n-1}$ be closed positive $(1,1)$-currents with analytic singularities on $X$, and $T,T'$ be closed positive $(1,1)$-currents in the same cohomology class such that $T\preceq_{\mathcal{I}} T'$, then
    \begin{equation}
    \begin{aligned}
        &\vol\left(T_1,\ldots,T_{n-1},T'\right)-\vol\left(T_1,\ldots,T_{n-1},T\right)\\
        =&\lim_{\pi\colon Y\rightarrow X}\sum_{D\subseteq Y} \Bigl(\nu(T,D)-\nu(T',D) \Bigr)\int_{\tilde{D}} \bigwedge_{i=1}^{n-1}\left.\Bigl(\pi^*T_i-\nu(T_i,D)[D]\Bigr)\right|_{\tilde{D}}.
    \end{aligned}
    \end{equation}
    Furthermore, the net on the right-hand side is eventually constant.
\end{corollary}
Here $\tilde{D}$ is the normalization of $D$.

\cref{thm:bdivint_diff_ineq_intro} contains many new inequalities of the movable intersection product as well. 
Instead of trying to give an exhaustive list of consequences of \eqref{eq:Ddiffineq1_intro_weak}, we will only explain one such result (\cref{cor:mov_modnef_minus_ineqgen}). Here in the introduction, we only mention a very elegant and special case of the latter as in  \cref{cor:mov_modnef_minus_ineq}.
\begin{theorem}
    Let $[E]$ be a divisorial closed positive $(1,1)$-current on $X$, say 
    \[
    [E]=\sum_i c_i E_i,
    \]
    where the $E_i$'s are distinct prime divisors on $X$ and $c_i>0$. Consider modified nef classes $\alpha_1,\ldots,\alpha_p,\beta\in \mathrm{H}^{1,1}(X,\mathbb{R})$. Assume that $\beta\geq \{E\}$, then
    \[
    \langle \alpha_1\wedge \cdots \wedge \alpha_p \wedge \beta \rangle-\Bigl\langle \alpha_1\wedge \cdots \wedge \alpha_p \wedge (\beta-\{E\}) \Bigr\rangle
        \geq_X \sum_{i} \Bigl( \nu(\beta-\{E\},E_i)+c_i \Bigr)\vol_{X|E_i}\left(\alpha_1,\ldots,\alpha_p\right).
    \]
\end{theorem}
The restricted volume of cohomology classes is defined in \cref{def:vol_mix_class}.
A weaker estimate with $c_i$ in place of $\nu(\beta-\{E\},E_i)+c_i$ can also be obtained using the theory of relative non-pluripolar products, at least when $[E]$ has finitely many prime components.

As another application of our main theorem, in \cref{sec:lom} we prove the following version of the loss of mass theorem stated in the more traditional language:
\begin{maintheorem}\label{thm:maint3}
    Let $\alpha\in \mathrm{H}^{1,1}(X,\mathbb{R})$ be a big class. Consider closed positive $(1,1)$-currents $S,T\in \alpha$ with $T\preceq_{\mathcal{I}}S$. Fix a prime divisor $D$ over $X$, then
    \begin{equation}\label{eq:volminusvol_mult_bd_intro}
    \vol S-\vol T\geq \Bigl( \nu(T,D)-\nu(S,D) \Bigr)^n\cdot  \frac{\vol S }{\Bigl(\nu_{\max}(\alpha,D)-\nu(\alpha,D)\Bigr)^n}.
    \end{equation}
\end{maintheorem} 
The notation $\nu_{\max}(\alpha,D)$ (resp. $\nu(\alpha,D)$) denotes the maximal (resp. minimal) Lelong number of closed positive currents in $\alpha$.
There is also a more general version of \eqref{eq:volminusvol_mult_bd_intro} for mixed volumes, see \cref{cor:mixed_loss_mass}.

This kind of problems have been extensively studied by Vu and Su, see \cite{Vu23, Su25, Suthesis}. As we shall see in the toric example \cref{ex:count_Su}, the second factor in this kind of estimate cannot be a universal constant, contrary to the assertions in the literature. 
Our result improves all known results in this direction. More importantly, our result is the first estimate with explicit constants.

Interestingly, our approach to \cref{thm:maint3} relies on the recent development in the theory of transcendental Okounkov bodies in \cite{DRWNXZ}. Based on the volume formula of transcendental Okounkov bodies, we established the following estimate of the restricted volume, which has independent interests as well.
\begin{theorem}
Let $\alpha\in \mathrm{H}^{1,1}(X,\mathbb{R})$ be a big class. Then for any prime divisor $D$ on $X$, we have
\[
\vol_{X|D}\Bigl(\alpha-t\{D\}\Bigr)\geq \frac{\vol \alpha}{\Bigl(\nu_{\max}(\alpha,D)-\nu(\alpha,D) \Bigr)^n}\cdot \min\Bigl\{t-\nu(\alpha,D),\nu_{\max}(\alpha,D)-t\Bigr\}^{n-1}
\]
as long as $\nu(\alpha,D)\leq t\leq \nu_{\max}(\alpha,D)$.
\end{theorem}
From our proof, it is easy to see that the estimate is sharp, at least in the toric setting.

Finally it is of interest if we can improve \eqref{eq:Ddiffineq1_intro_weak} in \cref{thm:bdivint_diff_ineq_intro} to an equality. For this purpose, the conjectural transcendental Morse inequality seems unavoidable. For the next result, we content ourselves to projective manifolds only, since the transcendental Morse inequality has been confirmed by Witt Nyström \cite{WN19dual} in this case.
\begin{maintheorem}\label{thm:main_int}
Assume that $X$ is projective.

Let $\mathbb{D}$ be a big and nef b-divisors over $X$, $T,T'$ be closed positive $(1,1)$-currents on $X$ in the same cohomology class such that $T\preceq_{\mathcal{I}}T'$. Then
    \begin{equation}\label{eq:Ddiffineq1_intro}
        \Bigl(\mathbb{D}^{n-1} \cap \mathbb{D}(T')\Bigr)-\Bigl(\mathbb{D}^{n-1} \cap \mathbb{D}(T)\Bigr) =  \lim_{\pi\colon Y\rightarrow X}\sum_{D\subseteq Y}\Bigl(\nu(T,D)-\nu(T',D) \Bigr) \vol_{Y|D}\left(\mathbb{D}_{Y} \right).
    \end{equation}     
If we further assume the transcendental Morse inequality, then the same holds for all compact Kähler manifolds $X$.
\end{maintheorem}
One could also polarize \eqref{eq:Ddiffineq1_intro} to get an equality with $\mathbb{D}_1,\ldots,\mathbb{D}_{n-1}$ in place of a single $\mathbb{D}$.

In dimension $1$, since the transcendental Morse inequality is trivially true, \eqref{eq:Ddiffineq1_intro} reduces to the following:
\begin{corollary}\label{cor:RS_intro}
    Assume that $X$ is a compact Riemann surface. Let $T,T'$ be closed positive $(1,1)$-currents on $X$ in the same cohomology class such that $T\preceq_{\mathcal{I}}T'$. Then
    \begin{equation}\label{eq:RS_intro}
        \vol T'-\vol T=\sum_{x\in X} \Bigl( \nu(T,x)-\nu(T',x) \Bigr).
    \end{equation}
\end{corollary}
This is already non-trivial. It clearly shows that the loss of mass on a Riemann surface can be decomposed as the sum of the Lelong number differences.
Based on \eqref{eq:RS_intro}, it seems natural to define a new Monge--Ampère measure of $T$ as $T-\sum_{x\in X}\nu(T,x)[x]$. The total mass of this measure is exactly $\vol T$, and this measure dominates the non-pluripolar measure of $T$. 
It is of interest to know if one can define similar measures in higher dimensions. This problem is closely related to the convergence of partial Bergman kernels, see \cite{DX21, BD26}.

Similarly, in dimension $2$, the transcendental Morse inequality has been confirmed in Y.~Deng's thesis \cite{Deng17}, and hence \eqref{eq:Ddiffineq1_intro} is also an equality. Since we have a relatively complete understanding of divisors over surfaces, it should be possible to extract useful algebrico-geometric information about Kähler surfaces from our equality. The author wishes to explore this point in the near future.

C.~Favre suggested that the right-hand side of \eqref{eq:Ddiffineq1_intro} could be interpreted as an integral over the Berkovich analytification of $X$ with respect to a limit measure lying on a suitable skeleton. This novel measure might be useful for the study of dynamical systems on $X$.

It is particularly illuminating to look at the toric situation. In this case, in \cref{sec:toric_mono}, we will show that this limit measure is precisely the classical mixed area measure between convex bodies, as defined by Aleksandrov and Fenchel--Jenssen in the 1930s. Furthermore, our formula \eqref{eq:Ddiffineq1_intro} is then precisely the Minkowski volume formula expressing the difference of mixed volumes of convex bodies in terms of the integral of the difference between the support functions against the mixed area measure:
\begin{corollary}\label{cor:Minkow_volume_formula_intro}
    Consider convex bodies $P_1,\ldots,P_{n-1},Q',Q\subseteq \mathbb{R}^n$. Assume that $Q'\supseteq Q$, then
    \[
     \vol\left( P_1,\ldots,P_{n-1},Q'\right)-\vol\left(  P_1,\ldots,P_{n-1}, Q\right)
     =\frac{1}{n}\int_{\mathbb{S}^{n-1}}\left( \Supp_{Q'}-\Supp_{Q} \right)\,\mathrm{d}\mathcal{S}\left(P_1,\ldots,P_{n-1}\right).
    \]
\end{corollary}
Here $\Supp$ denotes the support function and $\mathcal{S}$ denotes the mixed area measure. See \cref{sec:toric_mono} for the details.
In fact, our \cref{thm:main_int} gives a new construction of the mixed area measure.

Similar to the mixed area measure, the general limit measure should be an object of great interest.

\paragraph{\textbf{Acknowledgments}}
The author would like to thank Chen Jiang, Yangyang Li, Bingyu Zhang, Bo Berndtsson and Shuang Su for helpful discussions. 
L’auteur souhaite également remercier Charles Favre pour les discussions mathématiques à l’USTC, malgré le fait que nous ayons des points de vue radicalement différents sur de nombreux autres aspects.

This draft was written during the author’s visit to Sichuan University. The author would like to thank Huadi Qu for her kind invitation. 

The author is supported by the National Key R\&D Program of China 2025YFA1018200.

\section{Preliminaries}

\subsection{Quasi-plurisubharmonic functions}\label{subsec:qpsh}
Let $X$ be a connected compact Kähler manifold of dimension $n$.

We briefly recall the notions of $P$ and $\mathcal{I}$-equivalences. For the details, see \cite[Chapter~3, Chapter~6]{Xiabook}.
\begin{definition}
    Let $\varphi,\psi$ be quasi-plurisubharmonic functions on $X$. We say $\varphi\sim_P \psi$ (resp. $\varphi\preceq_P \psi$) if there is a closed smooth real $(1,1)$-form $\theta$ on $X$ such that $\varphi,\psi\in \PSH(X,\theta)_{>0}$ and 
    \[
    P_{\theta}[\varphi]=P_{\theta}[\psi]\quad (\textup{resp.} P_{\theta}[\varphi]\leq P_{\theta}[\psi]).
    \]
\end{definition}
Here $\PSH(X,\theta)$ denotes the space of $\theta$-plurisubharmonic functions on $X$ and $\PSH(X,\theta)_{>0}$ denotes the subset consisting of $\varphi\in \PSH(X,\theta)$ with $\int_X \theta_{\varphi}^n>0$, with $\theta_{\varphi}=\theta+\ddc\varphi$. 
Here and in the sequel, the Monge--Ampère type product $\theta_{\varphi}^n$ is always understood in the non-pluripolar sense of \cite{BT87, GZ07, BEGZ10}. The envelope $P_{\theta}$ is defined as follows:
\[
P_{\theta}[\varphi]\coloneqq \sups_{\!\!\! C\in \mathbb{R}} (\varphi+C)\land 0,
\]
where $(\varphi+C)\land 0$ is the maximal element in $\PSH(X,\theta)$ dominated by both $\varphi+C$ and $0$.

Given a closed smooth real $(1,1)$-form $\theta$ on $X$ so that $\varphi,\psi\in \PSH(X,\theta)$, we also say $\theta_{\varphi}\sim_P \theta_{\psi}$ (resp. $\theta_{\varphi}\preceq_P \theta_{\psi}$) if $\varphi\sim_P \psi$ (resp. $\varphi\preceq_P \psi$). 
We write $\varphi\preceq \psi$ is $\varphi\leq \psi+C$ for some constant $C$. Similarly, we write $\theta_{\varphi}\preceq \theta_{\psi}$.
The same convention applies also to the $\mathcal{I}$-partial order introduced later.

\begin{definition}\label{def:Iequiv}
    Let $\varphi,\psi$ be quasi-plurisubharmonic functions on $X$. We say $\varphi\sim_{\mathcal{I}} \psi$ (resp. $\varphi\preceq_{\mathcal{I}}\psi$) if $\mathcal{I}(\lambda\varphi)=\mathcal{I}(\lambda\psi)$ (resp. $\mathcal{I}(\lambda\varphi)\subseteq\mathcal{I}(\lambda\psi)$) for all real $\lambda>0$. 
\end{definition}
Here $\mathcal{I}$ denotes the multiplier ideal sheaf in the sense of Nadel. 

If $\theta$ is a closed smooth real $(1,1)$-form such that $\varphi,\psi\in \PSH(X,\theta)$, then $\varphi\preceq_{\mathcal{I}} \psi$ if and only if
\[
P_{\theta}[\varphi]_{\mathcal{I}}\leq P_{\theta}[\psi]_{\mathcal{I}},
\]
where
\[
P_{\theta}[\varphi]_{\mathcal{I}}=\sup\left\{\eta\in \PSH(X,\theta):\eta\leq 0, \mathcal{I}(\lambda\varphi)\supseteq \mathcal{I}(\lambda\eta)\textup{ for all }\lambda>0 \right\}.
\]
Equivalently, we may replace $\supseteq$ by $=$ in this equation.

Another equivalent formulation of \cref{def:Iequiv} is that for any prime divisor $E$ over $X$, we have
\[
\nu(\varphi,E)=\nu(\psi,E)\quad \textup{resp. } \nu(\varphi,E)\geq\nu(\psi,E).
\]
Here $\nu$ denotes the generic Lelong number. 

We briefly recall the notion of prime divisors over a complex space $Z$. A prime divisor over $Z$ is a prime divisor $E$ on $Y$, where $\pi\colon Y\rightarrow Z$ is a proper bimeromorphic map from a complex manifold. Consider another prime divisor $E'$ on $Y'$, where $\pi'\colon Y'\rightarrow Z$ is another such map. These divisors are considered equivalent if there is a proper bimeromorphic map $\pi''\colon Y''\rightarrow Z$ from a complex manifold dominating both $\pi$ and $\pi'$ so that the strict transforms of $E$ and $E'$ on $Y''$ are the same. By abuse of language, when we talk about a prime divisor over $Z$, we sometimes refer to such an equivalence class.

Given any $\varphi\in \PSH(X,\theta)$, we have 
\[
\varphi-\sup_X \varphi\leq P_{\theta}[\varphi]\leq P_{\theta}[\varphi]_{\mathcal{I}}.
\]

The operation $P_{\theta}[\bullet]_{\mathcal{I}}$ is idempotent. We say $\varphi\in \PSH(X,\theta)$ is \emph{$\mathcal{I}$-model} if $P_{\theta}[\varphi]_{\mathcal{I}}=\varphi$. Similarly, on the subset $\PSH(X,\theta)_{>0}$, the operation $P_{\theta}[\bullet]$ is also idempotent. We say $\varphi\in \PSH(X,\theta)_{>0}$ is \emph{model} if $P_{\theta}[\varphi]=\varphi$.

A quasi-plurisubharmonic function $\varphi$ on $X$ is called \emph{$\mathcal{I}$-good} if there is a closed smooth real $(1,1)$-form $\theta$ on $X$ such that $\varphi\in \PSH(X,\theta)_{>0}$ and 
\[
P_{\theta}[\varphi]=P_{\theta}[\varphi]_{\mathcal{I}}.
\]
This notion is independent of the choice of $\theta$.
For any closed smooth real $(1,1)$-form $\theta'$ on $X$ so that $\theta'+\ddc\varphi\geq 0$, we also say the current $\theta'_{\varphi}$ is $\mathcal{I}$-good. 

Given a closed positive $(1,1)$-current $T$ on $X$, we write $T=\theta+\ddc\varphi$, then we define
\[
\vol T\coloneqq \int_X \left(\theta+\ddc P_{\theta}[\varphi]_{\mathcal{I}}\right)^n.
\]
This definition is independent of the choice of the decomposition of $T$ as $\theta+\ddc \varphi$. In general,
\[
\vol T\geq \int_X T^n,
\]
and if $\int_X T^n>0$, equality holds if and only if $T$ is $\mathcal{I}$-good. More generally, if $T_1,\ldots,T_n$ are closed positive $(1,1)$-currents on $X$ with positive non-pluripolar masses, we define
\[
\vol(T_1,\ldots,T_n)\coloneqq \int_X \left(\theta_1+\ddc P_{\theta_1}[\varphi_1]\right)\wedge \cdots \wedge \left(\theta_n+\ddc P_{\theta_n}[\varphi_n]\right),
\]
where we write $T_i=\theta_i+\ddc \varphi_i$ for each $i$. In general, when the masses possibly vanish, we define
\[
\vol(T_1,\ldots,T_n)\coloneqq \lim_{\epsilon\to 0+}\vol(T_1+\epsilon\omega,\ldots,T_n+\epsilon\omega)
\]
for any Kähler form $\omega$ on $X$. Note that
\[
\vol(T)=\vol(T,\ldots,T)
\]
for any closed positive $(1,1)$-current $T$ on $X$.

We say a closed positive $(1,1)$-current $T$ has analytic singularities if locally $T$ can be written as $\ddc f$, where $f$ is a plurisubharmonic function of the following form:  
\[
c\log (|f_1|^2+\cdots +|f_N|^2)+R, 
\]
where $c\in \mathbb{Q}_{\geq 0}$, $f_1,\ldots,f_N$ are holomorphic functions on $X$ and $R$ is a bounded function. When we write $T=\theta+\ddc\varphi$ for some smooth closed real $(1,1)$-form $\theta$ and $\varphi\in \PSH(X,\theta)$, we also say $\varphi$ has analytic singularities. 

As a particular case, if $D$ is an effective $\mathbb{Q}$-divisor on $X$, we say a closed positive $(1,1)$-current $T$ has \emph{log singularities} along $D$ if $T-[D]$ is positive, and has locally bounded potentials. It is easy to see that $T$ has analytic singularities. Conversely, if we begin with $T$ with analytic singularities, there is always a modification $\pi\colon Y\rightarrow X$ so that $\pi^*T$ has log singularities along an effective $\mathbb{Q}$-divisor on $Y$. See \cref{def:modif} for our definition of modifications.

Let $\theta$ be a smooth closed real $(1,1)$-form on $X$ and $\eta\in \PSH(X,\theta)$. We say a sequence $(\eta^j)_j$ of quasi-plurisubharmonic functions is a \emph{quasi-equisingular approximation} of $\eta$ if the following are satisfied:
\begin{enumerate}
    \item for each $j$, $\eta^j$ has analytic singularities;
    \item $(\eta^j)_j$ is decreasing with limit $\eta$;
    \item for each $\lambda'>\lambda>0$, we can find $j_0>0$ so that for $j\geq j_0$,
    \[
    \mathcal{I}(\lambda'\eta^j)\subseteq \mathcal{I}(\lambda \eta);
    \]
    \item There is a decreasing sequence $(\epsilon_j)_j$ in $\mathbb{R}_{\geq 0}$ with limit $0$, and a Kähler form $\omega$ on $X$ so that 
    \[
        \eta^j\in \PSH\left(X,\theta+\epsilon_j\omega\right)
    \]
    for each $j>0$.
\end{enumerate}
The existence of quasi-equisingular approximations is guaranteed by \cite{DPS01}. We also say $(\theta+\ddc\eta^j)_j$ is a quasi-equisingular approximation of $\theta+\ddc \eta$. When $\theta_{\eta}$ is a Kähler current, we can (and we always do) take a quasi-equisingular approximation in the same cohomology class of $\{\theta\}$.

Suppose that $\{\theta\}$ is big.
It is shown in \cite{DDNLmetric} that there is a pseudometric $d_S$ on $\PSH(X,\theta)$ satisfying the following inequality:
For any $\varphi,\psi\in \PSH(X,\theta)$, we have
    \begin{equation}\label{eq:ds_biineq}
        \begin{split}
            d_S(\varphi,\psi)\leq & \frac{1}{n+1}\sum_{j=0}^n \left( 2\int_X \theta_{\varphi\lor \psi}^j\wedge \theta_{V_{\theta}}^{n-j}-\int_X \theta_{\varphi}^j\wedge \theta_{V_{\theta}}^{n-j}-\int_X \theta_{\psi}^j\wedge \theta_{V_{\theta}}^{n-j} \right) \\
            \leq & C_n d_S(\varphi,\psi),
        \end{split}
\end{equation}
where $C_n=3(n+1)2^{n+2}$. Here $V_{\theta}=\max\{\varphi\in \PSH(X,\theta):\varphi\leq 0\}$.
Moreover, $d_S(\varphi,\psi)=0$ if and only if $\varphi\sim_P \psi$. In particular, the $d_S$-pseudometric descends to a pseudometric (still denoted by $d_S$) on the space of closed positive $(1,1)$-currents in $\{\theta\}$. 

Given a net of closed positive $(1,1)$-currents $T_i$ in $\{\theta\}$, and another closed positive $(1,1)$-current $T$ in $\{\theta\}$. Then $T_i\xrightarrow{d_S}T$ if and only if $T_i+\omega\xrightarrow{d_S}T+\omega$ for any Kähler form $\omega$ on $X$.

In general, given closed positive $(1,1)$-currents $T_i$ and $T$ on $X$, we say $T_i\xrightarrow{d_S} T$ if we can find Kähler forms $\omega_i$ and $\omega$ on $X$ such that the $T_i+\omega_i$'s and $T+\omega$ represent the same cohomology class and $T_i+\omega_i\xrightarrow{d_S} T+\omega$. This definition is independent of the choices of the $\omega_i$'s and $\omega$.
Quasi-equisingular approximations provide the primary source of $d_S$-convergent sequences.

The proceeding theory can be easily extended to compact normal Kähler spaces, as explained in the appendix of \cite{Xiabook}.

\subsection{Modifications and cones}\label{subsec:modif}
Let $X$ be a reduced compact Kähler space of dimension $n$.

In this paper, we use the word \emph{modification} in a very non-standard sense.
\begin{definition}\label{def:modif}
    A \emph{modification} of $X$ is a bimeromorphic morphism $\pi\colon Y\rightarrow X$, which is a finite composition of blow-ups with smooth centers.

        We say a modification $\pi'\colon Z\rightarrow X$ \emph{dominates} another $\pi\colon Y\rightarrow X$ if there is a morphism $g\colon Z\rightarrow Y$ making the following diagram commutative:
    \begin{equation}\label{eq:domi}
    \begin{tikzcd}
Z \arrow[rr,"g"] \arrow[rd, "\pi'"'] &   & Y \arrow[ld, "\pi"] \\
                                 & X. &                    
\end{tikzcd}
    \end{equation}
\end{definition}
The modifications of $X$ together with the domination relation form a directed set.

Fix a reference Kähler form $\omega$ on $X$.
Recall that a class $\alpha\in \mathrm{H}^{1,1}(X,\mathbb{R})$ is \emph{modified nef} or \emph{movable} if for any $\epsilon>0$, we can find a closed $(1,1)$-current $T\in \alpha$ such that
\begin{enumerate}
    \item $T+\epsilon\omega\geq 0$;
    \item $\nu(T+\epsilon\omega,D)=0$ for any prime divisor $D$ on $X$.
\end{enumerate}
This definition is independent of the choice of $\omega$.
Here $\nu(\bullet,D)$ denote the generic Lelong number along $D$.

These classes are called \emph{nef en codimension $1$} in Boucksom's thesis \cite{Bou02}, where they were introduced for the first time. Modified nef classes form a closed convex cone in $\mathrm{H}^{1,1}(X,\mathbb{R})$. Note that a modified nef class is necessarily pseudoeffective. A nef class is obviously modified nef.

Recall the multiplicity of a cohomology class as defined in \cite[Section~2.1.3]{Bou02}.
\begin{definition}
    Let $\alpha\in \mathrm{H}^{1,1}(X,\mathbb{R})$ be a pseudoeffective class and $D$ be a prime divisor on $X$. We define the Lelong number $\nu(\alpha,D)$ as follows:
    \begin{enumerate}
        \item When $\alpha$ is big, define $\nu(\alpha,D)=\nu(T,D)$ for any closed positive $(1,1)$-current $T\in \alpha$ with minimal singularities (namely, a current in $\alpha$ that is less singular than any current in $\alpha$).
        \item In general, define
        \[
        \nu(\alpha,D)\coloneqq \lim_{\epsilon\to 0+}\nu(\alpha+\epsilon \{\omega\},D).
        \]
    \end{enumerate}
\end{definition}
When $\alpha$ is big, (2) is compatible with (1) and the definition is independent of the choice of $\omega$.
By definition, a pseudoeffective class $\alpha$ is modified nef if and only if $\nu(\alpha,D)=0$ for all prime divisors $D$ on $X$.

Let $T$ be a closed positive $(1,1)$-current on $X$. Then we define the \emph{regular part} $\Reg T$ of $T$ as the regular part of $T$ with respect to Siu's decomposition. In other words, we write
\begin{equation}\label{eq:Siudec}
T=\Reg T+\sum_i c_i [E_i],
\end{equation}
where $E_i$ is a countable collection of prime divisors on $X$ and $c_i=\nu(T,E_i)>0$; the regular part $\Reg T$ is a closed positive $(1,1)$-current whose generic Lelong number along each prime divisor on $X$ is $0$.

\begin{definition}\label{def:nd}
    We say a closed positive $(1,1)$-current $T$ on $X$ is \emph{non-divisorial} (resp. \emph{divisorial}) if $T=\Reg T$ (resp. $\Reg T=0$).
\end{definition}
Note that the cohomology class of a non-divisorial current is always modified nef. Conversely, a current with minimal singularities in a \emph{big} and modified nef class is always non-divisorial.

\subsection{Partially ordered linear spaces}\label{subsec:polinear}

Let $V$ be a finite-dimensional real vector space, and $C\subseteq V$ be a convex cone satisfying the following assumptions:
\begin{enumerate}[label=(C\arabic*), ref=(C\arabic*)]
  \item \label{cond:c1} $C$ is closed;
  \item \label{cond:c2} $C$ is pointed, namely $C\cap(-C)=\{0\}$;
  \item \label{cond:c3} $C$ is full-dimensional, namely $V=C-C$.
\end{enumerate}
Strictly speaking, \ref{cond:c3} is not necessary: We can always replace $V$ by $C-C$.

Given $x,y\in V$, we define $x\leq y$ if $y-x\in C$. Thanks to \ref{cond:c2}, $\leq$ defines a partial order. In particular, we can talk about increasing and decreasing nets.
This partial order has the following properties:
    \begin{enumerate}
        \item[(VO1)] Given $x,y,z\in V$, if $x\leq y$, then $x+z\leq y+z$;
        \item[(VO2)] for any $\lambda\geq 0$ and $x,y\in V$, if $x\leq y$, then $\lambda x\leq \lambda y$.
    \end{enumerate}
In the terminology of \cite{AT07}, $\leq$ is a \emph{vector-ordering} of $V$. The results in this section can be proved using the more general theory in \cite{AT07}, but we shall give elementary proofs for the ease of the readers.

Let $V^{\vee}$ denote the dual vector space of $V$.
Let $C^{\vee}\subseteq V^{\vee}$ be the dual cone of $C$, namely,
\[
C^{\vee}=\left\{ \ell\in V^{\vee}:\ell|_C\geq 0 \right\}.
\]
It is well-known that under the assumption of \ref{cond:c1} and \ref{cond:c2}, $C^{\vee}$ has full dimension. See \cite[Exercise~2.31]{BV04}.

\begin{proposition}\label{prop:mono_conv}
    Suppose that $(x_i)_{i\in I}$ is a decreasing net in $C$, then $(x_i)_i$ converges to some element in $x\in C$. Moreover, $x\leq x_i$ for all $i\in I$. 
\end{proposition}
\begin{proof}
    Since $C^{\vee}$ has full dimension, we can choose $\ell_1,\ldots,\ell_n\in C^{\vee}$ forming a basis of $V^{\vee}$. Then the linear isomorphism
    \[
    L\coloneqq (\ell_1,\ldots,\ell_n)\colon V\rightarrow \mathbb{R}^n
    \]
    maps $C$ into a cone $C'$ contained in the first quadrant. 

    In particular, $(L(x_i))_{i\in I}$ is a net in the first quadrant and each component forms a decreasing net. The usual monotone convergence theorem shows that $(L(x_i))_i$ has a limit $y$ in the first quadrant. Then $x_i\to x\coloneqq L^{-1}(y)$. Thanks to \ref{cond:c1}, $x\in C$. 

    Next fix $i\in I$, we have
    \[
    x_j\leq x_i,\quad \forall j\in I,j\geq i.
    \]
    Using \ref{cond:c1} again, we find $x\leq x_i$.
\end{proof}
\begin{corollary}\label{cor:mono_conv}
    Suppose that $(x_i)_{i\in I}$ is an increasing net in $V$ and $y\in V$ is such that $x_i\leq y$ for all $i$, then $(x_i)_i$ converges to some $x\leq y$. Furthermore, $x\geq x_i$ for each $i\in I$.
\end{corollary}
\begin{proof}
    It suffices to apply \cref{prop:mono_conv} to the net $(y-x_i)_{i\in I}$.
\end{proof}

\begin{proposition}\label{prop:incdom}
    Let $(y_i)_{i>0}$ be a sequence in $C$ so that $\sum_{i=1}^{\infty}y_i$ converges. Consider sequences $(x_i^j)_{j>0}$ in $C$ for all $i>0$ so that
    \[
    \begin{aligned}
        x_i^j\leq y_i,&\quad \forall i,j>0;\\
        \lim_{j\to\infty} x_i^j=x_i,&\quad \forall i>0.
    \end{aligned}
    \]
    Then for all $j>0$, $\sum_{i=1}^{\infty} x_i^j$
    converges and
    \begin{equation}\label{eq:dom_conv1}
    \lim_{j\to\infty}\sum_{i=1}^{\infty} x_i^j=\sum_{i=1}^{\infty} x_i.
    \end{equation}
\end{proposition}
\begin{proof}
    For each $j>0$, the convergence of $\sum_{i=1}^{\infty} x_i^j$ follows from \cref{cor:mono_conv}.

    As for \eqref{eq:dom_conv1}, as in the proof of \cref{prop:mono_conv}, assume that $V=\mathbb{R}^n$ and $C\subseteq \mathbb{R}_{\geq 0}^n$. In this case, it suffices to apply the dominated convergence theorem to each component.
\end{proof}

\begin{proposition}\label{prop:Levi}
    Let $y\in C$. Consider increasing sequences $(x_i^j)_{j>0}$ in $C$ for all $i>0$ with limits $x_i$. Assume that
    \[
    \sum_{i=1}^{\infty}x_i^j\leq y,\quad \forall j>0,
    \]
    Then
    \[
    \lim_{j\to\infty}\sum_{i=1}^{\infty}x_i^j=\sum_{i=1}^{\infty}x_i.
    \]
\end{proposition}
\begin{proof}
    As in the proof of \cref{prop:mono_conv}, assume that $V=\mathbb{R}^n$ and $C\subseteq \mathbb{R}_{\geq 0}^n$. In this case, it suffices to apply Levi's monotone convergence theorem to each component.
\end{proof}

\begin{proposition}\label{prop:squeeze}
    Let $(x_i)_{i\in I}$, $(y_i)_{i\in I}$, $(z_i)_{i\in I}$ be nets in $C$ with $x_i\to x$ for some $x\in C$ and $z_i\to x$. Assume that $x_i\geq y_i\geq z_i$ for all $i\in I$, then $y_i\to x$.
\end{proposition}
\begin{proof}
    As in the proof of \cref{prop:mono_conv}, we may assume that $V=\mathbb{R}^n$ and $C=\mathbb{R}^n_{\geq 0}$. Then it suffices to apply the squeeze theorem to the components.
\end{proof}

\subsection{Cones in cohomology}
Let $X$ be a connected compact Kähler manifold of dimension $n$. 

In \cref{sec:posform}, we shall briefly recall Lelong's theory of positive forms. 

\begin{definition}\label{def:poscoh}
    We say a class $\alpha\in \mathrm{H}^{p,p}(X,\mathbb{R})$ is \emph{positive} if for all weakly positive closed $(n-p,n-p)$-form $F$ on $X$, we have
    \[
    \alpha\cap \{F\}\geq 0.
    \]
    We write $\alpha\geq_X 0$ in this case. 
\end{definition}
Here $\{F\}$ refers to the cohomology class in $\mathrm{H}^{n-p,n-p}(X,\mathbb{R})$ represented by $F$. The cap notation $\cap$ refers to the cohomology pairing.

When we have two classes $\alpha,\beta\in \mathrm{H}^{p,p}(X,\mathbb{R})$, we write $\alpha\geq_X \beta$ if $\alpha-\beta\geq_X 0$. Similarly, the notation $\leq_X$ has the obvious meaning.

If $\alpha$ contains a closed strongly positive $(p,p)$-current, it is clearly positive. The author does not know if the converse holds in general. When $p=n-1$, the converse follows from  \cite[Theorem~0.1]{DP04}. See \cite{CT15} for a proof without relying on the dubious Demailly regularization on singular complex spaces.
When $p=1$, the converse also holds, namely a class in $\mathrm{H}^{1,1}(X,\mathbb{R})$ is positive if and only if it is pseudoeffective. In this case, we usually omit the subindex $X$ and write $\alpha\geq \beta$.
The non-trivial implication is a theorem of Lamari, see \cite[Lemme~3.3]{Lam99}.

An example of positive classes is given by non-pluripolar products.
\begin{proposition}\label{prop:npp_strongpos}
    Let $T_1,\ldots,T_p$ be closed positive $(1,1)$-currents on $X$. Then $T_1\wedge \cdots \wedge T_p$ is strongly positive, and hence $\{T_1\wedge \cdots \wedge T_p\}$ is positive.
\end{proposition}
Here and in the sequel $T_1\wedge \cdots \wedge T_p$ denotes the non-pluripolar product as before.

Although this result is definitely known to experts, it seems missing from the literature. We take this opportunity to give a proof. 

\begin{proof}
    It suffices to show that $T_1\wedge \cdots \wedge T_p$ is strongly positive. The problem is then local. We are reduced to the following: Suppose that $U$ is a domain set in $\mathbb{C}^n$ and $\varphi_1,\ldots,\varphi_p$ are plurisubharmonic functions on $U$. Assume that $\ddc \varphi_1\wedge \cdots \wedge \ddc \varphi_p$ is well-defined, then it is strongly positive.

    \textbf{Step~1}. We first assume that each $\varphi_i$ is bounded. Then taking convolution with Friedrichs mollifiers, we may assume that each $\varphi_i$ is smooth. In this case, the strong positivity of $\ddc \varphi_1\wedge \cdots \wedge \ddc \varphi_p$ follows from \cite[Page~169, Proposition~1.11]{Dembook2}.

    \textbf{Step~2}. We handle the general case.

    Let $\eta$ be a compactly supported smooth closed weakly positive $(n-p,n-p)$-form on $U$. We need to show that
    \begin{equation}\label{eq:strong_pos_npp}
    \int_U \left( \ddc \varphi_1\wedge \cdots \wedge \ddc \varphi_p \right) \wedge \eta\geq 0.
    \end{equation}
    For this purpose, we apply \cite[Lemma~1.5]{BEGZ10} and get
    \[
    \int_U \left( \ddc \varphi_1\wedge \cdots \wedge \ddc \varphi_p \right) \wedge \eta=\lim_{C\to\infty} \int_{\{\varphi_i>-C\}} \left( \bigwedge_{i=1}^p \ddc \left(\varphi_i\lor (-C) \right) \right)\wedge \eta.
    \]
    But each term is non-negative by Step~1, and hence \eqref{eq:strong_pos_npp} follows.
\end{proof}

\begin{proposition}\label{prop:Posclass_goodcone}
    The cone of positive classes in $\mathrm{H}^{p,p}(X,\mathbb{R})$ satisfy \ref{cond:c1}, \ref{cond:c2} and \ref{cond:c3}.
\end{proposition}
In particular, the abstract results proved in \cref{subsec:polinear} can be applied. In the sequel of this paper, we shall apply them without further mentioning.
\begin{proof}
    The condition \ref{cond:c1} is obvious. 
    
    Let us prove \ref{cond:c2}. Let $\alpha\in \mathrm{H}^{p,p}(X,\mathbb{R})$ be a class with vanishing intersection with all classes of the form $\{F\}$, where $F$ is a closed weakly positive $(n-p,n-p)$-form. We need to show that $\alpha=0$. 
    
    Fix a general cohomology class $\beta\in \mathrm{H}^{n-p,n-p}(X,\mathbb{R})$, it suffices to show that 
    \begin{equation}\label{eq:alphabetaint0}
    \alpha\cap \beta=0.
    \end{equation}
    For this purpose, take a closed real $(n-p,n-p)$-form $G$ representing $\beta$. Fix a Kähler form $\omega$ on $X$. We observe that
    \[
    G+C\omega^{n-p}
    \]
    is positive when $C$ is large enough. In fact, thanks to the compactness of $X$, this problem is essentially local, we can consider it on a coordinate chart on $X$. Then this is just a simple linear algebra. For later use, let us observe that we can in fact guarantee that $G+C\omega^{n-p}$ is strongly positive, as follows from \cref{thm:HK}.

    Now by our assumption, 
    \[
    \alpha\cap \left\{C\omega^{n-p}\right\}=0,\quad \alpha\cap  \left\{G+C\omega^{n-p}\right\}=0.
    \]
    Therefore, \eqref{eq:alphabetaint0} follows.

    Finally let us prove \ref{cond:c3}. This follows form the observation in proving \ref{cond:c2}: A closed smooth real $(p,p)$-form can always be represented as the difference of two closed smooth strongly positive $(p,p)$-forms. 
\end{proof}

\begin{proposition}\label{prop:pos_push}
    Let $f\colon Y\rightarrow X$ be a proper morphism from a Kähler manifold, if $\alpha\in \mathrm{H}^{p,p}(Y,\mathbb{R})$ is positive, then so is $f_*\alpha$.
\end{proposition}
Note that we do not require that $f$ be bimeromorphic, nor surjective. 
\begin{proof}
    This follows from the fact that the pull-back of a weakly positive form is weakly positive.
\end{proof}

\begin{proposition}\label{prop:nefint_pre_pos}
    Let $\alpha\in \mathrm{H}^{1,1}(X,\mathbb{R})$ be a nef class and $\beta\in \mathrm{H}^{p,p}(X,\mathbb{R})$ be a positive class. Then $\alpha\cap \beta$ is positive.
\end{proposition}
\begin{proof}
    Without loss of generality, we may assume that $\alpha$ is a Kähler class. Then our assertion follows from the fact that the wedge product of a positive $(1,1)$-form with a weakly positive form is weakly positive.
\end{proof}

\begin{lemma}\label{lma:posclass_vanish}
    Let $\alpha\in \mathrm{H}^{p,p}(X,\mathbb{R})$ be a positive class. Assume that for each Kähler class $\beta$, we have $\alpha\cap \beta^{n-p}=0$, then $\alpha=0$.
\end{lemma}
\begin{proof}
    Suppose that $\alpha\neq 0$, then we can find a closed real $(n-p,n-p)$-form $F$ so that $\alpha\cap \{F\}<0$. 
    As in the proof of \cref{prop:Posclass_goodcone}, we can find a Kähler form $\omega$ on $X$ so that $F+\omega^{n-p}$ is positive. Then by our assumption, we know that $\alpha\cap \{F+\omega^{n-p}\}<0$, which is a contradiction.
\end{proof}

A source of positive classes is provided by the following monotonicity theorem proved in increasing order of generality by \cite{WN19, DDNL18mono, Vu20}:
\begin{theorem}\label{thm:mono_thm_partial}
    Let $\alpha_1,\ldots,\alpha_p\in \mathrm{H}^{1,1}(X,\mathbb{R})$ be pseudoeffective classes. Consider closed positive $(1,1)$-currents $T_i,S_i\in \alpha_i$ for $i=1,\ldots,p$. Assume that $T_i\preceq_P S_i$ for all $i$. Then 
    \begin{equation}\label{eq:mono_thm_partial}
        \left\{ T_1\wedge \cdots \wedge T_p \right\} \leq_X \left\{ S_1\wedge \cdots \wedge S_p \right\}.
    \end{equation}
\end{theorem}
The notion of $P$-partial order is recalled in  \cref{subsec:qpsh}. In particular, the cohomology class
\[
\{T_1\wedge \cdots \wedge T_p\}
\]
depends only on the $P$-equivalence classes of the $T_i$'s and the cohomology classes of the $T_i$'s.

\begin{proof}
    Take a Kähler form $\omega$ on $X$ and replacing $T_i$ and $S_i$ by $T_i+\epsilon\omega$ and $S_i+\epsilon\omega$ respectively for some small $\epsilon>0$, we can reduce to the case where the $T_i$'s and the $S_i$'s are all Kähler currents.

    Take a smooth closed real $(1,1)$-form $\theta_i\in \alpha_i$ and write
    \[
    T_i=\theta_i+\ddc \varphi_i,\quad S_i=\theta_i+\ddc \psi_i.
    \]
    Then $\varphi_i\preceq_P \psi_i$.

    \textbf{Step~1}. We first prove that
    \[
        \Bigl\{ \left(\theta_1+\ddc \psi_1\right)\wedge \cdots \wedge \left(\theta_p+\ddc \psi_p\right) \Bigr\}=\Bigl\{ \left(\theta_1+\ddc P_{\theta_1}[\psi_1]\right)\wedge \cdots \wedge \left(\theta_p+\ddc P_{\theta_p}[\psi_p]\right) \Bigr\}.
    \]
    The $\leq_X$ direction is proved in \cite[Theorem~4.4]{Vu20}. Therefore, using \cref{lma:posclass_vanish}, in order to prove the equality, we may assume that $p=n$. In this case, the assertion follows from the usual monotonicity theorem. See \cite[Proposition~6.1.4]{Xiabook}.

    \textbf{Step~2}. By Step~1, we may replace $\psi_i$ by $P_{\theta_i}[\psi_i]$ and assume that $\varphi_i\preceq \psi_i$ for each $i=1,\ldots,p$. 
    By Vu's result \cite[Theorem~4.4]{Vu20} again, we conclude \eqref{eq:mono_thm_partial}.
\end{proof}
Let us also recall another result for later use.
\begin{theorem}\label{thm:dscontnpp1}
    Let $\alpha_1,\ldots,\alpha_p\in \mathrm{H}^{1,1}(X,\mathbb{R})$ be big classes. Consider nets of closed positive $(1,1)$-currents $(T_i^j)_{j\in J}$ in $\alpha_i$ for each $i=1,\ldots,p$. Assume that $T_i^j\xrightarrow{d_S} T_i\in \alpha_i$ for each $i=1,\ldots,p$. Then
    \begin{equation}\label{eq:dscontnpp1}
        \lim_{j\in J} \left\{T_1^j\wedge \cdots \wedge T_p^j \right\}=\left\{T_1 \wedge \cdots \wedge T_p \right\}.
    \end{equation}
\end{theorem}
This follows \emph{verbatim} from the proof of \cite[Theorem~6.2.1]{Xiabook}, with \cref{thm:mono_thm_partial} in place of the usual monotonicity theorem.

For later use, we recall the behavior of the non-pluripolar products under bimeromorphic morphisms:
\begin{proposition}\label{prop:bime_npp}
    Let $\pi\colon Y\rightarrow X$ be a proper bimeromorphic morphism from a Kähler manifold. Let $T_1,\ldots,T_p$ be closed positive $(1,1)$-currents on $X$. Then
    \[
    \begin{aligned}
    \pi^{\lozenge}(T_1\wedge \cdots \wedge T_p)=&\pi^*T_1\wedge \cdots \wedge \pi^*T_p,\\
    \pi_*\left( \pi^*T_1\wedge \cdots \wedge \pi^*T_p \right)=&T_1\wedge \cdots \wedge T_p.
    \end{aligned}
    \]
\end{proposition}
Here $\pi^{\lozenge}$ is defined as follows: Take a smallest Zariski closed subset $Z\subseteq X$ so that $\pi$ is an isomorphism outside $Z$. Then given a closed positive $(p,p)$-current $T$ on $X$, we let $\pi^{\lozenge}T$ be the zero-extension of $\pi|_{Y\setminus \pi^{-1}(Z)}^*(T)$ to $Y$. 

Both assertions follow easily from the fact that the non-pluripolar products put no mass on proper analytic sets.

\begin{definition}\label{def:stric_mov}
A class $\alpha\in \mathrm{H}^{p,p}(X,\mathbb{R})$ is \emph{strictly movable} if there is a proper bimeromorphic morphism $\pi\colon Y\rightarrow X$ from a Kähler manifold $Y$, and Kähler classes $\beta_1,\ldots,\beta_p\in \mathrm{H}^{1,1}(Y,\mathbb{R})$ so that 
\begin{equation}\label{eq:stric_mov}
\alpha=\pi_*\left( \beta_1\cap \cdots\cap \beta_p \right).
\end{equation}
A class in the closed convex cone generated by strictly movable classes is called a \emph{movable class}.
\end{definition}
The cone of movable classes in $\mathrm{H}^{p,p}(X,\mathbb{R})$ is a closed convex cone. A movable class is clearly positive.

\begin{example}\label{ex:movable_special}
    When $p=1$, a class $\alpha\in \mathrm{H}^{1,1}(X,\mathbb{R})$ is movable if and only if it is modified nef. This is \cite[Proposition~2.1.2]{Bou02}.

    When $p=0$ or $p=n$, we have canonical identifications $\mathrm{H}^{p,p}(X,\mathbb{R})\cong \mathbb{R}$. A class corresponding to $t\in \mathbb{R}$ is movable if and only if $t\geq 0$.
\end{example}
\begin{example}\label{ex:mov_pushnef}
    In \eqref{eq:stric_mov}, if $\beta_1,\ldots,\beta_p\in \mathrm{H}^{1,1}(Y,\mathbb{R})$ are nef, $\alpha$ is still movable. This follows immediately from the continuity of $\pi_*$.
\end{example}

\begin{proposition}\label{prop:movable_push}
    Let $\pi\colon Y\rightarrow X$ be a proper bimeromorphic morphism from a Kähler manifold $Y$. Suppose that $\alpha\in \mathrm{H}^{p,p}(Y,\mathbb{R})$ is movable, then so is $\pi_*\alpha$.
\end{proposition}
\begin{proof}
    We may assume that $\alpha$ is strictly movable. Then our assertion follows immediately from the definition.
\end{proof}

\begin{proposition}\label{prop:pos_relation}Let $\alpha\in \mathrm{H}^{1,1}(X,\mathbb{R})$ be a pseudoeffective class, and $\beta\in \mathrm{H}^{p,p}(X,\mathbb{R})$ be a movable class. Then
\begin{enumerate}
    \item $\alpha\cap \beta$ is positive;
    \item if furthermore $\alpha$ is nef, then $\alpha\cap \beta$ is movable.
\end{enumerate}
\end{proposition}
\begin{proof}
    We may assume that $\beta$ is strictly movable. Take a proper bimeromorphic morphism $\pi\colon Y\rightarrow X$ from a Kähler manifold $Y$, and Kähler classes $\gamma_1,\ldots,\gamma_p\in \mathrm{H}^{1,1}(Y,\mathbb{R})$ with
    \[
    \beta=\pi_*\left( \gamma_1\cap \cdots \cap \gamma_p \right).
    \]
    Then
    \[
    \alpha\cap \beta=\pi_*\left( \pi^*\alpha\cap \gamma_1\cap \cdots \cap \gamma_p \right).
    \]
    
    (1) By \cref{prop:pos_push}, it suffices to show that $ \pi^*\alpha\cap \gamma_1\cap \cdots \cap \gamma_p$ is positive, which follows from \cref{prop:nefint_pre_pos}.

    (2) Thanks to \cref{prop:movable_push}, it suffices to show that $\pi^*\alpha\cap \gamma_1\cap \cdots \cap \gamma_p$ is movable. But this is obvious.
\end{proof}
For mnemonic purposes, we summarize \cref{prop:pos_relation} and \cref{prop:nefint_pre_pos} in \cref{table:product_pos}.
\begin{table}[ht]
\centering
\begin{tabular}{|c|c|c|}
\hline
$\alpha \in \mathrm{H}^{1,1}(X,\mathbb{R})$ 
& $\beta \in \mathrm{H}^{p,p}(X,\mathbb{R})$ 
& $\alpha \cap \beta$ \\
\hline
nef & positive & positive \\
\hline
nef & movable & movable \\
\hline
positive & movable & positive \\
\hline
\end{tabular}
\caption{Positivity properties of intersection products}
\label{table:product_pos}
\end{table}

\subsection{Movable intersection theory}\label{subsec:mov_int}
Let $X$ be a connected compact Kähler manifold of dimension $n$. Consider pseudoeffective $(1,1)$-classes $\alpha_1,\ldots,\alpha_p\in \mathrm{H}^{1,1}(X,\mathbb{R})$.

\begin{definition}
    The \emph{movable intersection} $\langle \alpha_1\wedge \cdots \wedge \alpha_p\rangle \in \mathrm{H}^{p,p}(X,\mathbb{R})$ is defined as follows:
    \begin{enumerate}
        \item When the $\alpha_i$'s are all big, we take closed positive $(1,1)$-currents with minimal singularities $T_{i,\min}$ in each $\alpha_i$ and let
        \[
        \langle \alpha_1\wedge \cdots \wedge \alpha_p\rangle =\left\{ T_{1,\min} \wedge \cdots \wedge T_{p,\min} \right\};
        \]
        \item in general, define
        \[
        \langle \alpha_1\wedge \cdots \wedge \alpha_p\rangle =\lim_{\epsilon\to 0+}\Bigl\langle (\alpha_1+\epsilon\beta)\wedge \cdots \wedge (\alpha_p+\epsilon\beta)\Bigr\rangle, 
        \]
        where $\beta\in \mathrm{H}^{1,1}(X,\mathbb{R})$ is a Kähler class.
    \end{enumerate}
\end{definition}
The movable intersection is independent of the choices we made. Moreover, it is always a positive class. The current definition is taken from \cite{BEGZ10}. A more traditional (but equivalent) definition can be found in \cite{Bou02, BDPP13}. 

\begin{proposition}\label{prop:mov_int_prop}
Let $\alpha_1,\ldots,\alpha_p,\alpha_1',\alpha,\beta_1,\ldots,\beta_p\in \mathrm{H}^{1,1}(X,\mathbb{R})$ be pseudoeffective classes, and $\lambda\geq 0$. We have the following properties:
\begin{enumerate}
    \item The movable intersection is symmetric: For each permutation $\sigma$ of $\{1,\ldots,p\}$, we have
    \[
    \left\langle \alpha_{\sigma(1)}\wedge \cdots \wedge \alpha_{\sigma(p)}\right\rangle=\langle \alpha_1\wedge \cdots \wedge \alpha_p\rangle.
    \]
    \item The movable intersection is superadditive in each variable: 
    \[
    \left\langle (\alpha_1+\alpha_1')\wedge \alpha_2\wedge  \cdots \wedge \alpha_p\right\rangle\geq_X \langle \alpha_1\wedge \alpha_2\wedge  \cdots \wedge \alpha_p\rangle+ \langle \alpha'_1\wedge\alpha_2\wedge  \cdots \wedge \alpha_p\rangle.
    \]
    \item The movable intersection is homogeneous in each variable: 
    \[
    \langle \lambda\alpha_1\wedge \alpha_2\wedge \cdots \wedge \alpha_p\rangle=\lambda\langle \alpha_1\wedge \alpha_2 \wedge \cdots \wedge \alpha_p\rangle.
    \]
    \item The movable intersection is increasing in each variable: Suppose that $\alpha_1\geq \alpha'_1$, then
    \[
    \langle \alpha_1\wedge \cdots \wedge \alpha_p\rangle\geq_X \langle \alpha'_1\wedge \alpha_2\wedge \cdots \wedge \alpha_p\rangle.
    \]
    \item We have
    \begin{equation}\label{eq:div_Zar}
    \alpha=\langle \alpha \rangle + \sum_{D\subseteq X} \nu(\alpha,D)\{D\},
    \end{equation}
    where $D$ runs over the set of prime divisors on $X$.
    \item We have
    \[
    \langle \alpha_1\wedge \cdots \wedge \alpha_p\rangle=\lim_{\epsilon\to 0+}\Bigl\langle (\alpha_1+\epsilon\beta_1)\wedge \cdots \wedge (\alpha_p+\epsilon\beta_p)\Bigr\rangle.
    \]
    \item When $\alpha_1$ is nef, we have
    \[
     \langle \alpha_1\wedge \cdots \wedge \alpha_p\rangle=\alpha_1\cap  \langle \alpha_2\wedge \cdots \wedge \alpha_p\rangle
    \]
\end{enumerate}
\end{proposition}
See \cite[Theorem~3.5]{BDPP13} for the proof of most parts. Part (6) is proved in \cite[Proposition~3.2.4]{Bou02}.

The decomposition \eqref{eq:div_Zar} is usually referred to as the divisorial Zariski decomposition or the Boucksom--Nakayama decomposition.

The volume of pseudoeffective class $\alpha\in \mathrm{H}^{1,1}(X,\mathbb{R})$ is defined as $\vol(\alpha)\coloneqq \langle \alpha^n \rangle$.

\section{Generalized intersection theory}
Let $X$ be a connected compact Kähler manifold of dimension $n$. In the whole section, $p,q$ will denote two non-negative integers.

We remind the readers that the word \emph{modification} takes a very non-standard meaning in this paper, as we recalled in \cref{def:modif}.
\subsection{The notion of b-classes}
We introduce the central object of interest in this paper --- The b-classes.
\begin{definition}\label{def:bclass}
    A \emph{(Weil) b-class} $\mathbb{T}$ of degree $p$ over $X$ is an element in 
    \begin{equation}\label{eq:projlimbclass}
    \varprojlim_{\pi\colon Y\rightarrow X} \mathrm{H}^{p,p}(Y,\mathbb{R}),
    \end{equation}
    where $\pi\colon Y\rightarrow X$ runs over the directed set of modifications of $X$. Here the limit is taken in the category of real vector spaces.
\end{definition}
    Equivalently, thanks to Hironaka's Chow lemma \cite[Corollary~2]{Hir75}, we could ask $\pi$ to run over the set of proper bimeromorphic morphisms from Kähler manifolds.
    
    The components of $\mathbb{T}$ with respect to \eqref{eq:projlimbclass} will be denoted by $\mathbb{T}_Y\in \mathrm{H}^{p,p}(Y,\mathbb{R})$. The vector space of b-classes of degree $p$ over $X$ is denoted by $\bDiv^p(X)$.
    
    The vector space $\bDiv^p(X)$ is endowed with the projective limit topology. In other words, the convergence of a net of b-classes $(\mathbb{T}^i)_{i\in I}$ to a b-class $\mathbb{T}$ means the convergence of $(\mathbb{T}^i_Y)_{i\in I}$ to $\mathbb{T}_Y$ with respect to the Euclidean topology for each modification $\pi\colon Y\rightarrow X$.  

\begin{definition}\label{def:bclassnef}
    A b-class $\mathbb{T}\in \bDiv^p(X)$ is \emph{nef} if the component $\mathbb{T}_Y$ is movable in the sense of \cref{def:stric_mov} for each modification $\pi\colon Y\rightarrow X$. 

    A b-class $\mathbb{T}\in \bDiv^p(X)$ is \emph{pseudoeffective} if the component $\mathbb{T}_Y$ is positive in the sense of \cref{def:poscoh} for each modification $\pi\colon Y\rightarrow X$. 
\end{definition}
\begin{remark}
    There is a subtlety here: We can ask $\pi\colon Y\rightarrow X$ to run over either all modifications or all proper bimeromorphic morphisms from Kähler manifolds. The corresponding notions of nef or pseudoeffective b-classes are identical thanks to \cref{prop:pos_push} and \cref{prop:movable_push}.
\end{remark}

Note that nef b-classes and pseudoeffective b-classes are both closed conditions, namely they are preserved under limits of nets, since the movable cone and the positive cone on a compact Kähler manifold are both closed. A nef b-class is pseudoeffective.

From the projective limit definition \eqref{eq:projlimbclass}, there is a canonical identification between $\bDiv^p(X)$ and $\bDiv^p(Y)$ for each proper bimeromorphic morphism $\pi\colon Y\rightarrow X$ between Kähler manifolds. We shall implicitly use this identification all the time. 

All notions related to b-divisors defined in this paper are invariant under changing $X$ to its modifications. We shall explicitly check this whenever the statement is not obvious. 

Furthermore, using this identification, we can make sense of $\bDiv^p(X)$ for an arbitrary reduced compact Kähler space $X$: Simply take a projective resolution $Y\rightarrow X$, and set $\bDiv^p(X)\coloneqq \bDiv^p(Y)$. 

\begin{example}\label{ex:bclass_specialcase}
    When $p=0$, there is a canonical identification $\bDiv^0(X) \cong \mathbb{R}$. An element $t\in \bDiv^0(X)$ is nef or pseudoeffective if and only if $t\geq 0$.
    Similarly, when $p=n$, there is also a canonical identification  $\bDiv^n(X) \cong \mathbb{R}$. An element $t\in \bDiv^n(X)$ is nef or pseudoeffective if and only if $t\geq 0$.

    When $p=1$, $\bDiv^1(X)$ is just the set of b-divisors over $X$, as we studied in the previous paper. An element in  $\bDiv^1(X)$ is nef (resp. pseudoeffective) in the sense of \cref{def:bclassnef} if and only if it is a nef (resp. pseudoeffective) b-divisor in the sense of \cite{Xiabdiv}. See \cite[Corollary~11.1.1]{Xiabook} for the details.
\end{example}

\begin{definition}\label{def:orderbclass}
    Given $\mathbb{T},\mathbb{T}'\in \bDiv^p(X)$, we write $\mathbb{T}\geq \mathbb{T}'$ if for each modification $\pi\colon Y\rightarrow X$, we have $\mathbb{T}_Y\geq_Y \mathbb{T}'_Y$. 
\end{definition}
Recall that $\geq_Y$ is defined right after \cref{def:poscoh}.

Thanks to \cref{prop:pos_push}, it suffices to check this condition for a cofinal set of modifications. Note that $\geq$ defines a partial order on $\bDiv^p(X)$.

\begin{remark}
    In \cref{def:orderbclass} we could equivalent ask $\pi$ to run over the set of proper bimeromorphic morphisms $\pi\colon Y\rightarrow X$ from Kähler manifolds, as a consequence of \cref{prop:pos_push}.
\end{remark}

Let us recall the following construction from \cite{XiaPPT, Xiabdiv}:
\begin{definition}\label{def:Dcurrent}
    Let $T$ be a closed positive $(1,1)$-current on $X$, we define a nef b-divisor $\mathbb{D}(T)$ over $X$ as follows:
    \begin{equation}\label{eq:DTY}
    \mathbb{D}(T)_Y=\left\{ \Reg\pi^*T \right\}\in \mathrm{H}^{1,1}(Y,\mathbb{R}),
    \end{equation}
    for all modification $\pi\colon Y\rightarrow X$.
\end{definition}
Here $\Reg$ is defined in \eqref{eq:Siudec}.

\begin{remark}
    From the obvious functoriality of Siu's decomposition, for a proper bimeromorphic morphism $\pi\colon Y\rightarrow X$ from a Kähler manifold $Y$, the component $\mathbb{D}(T)_Y$ of $\mathbb{D}(T)$ is given by exactly the same formula \eqref{eq:DTY}.
\end{remark}
Note that $\mathbb{D}(T)$ depends linearly on $T$ in the following sense: If $T'$ is another closed positive $(1,1)$-current on $X$, and $\lambda \geq 0$, then
\begin{equation}\label{eq:D_curr_linear}
\mathbb{D}(T+T')=\mathbb{D}(T)+\mathbb{D}(T'),\quad \mathbb{D}(\lambda T)=\lambda \mathbb{D}(T).
\end{equation}

Recall that a nef b-divisor $\mathbb{D}$ is big if its volume
\[
\vol \mathbb{D}\coloneqq \lim_{\pi\colon Y\rightarrow X} \vol \mathbb{D}_Y
\]
is positive.

The main result in \cite{Xiabdiv} says
\begin{theorem}\label{thm:main_paper1}
Fix a modified nef and big class $\alpha \in \mathrm{H}^{1,1}(X,\mathbb{R})$.
    The map $\mathbb{D}$ in \cref{def:Dcurrent} induces a canonical bijection between
    \begin{enumerate}
        \item the set of non-divisorial closed positive $(1,1)$-currents in $\alpha$ with positive volumes modulo $\mathcal{I}$-equivalence, and
        \item the set of nef and big b-divisors $\mathbb{D}$ over $X$ with $\mathbb{D}_X=\alpha$.
    \end{enumerate}
    The correspondence preserves the volumes.
    
    Moreover, given any nef and big b-divisor $\mathbb{D}$ over $X$, we can always find a non-divisorial \emph{$\mathcal{I}$-good} closed positive $(1,1)$-current $T$ with $\mathbb{D}=\mathbb{D}(T)$.
\end{theorem}
The relevant notions are recalled in \cref{subsec:qpsh}. In general, we cannot take $T$ as a Kähler current, however, we have the following useful observation: If $\mathbb{D}$ is a nef b-divisor, $\omega$ is a Kähler form on $X$, then $\mathbb{D}+\mathbb{D}(\omega)$ can be represented as $\mathbb{D}(T)$ where $T$ is an $\mathcal{I}$-good non-divisorial Kähler current. It suffices to write $\mathbb{D}+\mathbb{D}(\omega)$ as $(\mathbb{D}+2^{-1}\mathbb{D}(\omega))+2^{-1}\mathbb{D}(\omega)$ and apply \cref{thm:main_paper1} to the term in the parentheses.

Recall that a nef b-divisor $\mathbb{D}$ over $X$ is Cartier if it admits a \emph{realization}: A pair $(\pi\colon Y\rightarrow X,\alpha)$ consisting of a proper bimeromorphic morphism $\pi\colon Y\rightarrow X$ from a Kähler manifold $Y$ and a nef class $\alpha\in \mathrm{H}^{1,1}(Y,\mathbb{R})$, so that for each modification $\pi'\colon Y'\rightarrow X$ dominating $\pi$, the value $\mathbb{D}_{Y'}$ is the pull-back of $\alpha$ through the morphism $Y'\rightarrow Y$. We also say $\mathbb{D}$ is realized on $Y$ by $\alpha$.

Note that we can always take $\pi$ to be a modification.

As a corollary, we have:
\begin{corollary}\label{cor:Demapp}
    For any nef b-divisor $\mathbb{D}$ over $X$, we can find a decreasing sequence $(\mathbb{D}^i)_i$ of nef Cartier b-divisors over $X$ with limit $\mathbb{D}$.
\end{corollary}
See \cite[Corollary~4.15]{Xiabdiv}.

\begin{definition}\label{def:bdiv_psefclass}
    Given a pseudoeffective class $\alpha\in \mathrm{H}^{1,1}(X,\mathbb{R})$, we define a nef b-divisor $\mathbb{D}(\alpha)$ over $X$ as follows:
    \[
    \mathbb{D}(\alpha)_Y=\langle \pi^*\alpha \rangle,
    \]
    where $\pi\colon Y\rightarrow X$ is a modification.
\end{definition}
Recall that $\langle \bullet \rangle$ refers to the movable intersection product as we recalled in \cref{subsec:mov_int}. Namely, $\mathbb{D}(\alpha)_Y$ is the movable part of $\pi^*\alpha$ with respect to the Boucksom--Nakayama decomposition.

When $\alpha$ is big, $\mathbb{D}(\alpha)=\mathbb{D}(T_{\min})$, where $T_{\min}\in \alpha$ is a current with minimal singularities. 
When $\alpha$ is nef, we have 
\begin{equation}\label{eq:Dalpha_nef}
\mathbb{D}(\alpha)_{Y}=\pi^*\alpha
\end{equation}
for each modification $\pi\colon Y\rightarrow X$.

The operator is concave in the following sense:
\begin{proposition}\label{prop:Dsupadd}
    Let $\alpha,\beta\in \mathrm{H}^{1,1}(X,\mathbb{R})$ be pseudoeffective classes. Then
    \begin{equation}\label{eq:Dsupadd}
    \mathbb{D}(\alpha)+\mathbb{D}(\beta)\leq \mathbb{D}(\alpha+\beta).
    \end{equation}
    Equality holds when $\alpha,\beta$ are both nef.
\end{proposition}
\begin{proof}
    In order to prove \eqref{eq:Dsupadd}, it suffices to prove 
    \[
    \mathbb{D}(\alpha)_X+\mathbb{D}(\beta)_X\leq \mathbb{D}(\alpha+\beta)_X.
    \]
    In other words,
    \[
    \sum_{D\subseteq X} \Bigl( \nu(\alpha,D)+\nu(\beta,D) \Bigr) \{D\}\geq \sum_{D\subseteq X} \nu(\alpha+\beta,D)\{D\},
    \]
    which follows from the simple observation: For any prime divisor $D$ on $X$, we have
    \[
    \nu(\alpha,D)+\nu(\beta,D) \geq \nu(\alpha+\beta,D).
    \]

    When $\alpha$ and $\beta$ are nef, the equality holds in \eqref{eq:Dsupadd} by our explicit description of the associated b-divisors in \eqref{eq:Dalpha_nef}.
\end{proof}

\subsection{The intersection theory}
We begin with the intersection between a Cartier nef b-divisor and a nef b-class. 
\begin{definition}\label{def:b_div_int_general}
Let $\mathbb{D}$ be a nef Cartier b-divisor over $X$ and $\mathbb{T}\in \bDiv^q(X)$ be a nef b-class. Then we define the nef class $\mathbb{D} \cap \mathbb{T}\in \bDiv^{q+1}(X)$ as follows: Let $(\pi\colon Y\rightarrow X, \alpha)$ be a realization of $\mathbb{D}$, then we define for each modification $\pi'\colon Y'\rightarrow X$ dominating $\pi$ through $g\colon Y'\rightarrow Y$,
\begin{equation}\label{eq:DcapTY'}
    \left( \mathbb{D} \cap \mathbb{T} \right)_{Y'}\coloneqq g^*\alpha \cap \mathbb{T}_{Y'}\in \mathrm{H}^{p+1,p+1}(Y,\mathbb{R}).
\end{equation}
Here as before, $\cap$ on the right-hand side denotes the intersection product of cohomology.
The class defined in \eqref{eq:DcapTY'} is movable by \cref{prop:pos_relation}.
The relevant notations are summarized in the following commutative diagram:
\[
    \begin{tikzcd}
        Y' \arrow[rd, "\pi'"'] \arrow[rr, "g"] &   & Y \arrow[ld, "\pi"] \\
                                       & X. &                    
        \end{tikzcd}
\]

\end{definition}
By the projection formula, these classes are compatible under pushforwards. To  be more precise, if $\pi''\colon Y''\rightarrow X$ is a modification dominating $\pi'$ through $h\colon Y''\rightarrow Y'$, as in the following commutative diagram:
\begin{equation}\label{eq:three_modif}
\begin{tikzcd}
Y'' \arrow[rd, "\pi''"'] \arrow[r, "h"] & Y' \arrow[d, "\pi'"] \arrow[r, "g"] & Y \arrow[ld, "\pi"] \\
                                        & X,                                   &                    
\end{tikzcd}
\end{equation}
then
\[
h_*\Bigl( (g\circ h)^*\alpha\cap \mathbb{T}_{Y''} \Bigr)=g^*\alpha\cap \mathbb{T}_{Y'}.
\]
Since the set of the $\pi'$'s dominating $\pi$ is cofinal in the directed set of modifications of $X$, the above definition indeed yields a nef b-class $\mathbb{D}\cap \mathbb{T}$ over $X$.
\begin{remark}
    By exactly the same argument, if $\pi'\colon Y'\rightarrow X$ is a proper bimeromorphic map from a Kähler manifold dominating $\pi$ through $g\colon Y'\rightarrow Y$, then the value $(\mathbb{D}\cap \mathbb{T})_{Y'}$ is given by exactly the same formula \eqref{eq:DcapTY'}.
\end{remark}

\begin{lemma}\label{lma:int_indep_real}
    The product $\mathbb{D} \cap \mathbb{T}$ in \cref{def:b_div_int_general} is independent of the choice of the realization $(\pi\colon Y\rightarrow X,\alpha)$ of $\mathbb{D}$.
\end{lemma}
\begin{proof}    
    Let $(\pi\colon Y'\rightarrow X,\alpha')$ be another realization of $\mathbb{D}$. We want to show that $\mathbb{D}\cap \mathbb{T}$ defined with respect to the two realizations are the same. For this purpose, we may assume that $\pi'$ dominates $\pi$ through a morphism $g\colon Y'\rightarrow Y$ and $\alpha'=g^*\alpha$. 
    
    Consider a modification $\pi''\colon Y''\rightarrow X$ dominating $Y'$ through a morphism $h\colon Y''\rightarrow Y'$. Then have a commutative diagram as \eqref{eq:three_modif}.
    Our assertion now means
    \[
        \left(h^*\alpha'\right)\cap \mathbb{T}_{Y''}=\left((g\circ h)^*\alpha\right)\cap \mathbb{T}_{Y''},
    \]
    which is clear.
\end{proof}
As a consequence, we have:
\begin{corollary}\label{cor:Cbdivint_indepmodel}
Let $\mathbb{D}$ be a nef Cartier b-divisor over $X$ and $\mathbb{T}\in \bDiv^q(X)$ be a nef b-class.
    Let $\pi\colon Y\rightarrow X$ be a proper bimeromorphic morphism from a Kähler manifold $Y$. Then $\mathbb{D}\cap \mathbb{T}$ in \cref{def:int_general1} is the same as $\mathbb{D}\cap \mathbb{T}$ when $\mathbb{D}$ and $\mathbb{T}$ are regarded as elements in $\bDiv^1(Y)$ and $\bDiv^p(Y)$.
\end{corollary}

We prove a few elementary properties of the intersection product. These properties will soon be replaced by the more general versions.
\begin{lemma}\label{lma:Dcapmono1}
    Let $\mathbb{D},\mathbb{D}'$ be Cartier nef b-divisors over $X$ and $\mathbb{T},\mathbb{T}'\in \bDiv^p(X)$ be nef b-classes. 
    \begin{enumerate}
        \item Assume that $\mathbb{D}\geq\mathbb{D}'$, then
        \[
            \mathbb{D}\cap \mathbb{T}\geq   \mathbb{D}'\cap \mathbb{T}.
        \]
        \item Assume that $\mathbb{T}\geq \mathbb{T}'$, then
        \[
            \mathbb{D}\cap \mathbb{T}\geq \mathbb{D}\cap \mathbb{T}'.
        \]
    \end{enumerate}
\end{lemma}
\begin{proof}
    Thanks to \cref{cor:Cbdivint_indepmodel}, we may assume that $\mathbb{D}$ and $\mathbb{D}'$ are both realized on $X$ by nef classes $\alpha,\alpha'\in \mathrm{H}^{1,1}(X,\mathbb{R})$. Fix a modification $\pi\colon Y\rightarrow X$.

    (1) We need to show that if $\alpha\geq \alpha'$, then
    \[
    \pi^*\alpha\cap \mathbb{T}_Y \geq_Y \pi^*\alpha'\cap \mathbb{T}_Y.
    \]
    This follows from \cref{prop:pos_relation}.

    (2) We need to show that if $\mathbb{T}_Y\geq_Y \mathbb{T}'_Y$, then
    \[
    \pi^*\alpha\cap \mathbb{T}_Y\geq_Y \pi^*\alpha\cap \mathbb{T}'_Y.
    \]
    This follows from \cref{prop:nefint_pre_pos}.
\end{proof}

\begin{lemma}\label{lma:Cartnefblinear}
    Let $\mathbb{D},\mathbb{D}'$ be nef Cartier b-divisors over $X$, $\lambda\geq 0$ and $\mathbb{T},\mathbb{T}'\in \bDiv^p(X)$ be nef b-classes. Then
    \begin{enumerate}
        \item $(\mathbb{D}+\mathbb{D}')\cap \mathbb{T}=\mathbb{D}\cap \mathbb{T}+\mathbb{D}'\cap \mathbb{T}$;
        \item $\mathbb{D}\cap (\mathbb{T}+\mathbb{T}')=\mathbb{D}\cap \mathbb{T}+\mathbb{D}\cap \mathbb{T}'$;
        \item $(\lambda\mathbb{D})\cap \mathbb{T}=\lambda\left( \mathbb{D}\cap \mathbb{T} \right)$;
        \item $\mathbb{D}\cap (\lambda\mathbb{T})=\lambda\left( \mathbb{D}\cap \mathbb{T} \right)$.
    \end{enumerate}
\end{lemma}
\begin{proof}
Thanks to \cref{cor:Cbdivint_indepmodel}, we may assume that $\mathbb{D}$ and $\mathbb{D}'$ are both realized on $X$, then all these assertions are obvious.
\end{proof}

Next we come to the general intersection theory between a nef b-divisor and a nef b-class.
\begin{definition}\label{def:int_general1}
    Let $\mathbb{D}$ be a nef b-divisor over $X$ and $\mathbb{T}\in \bDiv^q(X)$ be a nef b-class. Then we define the nef b-class $\mathbb{D} \cap \mathbb{T}\in \bDiv^{q+1}(X)$ as follows: Take a decreasing sequence $(\mathbb{D}_i)_i$ of nef Cartier b-divisors over $X$ converging to $\mathbb{D}$ whose existence is guaranteed by \cref{cor:Demapp}, then let
    \begin{equation}\label{eq:DcapT_gen}
        \mathbb{D}\cap \mathbb{T}\coloneqq \lim_{i\to\infty}\mathbb{D}_i\cap \mathbb{T}.
    \end{equation}
\end{definition}
Thanks to \cref{lma:Dcapmono1}, the sequence $(\mathbb{D}_i\cap \mathbb{T})_i$ is decreasing. Since we consider the projective limit topology on the space of b-classes, the convergence in \eqref{eq:DcapT_gen} means the convergence of the components on each model. The existence of the limit in \eqref{eq:DcapT_gen} then follows from \cref{prop:mono_conv}. 

\begin{lemma}\label{lma:int_well_de}
    The intersection $\mathbb{D}\cap \mathbb{T}$ in \cref{def:int_general1} is independent of the choice of the sequence $(\mathbb{D}_i)_i$.
\end{lemma}
In particular, when $\mathbb{D}$ is Cartier, the b-class $\mathbb{D}\cap \mathbb{T}$ defined in \cref{def:int_general1} coincides with that in \cref{def:b_div_int_general}.
\begin{proof}
    We continue to use the notations in \cref{def:int_general1}.

    We claim that for each nef Cartier b-divisor $\mathbb{D}'$ over $X$ with $\mathbb{D}'\geq \mathbb{D}$, we have
    \begin{equation}\label{eq:DcapTdom_temp1}
    \lim_{i\to\infty} \mathbb{D}_i\cap \mathbb{T} \leq \mathbb{D}'\cap \mathbb{T}.
    \end{equation}
    From this assertion, the lemma trivially follows.

    In order to prove \eqref{eq:DcapTdom_temp1}, it suffices to establish
    \begin{equation}\label{eq:DcapTdom_temp2}
    \lim_{i\to\infty} \left( \mathbb{D}_i\cap \mathbb{T}\right)_X \leq_X \left( \mathbb{D}'\cap \mathbb{T} \right)_X.
    \end{equation}
    In fact, applying \eqref{eq:DcapTdom_temp2} to $\mathbb{D},\mathbb{D}'$, $\mathbb{D}_i$ and $\mathbb{T}$ regarded as b-classes over a modification of $X$, we conclude the general statement \eqref{eq:DcapTdom_temp1}.

    Thanks to \cref{prop:pos_push}, in order to prove \eqref{eq:DcapTdom_temp2}, we may further assume that $\mathbb{D}'$ is realized on $X$ by a nef class $\alpha\in \mathrm{H}^{1,1}(X,\mathbb{R})$. Then we are left with
    \begin{equation}\label{eq:DcapTdom_temp3}
    \lim_{i\to\infty} \left( \mathbb{D}_i\cap \mathbb{T}\right)_X \leq_X \alpha\cap \mathbb{T}_X.
    \end{equation}

    Fix a Kähler class $\beta\in \mathrm{H}^{1,1}(X,\mathbb{R})$. For each $\epsilon>0$, we can find $i_0>0$ so that when $i>i_0$,
    \[
    \mathbb{D}_{i,X}\leq \alpha+\epsilon \beta.
    \]
    Fix an $i>i_0$. Then
    \[
    \mathbb{D}_i\leq \mathbb{D}(\alpha+\epsilon\beta)= \mathbb{D}(\alpha)+\epsilon \mathbb{D}(\beta),
    \]
    where the inequality follows from \cref{lma:mod_nef_dom} and the equality follows from \cref{prop:Dsupadd}.
    Therefore, by \cref{lma:Dcapmono1} and \cref{lma:Cartnefblinear},
    \[
        \left( \mathbb{D}_i\cap \mathbb{T} \right)_X\leq_X \left( \mathbb{D}(\alpha)\cap \mathbb{T} \right)_X +\epsilon \left( \mathbb{D}(\beta)\cap \mathbb{T}\right)_X.
    \]
    Letting $i\to \infty$ and then $\epsilon \to 0+$, \eqref{eq:DcapTdom_temp3} follows.
\end{proof}

\begin{lemma}\label{lma:mod_nef_dom}
    Let $\pi\colon Y\rightarrow X$ be a proper bimeromorphic morphism from a Kähler manifold $Y$. Then for any modified nef class $\alpha\in \mathrm{H}^{1,1}(Y,\mathbb{R})$, we have
    \[
    \alpha\leq \pi^*\pi_*\alpha.
    \]
\end{lemma}
\begin{proof}
    Without loss of generality, we may assume that $\alpha$ is big. Take a current $T$ with minimal singularities $T\in \alpha$. Then $\pi^*\pi_*T$ is a closed positive $(1,1)$-current in  $\pi^*\pi_*\alpha$ whose regular part with respect to Siu's decomposition coincides with $T$. Our assertion follows.
\end{proof}

\begin{proposition}\label{prop:int_indp_model}
    Let $\pi\colon Y\rightarrow X$ be a modification. Let $\mathbb{D}$ be a nef b-divisor over $X$ and $\mathbb{T}\in \bDiv^q(X)$ be a nef b-class.
    Then $\mathbb{D}\cap \mathbb{T}$ in \cref{def:b_div_int_general} remains invariant if we regard $\mathbb{D}$ and $\mathbb{T}$ as elements in $\bDiv^1(Y)$ and $\bDiv^p(Y)$.
\end{proposition}
\begin{proof}
    This is a simple consequence of the Cartier case proved in \cref{cor:Cbdivint_indepmodel}.
\end{proof}

\begin{proposition}\label{prop:b_int_lin}
    Let $\mathbb{D},\mathbb{D}'$ be nef b-divisors over $X$, $\lambda\geq 0$ and $\mathbb{T},\mathbb{T}'\in \bDiv^p(X)$ be nef b-classes. Then
    \begin{enumerate}
        \item $(\mathbb{D}+\mathbb{D}')\cap \mathbb{T}=\mathbb{D}\cap \mathbb{T}+\mathbb{D}'\cap \mathbb{T}$;
        \item $\mathbb{D}\cap (\mathbb{T}+\mathbb{T}')=\mathbb{D}\cap \mathbb{T}+\mathbb{D}\cap \mathbb{T}'$;
        \item $(\lambda\mathbb{D})\cap \mathbb{T}=\lambda\left( \mathbb{D}\cap \mathbb{T} \right)$;
        \item $\mathbb{D}\cap (\lambda\mathbb{T})=\lambda\left( \mathbb{D}\cap \mathbb{T} \right)$.
    \end{enumerate}
\end{proposition}
\begin{proof}
     These are straightforward consequences of the Cartier case proved in  \cref{lma:Cartnefblinear}.
\end{proof}

If we have nef b-divisors $\mathbb{D}_1,\ldots,\mathbb{D}_p$ over $X$ and a nef b-class $\mathbb{T}\in \bDiv^q(X)$, then we can define
\[
\mathbb{D}_1\cap \cdots \cap \mathbb{D}_p\cap \mathbb{T}\coloneqq \mathbb{D}_1\cap \left( \mathbb{D}_2\cap \cdots \cap \mathbb{D}_p\cap \mathbb{T}\right)\in \bDiv^{p+q}(X)
\]
inductively.

\begin{example}
    Let $\mathbb{D}_1,\ldots,\mathbb{D}_p$ be nef b-divisors over $X$. Consider $1\in \bDiv^0(X)$ with the identification in \cref{ex:bclass_specialcase} in mind. We have
    \[
    \mathbb{D}_1\cap \dots \cap \mathbb{D}_p \cap 1=\mathbb{D}_1\cap \dots \cap \mathbb{D}_p.
    \]
    To see this, by induction on $p$, we may assume that $p=1$, then our assertion becomes
    \[
    \mathbb{D}_1\cap 1=\mathbb{D}_1.
    \]
    By approximation, it suffices to prove this when $\mathbb{D}_1$ is Cartier.
    Furthermore, \cref{cor:Cbdivint_indepmodel} allows us to reduce to the case where $\mathbb{D}_1$ is realized on $X$. Then our assertion is trivial.

    When $p=n$, the product $\mathbb{D}_1\cap \dots \cap \mathbb{D}_n$ recovers the intersection product
    \[
    \left(\mathbb{D}_1,\ldots,\mathbb{D}_n \right)
    \]
    studied in \cite{Xiabdiv} after taking the canonical identification $\bDiv^n(X)=\mathbb{R}$ in \cref{ex:bclass_specialcase} into account.
\end{example}

We prove a monotonicity result:
\begin{proposition}\label{prop:mono_int_b}
    Let $\mathbb{D},\mathbb{D}'$ be nef b-divisors over $X$ and $\mathbb{T},\mathbb{T}'\in \bDiv^p(X)$ be nef b-classes. 
    \begin{enumerate}
        \item Assume that $\mathbb{D}\geq\mathbb{D}'$, then
    \[
        \mathbb{D}\cap \mathbb{T}\geq   \mathbb{D}'\cap \mathbb{T}.
    \]
    \item Assume that $\mathbb{T}\geq\mathbb{T}'$, then
    \[
        \mathbb{D}\cap \mathbb{T}\geq   \mathbb{D}\cap \mathbb{T}'.
    \]
    \end{enumerate}
\end{proposition}
\begin{proof}
     (1) Take a nef Cartier b-divisor $\mathbb{D}''$ over $X$ with $\mathbb{D}''\geq \mathbb{D}$. We need to show that
     \[
     \mathbb{D}''\cap \mathbb{T}\geq \mathbb{D}'\cap \mathbb{T}.
     \]
     This is already proved during the proof of \cref{lma:int_well_de}.

    (2) By definition, we may assume that $\mathbb{D}$ is Cartier. Then the assertion follows from \cref{lma:Dcapmono1}.
\end{proof}

\begin{theorem}\label{thm:bdiv_int_cont1}
    Let $(\mathbb{D}_i^j)_{j\in J}$ be decreasing nets of nef b-divisors over $X$ with limits $\mathbb{D}_i$ for each $i=1,\ldots,p$. Let $(\mathbb{T}^j)_{j\in J}$ be a decreasing net of nef b-classes in $\bDiv^q(X)$ with limit $\mathbb{T}$. Then
    \[
        \lim_{j\in J} \mathbb{D}_1^j\cap \cdots \cap \mathbb{D}_p^j\cap \mathbb{T}^j=\mathbb{D}_1\cap \cdots \cap \mathbb{D}_p\cap \mathbb{T}.
    \]
\end{theorem}
\begin{proof}
    By induction on $p\geq 1$ and using \cref{prop:mono_int_b}, it suffices to prove our assertion in the case $p=1$.
    In this case, we omit the subindex $i$ and we are left with
    \begin{equation}\label{eq:mon_convthm_temp1}
        \lim_{j\in J} \mathbb{D}^j\cap \mathbb{T}^j=\mathbb{D}\cap \mathbb{T}.
    \end{equation}
    We first observe that the limit on the left-hand side of \eqref{eq:mon_convthm_temp1} exists, as a consequence of \cref{prop:mono_int_b} and \cref{prop:mono_conv}. Furthermore, the $\geq$ direction in \eqref{eq:mon_convthm_temp1} holds.

    \textbf{Step~1}. We first assume that $\mathbb{T}^j=\mathbb{T}$ for all $j\in J$. In this case, take a decreasing sequence $(\mathbb{D}'^{k})_{k>0}$ of Cartier nef b-divisors with limit $\mathbb{D}$. Fix $k>0$, we need to show that
    \begin{equation}\label{eq:limDjTleq}
    \lim_{j\in J}\mathbb{D}^j\cap \mathbb{T} \leq \mathbb{D}'^{k}\cap \mathbb{T}.
    \end{equation}
   
    For this purpose, we may further assume that $\mathbb{D}'^k$ is realized on $X$ by a nef class $\alpha\in \mathrm{H}^{1,1}(X,\mathbb{R})$. Fix a Kähler class $\beta$ on $X$, then for each $\epsilon>0$, we can find $j_0\in J$ so that when $j\geq j_0$, we have
    \[
    \mathbb{D}^j_X\leq \alpha+\epsilon\beta.
    \]
    Therefore, thanks to \cref{lma:mod_nef_dom} and \cref{prop:Dsupadd},
    \[
    \mathbb{D}^j \leq \mathbb{D}'^k+\epsilon\mathbb{D}(\beta),\quad j\geq j_0.
    \]
    Using \cref{prop:mono_int_b} and \cref{prop:b_int_lin}, we find
    \[
    \mathbb{D}^j\cap \mathbb{T}\leq \mathbb{D}'^k\cap \mathbb{T}+\epsilon\mathbb{D}(\beta)\cap \mathbb{T},\quad j\geq j_0.
    \]
    Taking limit with respect to $j$ and then letting $\epsilon\to 0+$, we derive \eqref{eq:limDjTleq}.

    \textbf{Step~2}. We assume that $\mathbb{D}^j=\mathbb{D}$ for all $j\in J$. Then we need to show that
    \begin{equation}\label{eq:limDcapTjlim_temp1}
    \lim_{j\in J} \mathbb{D}\cap \mathbb{T}^j\leq \mathbb{D}\cap \mathbb{T}.
    \end{equation}
    In this case, take a decreasing sequence $(\mathbb{D}'^{k})_{k>0}$ of nef Cartier b-divisors with limit $\mathbb{D}$. Suppose that we can prove \eqref{eq:limDcapTjlim_temp1} with $\mathbb{D}'^k$ in place of $\mathbb{D}$, then due to \cref{prop:mono_int_b} we have
    \[
        \lim_{j\in J} \mathbb{D}\cap \mathbb{T}^j\leq  \lim_{j\in J} \mathbb{D}'^k\cap \mathbb{T}^j\\
        = \mathbb{D}'^k\cap \mathbb{T}
    \]
    for each $k>0$. Letting $k\to\infty$, we conclude \eqref{eq:limDcapTjlim_temp1}.

    Therefore, we may assume that $\mathbb{D}$ is Cartier when proving \eqref{eq:limDcapTjlim_temp1}. Replacing $X$ by a modification, we may further assume that $\mathbb{D}$ is realized on $X$ by a nef class $\alpha\in \mathrm{H}^{1,1}(X,\mathbb{R})$. Then \eqref{eq:limDcapTjlim_temp1} means the following: Let $\pi\colon Y\rightarrow X$ be a modification, then
    \[
    \lim_{j\in J} \pi^*\alpha \cap \mathbb{T}^j_Y\leq_Y \pi^*\alpha \cap \mathbb{T}_Y,
    \]
    which is obvious.

    \textbf{Step~3}. We prove the general case. 
    
    For each fixed $j_0\in J$, by Step~1 and \cref{prop:mono_int_b}, we have
    \[
    \mathbb{D}\cap \mathbb{T}^{j_0}=\lim_{j\in J} \mathbb{D}^j\cap \mathbb{T}^{j_0}
    \geq  \lim_{j\in J} \mathbb{D}^j\cap \mathbb{T}^{j}.
    \]
    Taking the limit with respect to $j_0$, we conclude \eqref{eq:mon_convthm_temp1} using Step~2.
    
\end{proof}

We summarize the properties of our intersection product. These properties generalize the corresponding statements for movable intersection product as in  \cref{prop:mov_int_prop}. They also generalize the corresponding properties when $p=n$ proved in \cite{Xiabdiv}.
\begin{proposition}\label{prop:bdiv_int_prop}
    Let $\mathbb{D}_1,\ldots,\mathbb{D}_p,\mathbb{D}_1',\mathbb{D}'$ be nef b-divisors over $X$, $\mathbb{T},\mathbb{T}'\in \bDiv^q(X)$ be nef b-classes, $\alpha\in \mathrm{H}^{1,1}(X,\mathbb{R})$ be a nef class, and $\lambda\geq 0$. We have the following properties:
    \begin{enumerate}
        \item The product is symmetric: For each permutation $\sigma$ of $\{1,\ldots,p\}$,  we have
        \[
            \mathbb{D}_1\cap \cdots \cap \mathbb{D}_p \cap \mathbb{T}=\mathbb{D}_{\sigma(1)}\cap \cdots \cap \mathbb{D}_{\sigma(p)} \cap \mathbb{T}.
        \]
        \item The product is additive in each variable: 
        \[
        \begin{aligned}
            \left( \mathbb{D}_1+\mathbb{D}_1'\right)\cap\mathbb{D}_2\cap  \cdots \cap \mathbb{D}_p \cap \mathbb{T}=& \mathbb{D}_1\cap\mathbb{D}_2\cap  \cdots \cap \mathbb{D}_p \cap \mathbb{T}+\mathbb{D}'_1\cap\mathbb{D}_2\cap  \cdots \cap \mathbb{D}_p \cap \mathbb{T},\\
            \mathbb{D}_1\cap \cdots \cap \mathbb{D}_p \cap \left(\mathbb{T}+\mathbb{T}'\right)=&\mathbb{D}_1\cap \cdots \cap \mathbb{D}_p \cap \mathbb{T}+\mathbb{D}_1\cap \cdots \cap \mathbb{D}_p \cap \mathbb{T}'.
        \end{aligned}
        \]
        \item The product is homogeneous in each variable:
        \[
        \begin{aligned}
            \left(\lambda\mathbb{D}_1\right)\cap\mathbb{D}_2\cap  \cdots \cap \mathbb{D}_p \cap \mathbb{T}=&\lambda \left(\mathbb{D}_1\cap\mathbb{D}_2\cap  \cdots \cap \mathbb{D}_p \cap \mathbb{T}\right),\\
            \mathbb{D}_1\cap \cdots \cap \mathbb{D}_p \cap \left(\lambda\mathbb{T}\right)=&\lambda\left(\mathbb{D}_1\cap \cdots \cap \mathbb{D}_p \cap \mathbb{T} \right).
        \end{aligned}
        \]
        \item The product is increasing in each variable: If we assume $\mathbb{D}_1\geq \mathbb{D}_1'$, we have
        \[
        \mathbb{D}_1\cap\mathbb{D}_2\cap  \cdots \cap \mathbb{D}_p \cap \mathbb{T}\geq \mathbb{D}_1'\cap\mathbb{D}_2\cap  \cdots \cap \mathbb{D}_p \cap \mathbb{T}.
        \]
        Similarly, if we assume $\mathbb{T}\geq \mathbb{T}'$, then
        \[
        \mathbb{D}_1\cap \cdots \cap \mathbb{D}_p \cap \mathbb{T}\geq \mathbb{D}_1\cap \cdots \cap \mathbb{D}_p \cap \mathbb{T}'.
        \]
        \item We have
        \[
        \Bigl(\mathbb{D}_1\cap \cdots \cap \mathbb{D}_p\cap \mathbb{D}(\alpha)\Bigr)_X=\Bigl(\mathbb{D}_1\cap \cdots \cap \mathbb{D}_p\Bigr)_X\cap \alpha.
        \]
    \end{enumerate}
\end{proposition}
\begin{proof}
    (1) We can easily reduce to the case $p=2$, and then we need to show
    \[
    \mathbb{D}_1\cap \mathbb{D}_2\cap \mathbb{T}=\mathbb{D}_2\cap \mathbb{D}_1\cap \mathbb{T}.
    \]
    By approximation using \cref{thm:bdiv_int_cont1}, we can easily reduce to the case where $\mathbb{D}_1$ and $\mathbb{D}_2$ are both Cartier.
    Thanks to \cref{cor:Cbdivint_indepmodel}, we may further assume that $\mathbb{D}_1$ and $\mathbb{D}_2$ are realized on $X$ by nef classes $\alpha_1,\alpha_2\in \mathrm{H}^{1,1}(X,\mathbb{R})$. In this case our assertion becomes the following: For any modification $\pi\colon Y\rightarrow X$, we have
    \[
    \pi^*\alpha_1\cap \left( \pi^*\alpha_2\cap \mathbb{T}_Y\right)=\pi^*\alpha_2\cap \left( \pi^*\alpha_1\cap \mathbb{T}_Y\right),
    \]
    which is obvious.

    (2) (3) These are consequences of \cref{prop:b_int_lin}.

    (4) This is a consequence of \cref{prop:mono_int_b}.

    (5) This follows immediately from the definition \cref{def:b_div_int_general}.
\end{proof}

\begin{example}\label{ex:int_nef_classes}
    Let $\alpha_1,\ldots,\alpha_p\in \mathrm{H}^{1,1}(X,\mathbb{R})$ be nef classes. Then for any modification $\pi\colon Y\rightarrow X$, we have
    \[
        \Bigl( \mathbb{D}(\alpha_1)\cap \dots \cap \mathbb{D}(\alpha_p) \Bigr)_Y=\pi^*\alpha_1\cap \cdots \cap \pi^*\alpha_p.
    \]
    This is an immediate consequence of \cref{prop:bdiv_int_prop}(5). We shall generalize this statement in \cref{cor:mov_int_generalized}.
\end{example}

The intersection theory of b-divisors is closely related to the non-pluripolar product.
\begin{theorem}\label{thm:DTandnpp}
    Let $T_1,\ldots,T_p$ be $\mathcal{I}$-good closed positive $(1,1)$-currents. Then
    \begin{equation}\label{eq:DTandnpp}
        \Bigl( \mathbb{D}(T_1)\cap \dots \cap \mathbb{D}(T_p) \Bigr)_Y= \left\{ \pi^*T_1\wedge \cdots \wedge \pi^*T_p \right\}
    \end{equation}
    for all modifications $\pi\colon Y\rightarrow X$.
\end{theorem}
\begin{proof}
    First observe that the right-hand side of \eqref{eq:DTandnpp} does define a nef b-class, thanks to \cref{prop:bime_npp} and \cref{prop:npp_strongpos}. Since the assumptions of the theorem remain valid after replacing the $T_i$'s by $\pi^*T_i$'s as well, it then suffices to prove 
    \begin{equation}\label{eq:DTandnpp_temp2}
        \Bigl( \mathbb{D}(T_1)\cap \dots \cap \mathbb{D}(T_p) \Bigr)_X= \left\{ T_1\wedge \cdots \wedge T_p \right\}
    \end{equation}

    In order to prove \eqref{eq:DTandnpp_temp2}, we can easily reduce to the case where the $T_i$'s are Kähler currents: Take a Kähler form $\omega$ and replace $T_i$ by $T_i+\epsilon \omega$ for a small enough $\epsilon>0$. If we manage to prove
    \[
     \Bigl( \mathbb{D}(T_1+\epsilon\omega)\cap \dots \cap \mathbb{D}(T_p+\epsilon\omega) \Bigr)_X= \left\{ (T_1+\epsilon\omega)\wedge \cdots \wedge (T_p+\epsilon\omega) \right\},
    \]
    letting $\epsilon\to 0+$, \eqref{eq:DTandnpp_temp2} follows, since $\mathbb{D}$ and the intersection product are both linear, see \eqref{eq:D_curr_linear} and \cref{prop:bdiv_int_prop}.

    Now assume that the $T_i$'s are Kähler currents. Take quasi-equisingular approximations $(T_i^j)_{j>0}$ of $T_i$ for each $i=1,\ldots,p$. The $\mathcal{I}$-goodness of $T_i$ guarantees that $T_i^j\xrightarrow{d_S}T_i$ as $j\to\infty$ for each $i$.
    Suppose that we can prove
    \[
    \left( \mathbb{D}(T_1^j)\cap \dots \cap \mathbb{D}(T_p^j) \right)_X= \left\{ T_1^j\wedge \cdots \wedge T_p^j \right\},\quad j>0.
    \]
    Thanks to \cite[Proposition~4.9]{Xiabdiv}, for each $i=1,\ldots,p$, the sequence $(\mathbb{D}(T_i^j))_j$ is decreasing with limit $\mathbb{D}(T_i)$. Letting $j\to\infty$ and applying \cref{thm:dscontnpp1} and \cref{thm:bdiv_int_cont1}, we conclude \eqref{eq:DTandnpp_temp2}.

    So we may assume that the $T_i$'s have analytic singularities. We may replace $X$ by a modification and assume that $T_1,\ldots,T_p$ in fact have log singularities. But then since both sides of \eqref{eq:DTandnpp} remain unchanged if we replace $T_i$ by $\Reg T_i$, it follows that we may assume that $T_i$ has bounded potential for each $i=1,\ldots,p$. In this case, $\{T_i\}$ is nef and $\mathbb{D}(T_i)=\mathbb{D}(\{T_i\})$ for $i=1,\ldots,p$. 
    In particular, by \cref{ex:int_nef_classes},
    \[
        \Bigl( \mathbb{D}(T_1)\cap \cdots \cap \mathbb{D}(T_p)\Bigr)_X=\{T_1\}\cap \cdots \cap \{T_p\}.
    \]
    By Bedford--Taylor theory, the right-hand side of \eqref{eq:DTandnpp_temp2} is the same class, and \eqref{eq:DTandnpp_temp2} follows.
\end{proof}

\begin{corollary}\label{cor:mov_int_generalized}
    Let $\alpha_1,\ldots,\alpha_p\in \mathrm{H}^{1,1}(X,\mathbb{R})$ be pseudoeffective classes, then for any modification $\pi\colon Y\rightarrow X$, we have
    \[
        \Bigl( \mathbb{D}(\alpha_1)\cap \dots \cap \mathbb{D}(\alpha_p) \Bigr)_Y= \langle \pi^*\alpha_1\wedge \cdots \wedge \pi^*\alpha_p \rangle.
    \]
\end{corollary}
In other words, the intersection theory of nef b-divisors generalizes the movable intersection theory, as promised in the introduction.
\begin{proof}
    We may assume that $Y=X$.

    Fix a Kähler class $\beta$ on $X$ and replace each $\alpha_i$ by $\alpha_i+\epsilon\beta$, we may assume that each $\alpha_i$ is big. 
    This is allowed thanks to \cref{prop:mov_int_prop}(6) and \cref{prop:bdiv_int_prop}.
    
    Then it suffices to apply \cref{thm:DTandnpp} to the currents $T_i\in \alpha_i$ with minimal singularities.
\end{proof}

As an application, we prove the following fundamental result regarding the movable intersection.
\begin{proposition}\label{prop:mov_int_push_ineq}
    Let $\pi\colon Y\rightarrow X$ be a proper bimeromorphic morphism from a Kähler manifold $Y$. Let $\alpha_1,\ldots,\alpha_p\in \mathrm{H}^{1,1}(Y,\mathbb{R})$ be pseudoeffective classes. Then
    \begin{equation}\label{eq:mov_int_push_ineq}
    \pi_*\langle \alpha_1\wedge \cdots \wedge \alpha_p\rangle\leq_X \langle \pi_*\alpha_1\wedge \cdots \wedge \pi_*\alpha_p\rangle.
    \end{equation}
\end{proposition}
\begin{proof}
    Fix a Kähler class $\beta$ on $Y$ and replacing $\alpha_i$ by $\alpha_i+\epsilon \beta$ for some small enough $\epsilon>0$, we may assume that the $\alpha_i$'s are big. This is allowed thanks to \cref{prop:mov_int_prop}(6).

    By \cite[Lemma~3.12]{Xiabdiv}, we have
    \[
        \mathbb{D}(\pi_*\alpha_i)\geq \mathbb{D}(\alpha_i),\quad i=1,\ldots,p.
    \]
    
    Now it follows from \cref{prop:bdiv_int_prop} that
    \[
        \mathbb{D}(\pi_*\alpha_1)\cap \cdots \cap \mathbb{D}(\pi_*\alpha_p)\geq \mathbb{D}(\alpha_1)\cap \cdots \cap \mathbb{D}(\alpha_p).
    \]
    In particular,
    \[
        \Bigl( \mathbb{D}(\pi_*\alpha_1)\cap \cdots \cap \mathbb{D}(\pi_*\alpha_p) \Bigr)_X\geq_X \Bigl(\mathbb{D}(\alpha_1)\cap \cdots \cap \mathbb{D}(\alpha_p)\Bigr)_X.
    \]
    Using \cref{cor:mov_int_generalized}, we can rewrite this as \eqref{eq:mov_int_push_ineq}.
\end{proof}
\begin{theorem}\label{thm:nefbint_aslimit}
    Let $\mathbb{D}_1,\ldots,\mathbb{D}_p$ be nef b-divisors over $X$. Then
    \begin{equation}\label{eq:int_b_div_mov_limit}
        \Bigl( \mathbb{D}_1 \cap \dots \cap \mathbb{D}_p \Bigr)_X= \lim_{\pi\colon Y\rightarrow X} \pi_* \Bigl\langle \mathbb{D}_{1,Y}\wedge \dots \wedge \mathbb{D}_{p,Y} \Bigr\rangle,
    \end{equation}
    where $\pi\colon Y\rightarrow X$ runs over all modifications of $X$.
\end{theorem}
We first observe that the right-hand side of \eqref{eq:int_b_div_mov_limit} is decreasing with respect to $\leq_X$ thanks to \cref{prop:mov_int_push_ineq}, and hence the limit exists by \cref{prop:mono_conv}.

The striking point is that the left-hand side of \eqref{eq:int_b_div_mov_limit} is a linear intersection product, while each term of the right-hand side is non-linear!
\begin{proof}
    \textbf{Step~1}. We first prove the $\leq_X$ inequality. 
    
    For this purpose, it suffices to prove the following:
    \[
        \left( \mathbb{D}_1 \cap \dots \cap \mathbb{D}_p \right)_X \leq_X \langle \mathbb{D}_{1,X}\wedge \dots \wedge \mathbb{D}_{p,X} \rangle.
    \]
    After adding a small multiple of $\mathbb{D}(\omega)$ to each $\mathbb{D}_i$, where $\omega$ is a Kähler form on $X$, we may assume that the $\mathbb{D}_i$'s are big and $\mathbb{D}_i=\mathbb{D}(T_i)$ for some $\mathcal{I}$-good non-divisorial Kähler currents $T_i$ for each $i=1,\ldots,p$. Then due to \cref{thm:DTandnpp}, it suffices to show that
    \[
        \left\{ T_1\wedge \cdots \wedge T_p \right\} \leq_X \left\{ T_{1,\min}\wedge \cdots \wedge T_{p,\min} \right\},
    \]
    where $T_{i,\min}\in \mathbb{D}_{i,X}$ is a current with minimal singularities. Then our assertion follows from the monotonicity theorem \cref{thm:mono_thm_partial}.

    \textbf{Step~2}. We reduce \eqref{eq:int_b_div_mov_limit} to the case of $p=n$ and the $\mathbb{D}_i$'s are big and $\mathbb{D}_i=\mathbb{D}(T_i)$ for some non-divisorial $\mathcal{I}$-good Kähler current $T_i$ for each $i=1,\ldots,p$.

    Thanks to \cref{lma:posclass_vanish} and Step~1, we may assume $p=n$. If suffices to handle the case where the $\mathbb{D}_i$'s are big as well. In fact, take a Kähler form $\omega$ on $X$. Suppose that for $\epsilon>0$ we have established
    \[
    \biggl( \Bigl( \mathbb{D}_1+\epsilon\mathbb{D}(\omega)\Bigr) \cap \dots \cap \Bigl( \mathbb{D}_n+\epsilon\mathbb{D}(\omega)\Bigr) \biggr)_X\geq  \lim_{\pi\colon Y\rightarrow X}\Bigl\langle (\mathbb{D}_{1,Y}+\epsilon\pi^*\{\omega\})\wedge \dots \wedge (\mathbb{D}_{n,Y}+\epsilon\pi^*\{\omega\}) \Bigr\rangle,
    \]
    then from the monotonicity \cref{prop:mov_int_prop}(4), we have
    \[
   \biggl( \Bigl( \mathbb{D}_1+\epsilon\mathbb{D}(\omega)\Bigr) \cap \dots \cap \Bigl( \mathbb{D}_n+\epsilon\mathbb{D}(\omega)\Bigr) \biggr)_X\geq  \lim_{\pi\colon Y\rightarrow X} \langle \mathbb{D}_{1,Y}\wedge \dots \wedge \mathbb{D}_{n,Y} \rangle.
    \]
    Letting $\epsilon \to 0+$, our assertion \eqref{eq:int_b_div_mov_limit} follows.

    \textbf{Step~3}. 
    We need to show that
    \begin{equation}\label{eq:intTiD_temp1}
    \int_X T_1\wedge \cdots \wedge T_n \geq \lim_{\pi\colon Y\rightarrow X} \Bigl\langle \mathbb{D}(T_1)_Y \wedge \dots \wedge \mathbb{D}(T_n)_Y \Bigr\rangle.
    \end{equation}
    When the $T_i$'s have analytic singularities, we may replace $X$ by a modification and reduce to the case where the $T_i$'s have log singularities. Then replacing $T_i$ by $\Reg T_i$, we may assume that the $T_i$'s have bounded local potentials. Then the $\{T_i\}$'s are indeed nef and
    \[
    \int_X T_1\wedge \cdots \wedge T_n = \{T_1\}\cap \cdots \cap \{T_n\} =\Bigl\langle \{T_1\} \wedge \dots \wedge \{T_n\} \Bigr\rangle.
    \]

    In general, take quasi-equisingular approximations $(T_i^j)_{j>0}$ of each $T_i$. Then we know that for each $j>0$,
    \[
    \int_X T_1^j\wedge \cdots \wedge T_n^j \geq \lim_{\pi\colon Y\rightarrow X} \Bigl\langle \mathbb{D}(T_1^j)_Y \wedge \dots \wedge \mathbb{D}(T_n^j)_Y \Bigr\rangle\geq \lim_{\pi\colon Y\rightarrow X} \Bigl\langle \mathbb{D}(T_1)_Y \wedge \dots \wedge \mathbb{D}(T_n)_Y \Bigr\rangle.
    \]
    Letting $j\to\infty$ and applying \cref{thm:dscontnpp1}, we conclude \eqref{eq:intTiD_temp1}.
\end{proof}

\section{Restricted volume}\label{sec:resvol}
Let $X$ be a connected compact Kähler manifold of dimension $n$. In this section, $p$ will denote a non-negative integer.
\subsection{The trace operator}
We first recall the notion of trace operators introduced in \cite{DX24}. For the details, see \cite[Chapter~8]{Xiabook}.

Assume that $D$ is a prime divisor on $X$.

\begin{definition}
    Let $T$ be a closed positive $(1,1)$-current on $X$ with $\nu(T,D)=0$. We define $\Tr_D T$ as a closed positive $(1,1)$-current on $\widetilde{D}$, the normalization of $D$, modulo $P$-equivalence. The definition goes as follows:
    \begin{enumerate}
        \item When $T$ has analytic singularities, we can simply define $\Tr_D T$ as $T|_D$;
        \item in general, take a quasi-equisingular approximation $(T_j)_j$ of $T$ and define $\Tr_D T$ as an arbitrary $d_S$-limit of $(T_j|_{\tilde{D}})_j$.
    \end{enumerate}
\end{definition}
When $T$ is a Kähler current on $X$ with $\nu(T,D)=0$, we can take a representative $\Tr_D T$ as a Kähler current in $\{T\}|_{\widetilde{D}}$. 
When $\widetilde{D}$ is not smooth, this means that we can find a smooth closed real $(1,1)$-form $\theta\in \{T\}$ so that $\Tr_D T$ can be represented by $\theta|_{\widetilde{D}}+\ddc \varphi$ for some $\theta|_{\widetilde{D}}$-psh function $\varphi$, and the current $\Tr_D T$ dominates a Kähler form on $\tilde{D}$.
We always choose such a representative. 

In this case, we define the \emph{restricted volume} of $T$ on $X$ as
\[
\vol_{X|D}(T)\coloneqq \vol \left( \Tr_D T \right).
\]
Note that $\vol_{X|D}(T)$ is independent of the choice of $\Tr_D T$, since any two choices are $P$-equivalent.

More generally, if $T$ is a closed positive $(1,1)$-current on $X$ with $\nu(T,D)=0$, we take a Kähler form $\omega$ on $X$, then we can define
\begin{equation}\label{eq:resvolcurr}
\vol_{X|D}(T)\coloneqq \lim_{\epsilon\to 0+}\vol_{X|D}(T+\epsilon\omega)
\end{equation}
This definition agrees with the previous definition when $T$ is a Kähler current. When $\vol_{X|D}(T)>0$, we can always find a representative of $\Tr_D T$ in $\{T\}|_{\widetilde{D}}$, and the restricted volume $\vol_{X|D}(T)$ is just the volume of such a representative. See \cite[Example~8.1.6]{Xiabook} for the details.

When $\nu(T,D)>0$, we define
\[
\vol_{X|D}(T)\coloneqq 0.
\]

Next we recall that there is a notion of trace operators of nef b-divisors defined in \cite{Xiabdiv}. Given a nef b-divisor $\mathbb{D}$ over $X$, a prime divisor $D$ on $X$, then there is a canonical way to define a nef b-divisor $\Tr_D\mathbb{D}$ over $D$. Instead of recalling the lengthy definition, we use the following assertion proved in \cite{Xiabdiv} as the definition:
\begin{definition}\label{def:trbdiv}
Let $\mathbb{D}$ be a nef b-divisor over $X$. We define $\Tr_D \mathbb{D}$ as a nef b-divisor over $D$ as follows:
\begin{enumerate}
    \item When $\mathbb{D}=\mathbb{D}(T)$ for a closed positive Kähler current $T$ on $X$, then we let
    \[
    \Tr_D\mathbb{D}\coloneqq \mathbb{D}(\Tr_D T).
    \]
    \item In general, we define
    \[
        \Tr_D \mathbb{D}\coloneqq \lim_{\epsilon\to 0+} \Tr_D \Bigl( \mathbb{D}+\epsilon\mathbb{D}(\omega) \Bigr)
    \]
    for any Kähler form $\omega$ on $X$.
\end{enumerate}
\end{definition}
Note that in (2), $\mathbb{D}+\epsilon \mathbb{D}(\omega)$ for any $\epsilon>0$ satisfies the condition in (1), as a consequence of \cref{thm:main_paper1}.

A priori, $\Tr_D T$ is a current on $\widetilde{D}$, so $\Tr_D \mathbb{D}$ is just a nef b-divisor over $\widetilde{D}$. But by definition, nef b-divisors over $D$ are the same as those over $\widetilde{D}$, so we can also regard $\Tr_D \mathbb{D}$ as a nef b-divisor over $D$.

Whenever $\Tr_D T$ has a representative in $\{T\}|_{\tilde{D}}$ (in particular when $\vol_{X|D}(T)>0$), we always have
\begin{equation}\label{eq:TrDmathbbD}
\Tr_D \mathbb{D}=\mathbb{D}(\Tr_D T),
\end{equation}
where $\Tr_D T$ is understood as a representative in $\{T\}|_{\tilde{D}}$. This is proved in \cite[Theorem~7.5]{Xiabdiv}.

\begin{proposition}\label{prop:trace_indepmodel}
    Let $\pi\colon Y\rightarrow X$ be a modification with $D'$ as the strict transform of $D$, then for any nef b-divisor $\mathbb{D}$ over $X$, we have
    \begin{equation}\label{eq:trd=trd}
        \Tr_D \mathbb{D}=\Tr_{D'} \mathbb{D}.
    \end{equation}
\end{proposition}
Here on the right-hand side, we regard $\mathbb{D}$ as a nef b-divisor over $Y$ and $\Tr_{D'}\mathbb{D}$ is a nef b-divisor over $D'$, which can be canonically identified with a nef b-divisor over $D$.
\begin{proof}
    Take a Kähler form $\omega$ on $X$. It suffice to prove \eqref{eq:trd=trd} with $\mathbb{D}+\epsilon\mathbb{D}(\omega)$ in place of $\mathbb{D}$ for any $\epsilon>0$. We may then assume that  there is a non-divisorial $\mathcal{I}$-good Kähler current $T$ on $X$ so that $\mathbb{D}=\mathbb{D}(T)$. Then $\Tr_D T$ can be represented by a current in $\{T\}|_{\widetilde{D}}$. We fix such a representative.

    Consider the following commutative diagram 
    \[
    \begin{tikzcd}
\widetilde{D'} \arrow[r] \arrow[d, "\tilde{p}"] & D' \arrow[r, hook] \arrow[d, "p"] & Y \arrow[d, "\pi"] \\
\widetilde{D} \arrow[r]                             & D \arrow[r, hook]                 & X.                 
\end{tikzcd}
    \]
    Observe that $\nu(\pi^*T,D')=0$ by Zariski's main theorem.
    From the basic properties of the trace operator \cite[Lemma~8.2.1]{Xiabook}, we have
    \[
    \tilde{p}^*\Tr_D T \sim_{P} \Tr_{D'}\pi^*T.
    \]
    In particular, $\Tr_{D'}\pi^*T$ can be represented by a current in $\pi^*\{T\}|_{\widetilde{D'}}$, we fix such a representative.
    
    Taking the induced b-divisors, we get
    \[
        \mathbb{D}\left(\tilde{p}^*\Tr_D T\right) = \mathbb{D} \left( \Tr_{D'}\pi^*T \right).
    \]
    This is exactly \eqref{eq:trd=trd}.
\end{proof}

We need a few basic properties of the trace operator.
\begin{proposition}\label{prop:trace_basic}
    Let $\mathbb{D}$, $\mathbb{D}'$ be nef b-divisors over $X$, and $\lambda\geq 0$. Then we have the following properties:
    \begin{enumerate}
        \item When $\mathbb{D}\leq \mathbb{D}'$ and $\mathbb{D}'_X=\mathbb{D}_X$, we have 
        \[
        \Tr_D\mathbb{D}\leq \Tr_D \mathbb{D}'.
        \]
        \item The trace operator is additive:
        \[
        \Tr_D \left( \mathbb{D}+\mathbb{D}'\right)=\Tr_D \mathbb{D}+\Tr_D \mathbb{D}'.
        \]
        \item The trace operator is homogeneous:
        \[
        \Tr_D \left( \lambda\mathbb{D}\right)=\lambda\Tr_D \mathbb{D}.
        \]
    \end{enumerate}
\end{proposition}
\begin{proof}
    These properties all follow from the corresponding properties of the trace operator of currents. See \cite[Proposition~8.2.1]{Xiabook}.

    We only give a detailed proof to (1), and the other two assertions are similar. Take a Kähler form $\omega$ on $X$. It suffices to show that for each $\epsilon>0$, we have
    \[
    \Tr_D(\mathbb{D}+\epsilon\mathbb{D}(\omega))\leq \Tr_D(\mathbb{D}'+\epsilon\mathbb{D}(\omega)).
    \]
    Therefore, thanks to \cref{thm:main_paper1}, we may assume that 
    \[
    \mathbb{D}=\mathbb{D}(T),\quad \mathbb{D}'=\mathbb{D}(T')
    \]
    for $\mathcal{I}$-good non-divisorial Kähler currents $T$ and $T'$ on $X$. Observe that $T\preceq_{\mathcal{I}}T'$ as a consequence of \cite[Corollary~4.1.4]{Xiabdiv}. But then, thanks to \cite[Proposition~8.2.1]{Xiabook}, we have
    \[
    \Tr_D T\preceq_{P}\Tr_D T'.
    \]
    Then by \cite[Corollary~4.1.4]{Xiabdiv} again, our assertion follows.
\end{proof}

\begin{theorem}\label{thm:trace_cont}
    Let $\alpha\in \mathrm{H}^{1,1}(X,\mathbb{R})$ be a big class. Assume that $(T_i)_{i\in I}$ is a net of closed positive $(1,1)$-currents in $\alpha$, decreasing with respect to the $\mathcal{I}$-partial order. Assume that $T_i\xrightarrow{d_S} T$ for some closed positive $(1,1)$-current $T\in \alpha$ satisfying $\nu(T,\alpha)=0$.
    Then
    \[
    \Tr_D T_i\xrightarrow{d_S} \Tr_D T.
    \]
    Furthermore,
    \begin{equation}\label{eq:resvol_cont1}
    \lim_{i\in I} \vol_{X|D}(T_i)=\vol_{X|D}(T).
    \end{equation}
\end{theorem}
\begin{proof}
    The former statement is proved in \cite[Proposition~8.2.2]{Xiabook}. If $\vol_{X|D}(T)>0$, then $\Tr_D T_i$ and $\Tr_D T$ can both be represented by currents in $\alpha|_{\widetilde{D}}$ with positive masses. In this case, \eqref{eq:resvol_cont1} follows from the $d_S$-continuity of the volume. See \cite[Theorem~6.2.5]{Xiabook}.

    If $\vol_{X|D}(T)=0$, fix a Kähler form $\omega$ on $X$. Fix $\epsilon>0$, we can find $\delta>0$ so that
    \[
    \vol_{X|D}(T+\delta\omega)\leq \epsilon.
    \]
    In this case, applying \eqref{eq:resvol_cont1} to the net $(T_i+\delta\omega)_i$, in view of \cite[Corollary~6.2.8]{Xiabook} we find that
    \[
    \varlimsup_{i\in I} \vol_{X|D}(T_i)\leq \lim_{i\in I} \vol_{X|D}(T_i+\delta\omega)=\vol_{X|D}(T+\delta\omega)\leq \epsilon.
    \]
    Since $\epsilon>0$ is arbitrary, we conclude that
    \[
    \lim_{i\in I} \vol_{X|D}(T_i)=0.
    \]
\end{proof}

\begin{corollary}\label{cor:conttrace}
    Let $(\mathbb{D}^i)_{i\in I}$ be a decreasing net of nef b-divisors over $X$ with limit $\mathbb{D}$. Assume that $\mathbb{D}^i_X$ is independent of the choice of $i\in I$. Then 
    \begin{equation}\label{eq:limTrDi}
        \lim_{i\in I} \Tr_D \mathbb{D}^i=\Tr_D \mathbb{D}. 
    \end{equation}
\end{corollary}
Note that thanks to \cref{prop:trace_basic}(1), $(\Tr_D\mathbb{D}^i)_i$ is decreasing, and hence the limit in \eqref{eq:limTrDi} exists.
\begin{proof}
    Fix a Kähler form $\omega$ on $X$. Thanks to the additivity of the trace operator proved in \cref{prop:trace_basic},
    it suffices to prove \eqref{eq:limTrDi} with $\mathbb{D}^i+\mathbb{D}(\omega)$ and $\mathbb{D}+\mathbb{D}(\omega)$ in place of $\mathbb{D}^i$ and $\mathbb{D}$. So by \cref{thm:main_paper1}, we may assume that there are non-divisorial $\mathcal{I}$-good Kähler currents $T_i,T\in \mathbb{D}_X$ with 
    \[
    \mathbb{D}_i=\mathbb{D}(T_i),\quad \mathbb{D}=\mathbb{D}(T)
    \]
    for all $i\in I$.
    
    It follows from \cite[Corollary~4.1.4]{Xiabdiv} and the $\mathcal{I}$-goodness of the $T_i$'s and $T$ that $(T_i)_i$ is decreasing with respect to the $P$-partial order and $T\preceq_P T_i$. 

    Since $T$ has positive mass, it follows from \cite[Corollary~6.2.6]{Xiabook} that $(T_i)_i$ admits a $d_S$-limit $S$ in $\{T\}$, then $\mathbb{D}(S)=\mathbb{D}(T)$. It follows from \cite[Corollary~4.1.4]{Xiabdiv} that $S\sim_P T$. Hence $T_i\xrightarrow{d_S} T$. Our assertion then follows from \cref{thm:trace_cont}. 
\end{proof}

Let us make sense of the trace operator in some simple cases.
\begin{example}\label{ex:trace_nef}
    Let $\alpha\in \mathrm{H}^{1,1}(X,\mathbb{R})$ be a nef class. Then 
    \[
    \Tr_D \mathbb{D}(\alpha)= \mathbb{D}\left(\alpha|_D\right).
    \]
    To see this, we may assume that $\alpha$ is a Kähler class. Take a Kähler form $\omega\in \alpha$, then by definition,
    \[
    \Tr_D \mathbb{D}(\alpha)=\mathbb{D}(\omega|_D)=\mathbb{D}\left(\alpha|_D\right).
    \]
\end{example}
\begin{example}\label{ex:trace_psef}
    Let $\alpha\in \mathrm{H}^{1,1}(X,\mathbb{R})$ be a big class. Assume that $D$ is smooth and not contained in the non-Kähler locus of $\alpha$. 
    Then
    \begin{equation}\label{eq:trace_psef_class}
    \Tr_D \mathbb{D}(\alpha)=\mathbb{D}\left(\Tr_D T_{\min}\right)=\mathbb{D}\left(T_{\min}|_D\right),
    \end{equation}
    where $T_{\min}\in \alpha$ is a current with minimal singularities.

    It follows from \cite[Proposition~8.3.1]{Xiabook} that $\Tr_D T_{\min}$ is represented by $T_{\min}|_D$. Therefore, we can apply \eqref{eq:TrDmathbbD} to conclude \eqref{eq:trace_psef_class}.
\end{example}

\begin{proposition}\label{prop:tracepsef_concave}
Let $\alpha,\beta\in \mathrm{H}^{1,1}(X,\mathbb{R})$ be modified nef classes.
Then
    \begin{equation}
    \Tr_D\mathbb{D}(\alpha)+\Tr_D\mathbb{D}(\beta)\leq \Tr_D\mathbb{D}(\alpha+\beta).
    \end{equation}
\end{proposition}
\begin{proof}
Since $\alpha+\beta$ is also modified nef, we have
\[
\mathbb{D}(\alpha)_X+\mathbb{D}(\beta)_X=\alpha+\beta= \mathbb{D}(\alpha+\beta)_X.
\]
Thanks to \cref{prop:trace_basic}, it suffices to show that
\[
\mathbb{D}(\alpha)+\mathbb{D}(\beta)\leq \mathbb{D}(\alpha+\beta),
\]
which is just \cref{prop:Dsupadd}.
\end{proof}

Next recall the notion of restricted volumes of a cohomology class studied in \cite{CT22, Mat13}. 
Fix a pseudoeffective class $\alpha\in \mathrm{H}^{1,1}(X,\mathbb{R})$. If $\alpha$ is big and $D$ is not contained in the non-Kähler locus of $X$, then we set
\[
\begin{aligned}
\vol_{X|D}(\alpha)\coloneqq \sup & \left\{ \int_D T|_D^{n-1} :  T\in \alpha \textup{ is a Kähler current with analytic singularities, } \right.\\
&\left.\nu(T,D)=0\right\}.
\end{aligned}
\]
Here $\int_D T|_D^{n-1}$ could be understood as the non-pluripolar product $\int_{\tilde{D}}T|_{\tilde{D}}^{n-1}$ on the normalization $\tilde{D}$ of $D$.

In general, if $\alpha$ is pseudoeffective and $D$ is not contained in the non-nef locus of $\alpha$, we take a Kähler form $\omega$ on $X$ and set
\[
\vol_{X|D}(\alpha)\coloneqq \lim_{\epsilon\to 0+}\vol_{X|D}(\alpha+\epsilon\{\omega\}).
\]
If $D$ is contained in the non-nef locus of $\alpha$, we set
\[
\vol_{X|D}(\alpha)\coloneqq 0.
\]

The restricted volume can be expressed as an intersection number of b-divisors.
\begin{theorem}\label{thm:coinc_class}
    Let $D$ be a prime divisor of $X$. Given a big class $\alpha\in \mathrm{H}^{1,1}(X,\mathbb{R})$ with $\nu(\alpha,D)=0$, then
    \begin{equation}\label{eq:res_vol_eq_volbd}
    \vol_{X|D}(\alpha)= \vol \Tr_D\mathbb{D}(\alpha).
    \end{equation}
\end{theorem}
From \eqref{eq:res_vol_eq_volbd}, we deduce that $\vol_{X|D}(\alpha)=\vol_{X|D}(\langle \alpha \rangle)$. This is also proved in \cite[Lemma~6.2]{CT22}.

Restricted volumes to higher codimensional subspaces admit similar expressions. Since we are not in need of them in this paper, we omit the details.
\begin{proof}
\textbf{Step~1}. We first assume that $D$ is not contained in the non-Kähler locus of $\alpha$.

Let $\pi\colon Y\rightarrow X$ be a modification so that the strict transform $D'$ of $D$ is smooth. Then by \cite[Lemma~4.2, Lemma~4.3]{CT22},
\[
\vol_{X|D}(\alpha)=\vol_{Y|D'}(\pi^*\alpha)
\]
and $D'$ is not in the non-Kähler locus of $\pi^*\alpha$. On the other hand, 
\[
\Tr_D\mathbb{D}(\alpha)=\Tr_{D'}\mathbb{D}(\pi^*\alpha)
\]
as a consequence of \cref{prop:trace_indepmodel}. So
\[
\Bigl( \Tr_D\mathbb{D}(\alpha) \Bigr)^{n-1}=\Bigl( \Tr_{D'}\mathbb{D}(\pi^*\alpha) \Bigr)^{n-1}.
\]
Therefore, we may further assume that $D$ is smooth. 
In this case our assertion follows from \cite[Proposition~8.3.1]{Xiabook}.

\textbf{Step~2}. Next we assume that $D$ is contained in the non-Kähler locus of $\alpha$. 

Fix a Kähler form $\omega$, then for any $\epsilon>0$, we have
\[
\Tr_D \mathbb{D}(\alpha+\epsilon\{\omega\})\geq \Tr_D \mathbb{D}(\alpha)
\]
by \cref{prop:tracepsef_concave}.
Therefore, our Step~1 and \cref{prop:bdiv_int_prop}(4) imply that
\[
\vol_{X|D}(\alpha)=\lim_{\epsilon\to 0+}\vol_{X|D}(\alpha+\epsilon\{\omega\})=\lim_{\epsilon\to 0+}\Bigl(\Tr_D \mathbb{D}(\alpha+\epsilon\{\omega\}) \Bigr)^{n-1}\geq \Bigl(\Tr_D \mathbb{D}(\alpha)\Bigr)^{n-1}. 
\]
Due to \cite[Theorem~1.1]{Vu23der}, $\vol_{X|D}(\alpha)=0$, so our assertion follows.
\end{proof}

\subsection{The restricted volume}
Let $D$ be a prime divisor over $X$. Motivated by \cref{thm:coinc_class}, we introduce the following definition:
\begin{definition}\label{def:resv}
    Let $\mathbb{D}_1,\ldots,\mathbb{D}_p$ be nef b-divisors over $X$.
    Let $\pi\colon Y\rightarrow X$ be a modification so that $D$ is a prime divisor on $Y$. 
    Then we define the \emph{restricted volume} of $\mathbb{D}_1,\ldots,\mathbb{D}_p$ to $D$ as
    \[
    \vol_{X|D} \left(\mathbb{D}_1,\ldots,\mathbb{D}_p \right)\coloneqq \pi_*\left( \Tr_D \mathbb{D}_1\cap \cdots \cap \Tr_D \mathbb{D}_p \right)_{Y}\in \mathrm{H}^{p+1,p+1}(X,\mathbb{R}).
    \]
\end{definition}
The notation 
\begin{equation}\label{eq:traceprod_explan}
\left(\Tr_D \mathbb{D}_1 \cap \dots \cap \Tr_D \mathbb{D}_p \right)_{Y}
\end{equation}
requires some explanation. The traces $\Tr_D \mathbb{D}_i$'s are nef b-divisors over $D$. For each resolution $W\rightarrow D$, we can then regard $\Tr_D \mathbb{D}_i$ as a nef b-divisor over $W$. The intersection $\Tr_D \mathbb{D}_1 \cap \dots \cap \Tr_D \mathbb{D}_p$ then makes sense as a b-class in $\bDiv^p(W)$. The pushforward of 
\[
\left(\Tr_D \mathbb{D}_1 \cap \dots \cap \Tr_D \mathbb{D}_p \right)_{W}
\]
with respect to $W \rightarrow D\hookrightarrow Y$ is then a well-defined positive class in $\mathrm{H}^{p+1,p+1}(Y,\mathbb{R})$, and it is independent of the choice of $W$. We denote this class by \eqref{eq:traceprod_explan}.

When $p=0$, $\vol_{X|D}(-)$ is $\{D\}$ if $D$ is a prime divisor on $X$ and $0$ if $D$ is exceptional over $X$.

\begin{lemma}\label{lma:resvol_indepres}
    The quantity $\vol_{X|D} \left(\mathbb{D}_1,\ldots,\mathbb{D}_p \right)$ defined in \cref{def:resv} is independent of the choice of $Y$.
\end{lemma}
\begin{proof}
    Take a different modification $\Pi\colon Z\rightarrow X$ so that $D$ is a prime divisor on $Z$. We want to show that the quantity $\vol_{X|D} \left(\mathbb{D}_1,\ldots,\mathbb{D}_p \right)$ defined with respect to $\pi$ and $\Pi$ are the same. For this purpose, we may assume that $\Pi$ dominates $\pi$ through a morphism $p\colon Z\rightarrow Y$. 
    Let $D'$ be the strict transform of $D$ on $Z$.     
    Then we need to show that
    \[
    \pi_*\left( \Tr_D \mathbb{D}_1\cap \cdots \cap \Tr_D \mathbb{D}_p \right)_{Y}=\Pi_* \left( \Tr_{D'} \mathbb{D}_1\cap \cdots \cap \Tr_{D'} \mathbb{D}_p \right)_{Z}.
    \]
    For this purpose, it suffices to show that 
    \[
    \left(\Tr_D \mathbb{D}_1\cap \cdots \cap \Tr_D \mathbb{D}_p \right)_{Y}=p_* \left( \Tr_{D'} \mathbb{D}_1\cap \cdots \cap \Tr_{D'} \mathbb{D}_p \right)_{Z}.
    \]
    Our assertion then follows from \cref{prop:trace_indepmodel} and \cref{prop:int_indp_model}.
\end{proof}

We first observe how this quantity behaves with respect to modifications:
\begin{proposition}\label{prop:trace_bimeo}
    Let $\pi\colon Y\rightarrow X$ be a modification. Then
    \[
    \vol_{X|D} \left(\mathbb{D}_1,\ldots,\mathbb{D}_p \right)=\pi_*\vol_{Y|D} \left(\mathbb{D}_1,\ldots,\mathbb{D}_p \right).
    \]
    Here on the right-hand side, we regard the $\mathbb{D}_i$'s are nef b-divisors over $Y$.
\end{proposition}
In particular, we can regard
\[
\left( \vol_{Y|D} \left(\mathbb{D}_1,\ldots,\mathbb{D}_p \right) \right)_{\pi\colon Y\rightarrow X}
\]
as an element in $\bDiv^{p+1}(X)$. We denote it by $\vol_{|D} \left(\mathbb{D}_1,\ldots,\mathbb{D}_p \right)$. 
\begin{proof}
Take a modification $\Pi\colon Z\rightarrow Y$ so that $D$ is a prime divisor on $Z$. Then thanks to \cref{lma:resvol_indepres} we only need to show that
\[
\pi_* \Pi_*\left( \Tr_D \mathbb{D}_1\cap \cdots \cap \Tr_D \mathbb{D}_p \right)_Z=(\pi\circ \Pi)_* \left( \Tr_D \mathbb{D}_1\cap \cdots \cap \Tr_D \mathbb{D}_p \right)_Z,
\]
which is obvious.
\end{proof}

\begin{proposition}\label{prop:volres_capnef}
    Let $\mathbb{D}_1,\ldots,\mathbb{D}_p$ be nef b-divisors over $X$. Consider a nef class $\alpha\in \mathrm{H}^{1,1}(X,\mathbb{R})$. Then
    \[
    \vol_{X|D}\Bigl(\mathbb{D}_1,\ldots,\mathbb{D}_p,\mathbb{D}(\alpha)\Bigr)=\vol_{X|D}\left(\mathbb{D}_1,\ldots,\mathbb{D}_p\right)\cap \alpha.
    \]
\end{proposition}
\begin{proof}
    By \cref{prop:trace_bimeo}, we may replace $X$ by a modification and assume that $D$ is a smooth prime divisor on $X$. Then by projection formula and \cref{ex:trace_nef} our assertion means
    \[
    \Bigl( \Tr_D \mathbb{D}_1\cap \cdots \cap  \Tr_D \mathbb{D}_p\cap \mathbb{D}(\alpha|_D)\Bigr)_X=\left( \Tr_D \mathbb{D}_1\cap \cdots \cap  \Tr_D \mathbb{D}_p\right)_X\cap \alpha
    |_D,
    \]
    which follows readily from \cref{prop:bdiv_int_prop}(5).
\end{proof}

The basic properties of the trace operator imply the following properties of the restricted volume:
\begin{proposition}\label{prop:res_vol}
    Let $\mathbb{D}_1,\ldots,\mathbb{D}_p,\mathbb{D}'_1$ be nef b-divisors over $X$, and $D$ be a prime divisor over $X$. Consider $\lambda\geq 0 $. Then we have the following properties:
    \begin{enumerate}
        \item The restricted volume is symmetric: Let $\sigma$ be a permutation of $\{1,\ldots,p\}$, then
        \[
        \vol_{|D} \left(\mathbb{D}_{\sigma(1)},\ldots,\mathbb{D}_{\sigma(p)} \right)=\vol_{|D} \left(\mathbb{D}_1,\ldots,\mathbb{D}_p \right).
        \]
        \item The restricted volume is additive in each variable:
        \[
        \vol_{|D} \left(\mathbb{D}_1+\mathbb{D}'_1,\mathbb{D}_2,\ldots,\mathbb{D}_p \right)=\vol_{|D} \left(\mathbb{D}_1,\mathbb{D}_2,\ldots,\mathbb{D}_p \right)+\vol_{|D} \left(\mathbb{D}'_1,\mathbb{D}_2,\ldots,\mathbb{D}_p \right).
        \]
        \item The restricted volume is homogeneous in each variable:
        \[
        \vol_{|D} \left(\lambda\mathbb{D}_1,\mathbb{D}_2,\ldots,\mathbb{D}_p \right)=\lambda \vol_{|D} \left(\mathbb{D}_1,\ldots,\mathbb{D}_p \right).
        \]
        \item The restricted volume is increasing in each variable in the following sense: If $\mathbb{D}_1'\geq \mathbb{D}_1$, $\mathbb{D}_{1,X}'=\mathbb{D}_{1,X}$, and $D\subseteq X$, then
        \[
        \vol_{X|D} \left(\mathbb{D}'_1,\ldots,\mathbb{D}_p \right)\geq \vol_{X|D} \left(\mathbb{D}_1,\ldots,\mathbb{D}_p \right). 
        \]
    \end{enumerate}
\end{proposition}
\begin{proof}
    By \cref{prop:trace_bimeo}, we may replace $X$ by a modification and assume that $D$ is a smooth prime divisor on $X$.

    (1) This follows immediately from \cref{prop:bdiv_int_prop}(1).

    (2) This follows immediately from \cref{prop:bdiv_int_prop}(2) and \cref{prop:trace_basic}(2).

    (3) This follows immediately from \cref{prop:bdiv_int_prop}(3) and \cref{prop:trace_basic}(3).

    (4) This follows from \cref{prop:trace_basic}(1) and \cref{prop:bdiv_int_prop}(4).
\end{proof}

\begin{proposition}\label{prop:contresvol}
    Let $(\mathbb{D}_i^j)_{j\in J}$ be decreasing nets of nef b-divisors over $X$ ($i=1,\ldots,p$). Assume that for each $i=1,\ldots,p$, the class  $\mathbb{D}_{i,X}^j$ does not depend on the choice of $j\in J$.
    Denote the limit of $(\mathbb{D}_i^j)_j$ by $\mathbb{D}_i$. Consider a prime divisor $D$ on $X$. Then we have a decreasing limit
    \[
        \lim_{j\in J}\vol_{X|D}\left(\mathbb{D}_1^j,\ldots,\mathbb{D}_p^j \right)=\vol_{X|D}\left(\mathbb{D}_1,\ldots,\mathbb{D}_p \right).
    \]
\end{proposition}
\begin{proof}
    It follows from \cref{cor:conttrace} and \cref{prop:trace_basic}(1) that for each $i=1,\ldots,p$, the net $(\Tr_D \mathbb{D}_i^j)_j$ is decreasing with limit $\Tr_D \mathbb{D}_i$. Our assertion then follows from \cref{thm:bdiv_int_cont1} and \cref{prop:res_vol}(4).
\end{proof}

The restricted volume is easy to understood in some special cases.
\begin{example}\label{ex:res_vol_nef}
    When $\mathbb{D}_i=\mathbb{D}(\alpha_i)$ for some nef classes $\alpha_i\in \mathrm{H}^{1,1}(X,\mathbb{R})$ for all $i=1,\ldots,p$, we have
    \[
    \vol_{X|D}\left(\mathbb{D}_1,\ldots,\mathbb{D}_p\right)=
    \begin{cases}
        0,& \textup{if }D\textup{ is exceptional};\\
        \alpha_1\cap \cdots \cap \alpha_p\cap \{D\},&\textup{otherwise}.
    \end{cases}
    \]
    To see this, let $\pi\colon Y\rightarrow X$ be a modification so that $D$ is a smooth prime divisor on $Y$. Then by definition and \cref{ex:trace_nef},
    \[
    \begin{aligned}
    \vol_{X|D}\left(\mathbb{D}_1,\ldots,\mathbb{D}_p\right)=&\pi_* \left(\Tr_D\mathbb{D}_1 \cap \cdots \cap \Tr_D \mathbb{D}_p \right)_Y\\
    =&\pi_* \Bigl( \mathbb{D}(\pi^*\alpha_1|_D)\cap \cdots \cap \mathbb{D}(\pi^*\alpha_p|_D) \Bigr)_Y\\
    =&\pi_* \left( \pi^*\alpha_1\cap \cdots \cap \pi^*\alpha_p\cap \{D\} \right)\\
    =& \alpha_1\cap \cdots \cap \alpha_p\cap \pi_*\{D\}.
    \end{aligned}
    \]
    Our assertion follows.
\end{example}

\begin{example}\label{ex:volrestDalphavol}
    Assume that $D$ is a prime divisor on $X$. Let $\alpha\in \mathrm{H}^{1,1}(X,\mathbb{R})$ be a big class with $\nu(\alpha,D)=0$. Then 
    \[
    \vol_{X|D}\left(\mathbb{D}(\alpha)^{n-1}\right)=\vol_{X|D}(\alpha).
    \]
    This is what we proved in \cref{thm:coinc_class}.
\end{example}
Motivated by this example, we introduce the following definition:
\begin{definition}\label{def:vol_mix_class}
Assume that $D$ is a prime divisor on $X$.
    Let $\alpha_1,\ldots,\alpha_p\in \mathrm{H}^{1,1}(X,\mathbb{R})$ be pseudoeffective classes with $\nu(\alpha_i,D)=0$ for all $i=1,\ldots,p$. Fix a Kähler class $\beta\in \mathrm{H}^{1,1}(X,\mathbb{R})$.
    We define
    \begin{equation}\label{eq:mix_res_vol_class}
    \vol_{X|D}(\alpha_1,\ldots,\alpha_p)\coloneqq \lim_{\epsilon\to 0+}\vol_{X|D}\Bigl(\mathbb{D}(\alpha_1+\epsilon\beta),\ldots,\mathbb{D}(\alpha_p+\epsilon\beta) \Bigr)\in \mathrm{H}^{p+1,p+1}(X,\mathbb{R}).
    \end{equation}
    If $\nu(\alpha_i,D)>0$ for some $i\in \{1,\ldots,p\}$, we let
    \[
    \vol_{X|D}(\alpha_1,\ldots,\alpha_p)\coloneqq 0.
    \]
\end{definition}
In particular, 
\[
\vol_{X|D}\Bigl(\underbrace{\alpha,\ldots,\alpha}_{n-1}\Bigr)=\vol_{X|D}(\alpha)
\]
for any pseudoeffective class $\alpha\in \mathrm{H}^{1,1}(X,\mathbb{R})$.

We first observe that the restricted volume is well-defined:
\begin{lemma}
    The limit in \eqref{eq:mix_res_vol_class} exists and is independent of the choice of $\beta$. 
\end{lemma}
\begin{proof}
    Since $\nu(\alpha_i,D)=0$, for any $\epsilon'>\epsilon>0$, we know that $D$ is not in the non-Kähler locus of $\alpha_i+\epsilon\beta$ for each $i$. We wish to show that 
    \begin{equation}\label{eq:volXD_temp11}
    \vol_{X|D}\Bigl(\mathbb{D}(\alpha_1+\epsilon\beta),\ldots,\mathbb{D}(\alpha_p+\epsilon\beta) \Bigr)\leq_X \vol_{X|D}\Bigl(\mathbb{D}(\alpha_1+\epsilon'\beta),\ldots,\mathbb{D}(\alpha_p+\epsilon'\beta) \Bigr).    
    \end{equation}
    Let $T_i^{\epsilon}$ be a current with minimal singularities in $\alpha_i+\epsilon\beta$ for each $\epsilon>0$ and $i=1,\ldots,p$. Then
    \[
    \left\{ \bigwedge_{i=1}^p T_i^{\epsilon}|_{\tilde{D}} \right\} \leq_X \left\{ \bigwedge_{i=1}^p T_i^{\epsilon}|_{\tilde{D}}+(\epsilon'-\epsilon)\omega|_{\tilde{D}} \right\}\leq_X \left\{ \bigwedge_{i=1}^p T_i^{\epsilon'}|_{\tilde{D}} \right\},
    \]
    where we have applied \cref{thm:mono_thm_partial} and \cref{prop:npp_strongpos}. Here as usual, we have omitted the obvious pushforwards.
    
    Note that the argument of \cite[Proposition~8.3.1]{Xiabook} shows that 
    \[
    \Tr_D \mathbb{D}(\alpha_i+\epsilon\beta)=\mathbb{D}\left(T_i^{\epsilon}|_{\tilde{D}}\right),
    \]
    where $T_i$ is a current with minimal singularities in $\alpha_i+\epsilon\beta$. Therefore, \eqref{eq:volXD_temp11} follows.

    Similarly, one can show that  \eqref{eq:mix_res_vol_class} is independent of the choice of $\beta$.
\end{proof}

The basic properties of this mixed restricted volume are summarized as below: 
\begin{proposition}\label{prop:res_vol_class}
    Assume that $D$ is a prime divisor on $X$.
    Let $\alpha_1,\alpha'_1,\ldots,\alpha_p\in \mathrm{H}^{1,1}(X,\mathbb{R})$ be pseudoeffective classes, $\lambda\geq 0$. Then we have the following:
    \begin{enumerate}
        \item The restricted volume is symmetric: For any permutation $\sigma$ of $\{1,\ldots,p\}$, we have
        \[
        \vol_{X|D}\left(\alpha_{\sigma(1)},\ldots,\alpha_{\sigma(p)} \right)=\vol_{X|D}\left(\alpha_1,\ldots,\alpha_p \right).
        \]
        \item The restricted volume is homogeneous in each variable:
        \[
        \vol_{X|D}\left(\lambda\alpha_1,\alpha_2,\ldots,\alpha_p \right)=\lambda\vol_{X|D}\left(\alpha_1,\alpha_2,\ldots,\alpha_p \right).
        \]
        \item The restricted volume is concave in each variable in the following sense: Suppose that $\nu(\alpha_1,D)=\nu(\alpha'_1,D)=0$, then
        \[
        \vol_{X|D}\left(\alpha_1,\alpha_2,\ldots,\alpha_p \right)+\vol_{X|D}\left(\alpha'_1,\alpha_2,\ldots,\alpha_p \right)\leq_X \vol_{X|D}\left(\alpha_1+\alpha'_1,\alpha_2,\ldots,\alpha_p \right).
        \]
        If $D$ is not in the non-Kähler loci of the $\alpha_i$'s, then
        \[
        \vol_{X|D}\left(\alpha_1,\ldots,\alpha_p \right)=\vol_{X|D}\Bigl(\mathbb{D}(\alpha_1),\ldots,\mathbb{D}(\alpha_p) \Bigr).
        \]
        \item The restricted volume is upper semi-continuous: Suppose that $(\alpha_i^j)_{j>0}$ are sequences of pseudoeffective classes in $\mathrm{H}^{1,1}(X,\mathbb{R})$ with limits $\alpha_i\in \mathrm{H}^{1,1}(X,\mathbb{R})$. Then for any weakly positive $(n-p-1,n-p-1)$-form $F$ on $X$, we have
        \begin{equation}\label{eq:mix_res_vol_class_usc}
        \varlimsup_{j\to\infty} \vol_{X|D}\left(\alpha_1^j,\ldots,\alpha_p^j \right) \cap \{F\}\leq \vol_{X|D}\left(\alpha_1,\ldots,\alpha_p \right) \cap \{F\}.
        \end{equation}
        \item Assume that $\nu(\alpha_i,D)=0$ for all $i=1,\ldots,p$. Then
        \begin{equation}\label{eq:mix_res_vol_class_inv_dzd}
            \vol_{X|D}\left(\alpha_1,\ldots,\alpha_p \right)=\vol_{X|D}\Bigl(\langle \alpha_1 \rangle,\ldots, \langle \alpha_p \rangle \Bigr).
        \end{equation}
    \end{enumerate}
\end{proposition}
The restricted volume is not continuous in general, as shown in \cite[Example~5.6]{CT22}.
\begin{proof}
(1) (2) These are trivial.

(3) We may assume that $\nu(\alpha_i,D)=0$ for $i=2,\ldots,p$ since otherwise there is nothing to prove. 

Let $\mathcal{A}\subseteq \mathrm{H}^{1,1}(X,\mathbb{R})$ be the set of big classes whose non-Kähler loci do not contain $D$. Then $\mathcal{A}$ is open and convex, as proved in \cite[Proposition~4.10]{Mat13}. It suffices to show that for $\alpha_1,\alpha'_1,\ldots,\alpha_p\in \mathcal{A}$, we have
\begin{equation}\label{eq:res_vol_class_con_temp1}
\begin{split}
& \vol_{X|D}\Bigl(\mathbb{D}(\alpha_1),\mathbb{D}(\alpha_2),\ldots,\mathbb{D}(\alpha_p) \Bigr)+\vol_{X|D}\Bigl(\mathbb{D}(\alpha'_1),\mathbb{D}(\alpha_2),\ldots,\mathbb{D}(\alpha_p) \Bigr)\\
\leq_X & \vol_{X|D}\Bigl(\mathbb{D}(\alpha_1+\alpha'_1),\mathbb{D}(\alpha_2),\ldots,\mathbb{D}(\alpha_p) \Bigr).
\end{split}
\end{equation}
Let $T_1,\ldots,T_p,T_1',T$ be currents with minimal singularities in $\alpha_1,\ldots,\alpha_p,\alpha'_1,\alpha_1+\alpha'_1$. Then due to \cref{thm:mono_thm_partial}, we have 
\[
    \left\{ T_1|_{\tilde{D}}\wedge T_2|_{\tilde{D}}\wedge \cdots \wedge T_{p}|_{\tilde{D}} \right\}+\left\{ T'_1|_{\tilde{D}}\wedge T_2|_{\tilde{D}}\wedge \cdots \wedge T_{p}|_{\tilde{D}} \right\} \leq_X \left\{ T|_{\tilde{D}}\wedge T_2|_{\tilde{D}}\wedge \cdots \wedge T_{n-p}|_{\tilde{D}}\right\},
\]
and \eqref{eq:res_vol_class_con_temp1} follows.

(4) We may assume without loss of generality that $\nu(\alpha_i,D)=0$ for each $i=1,\ldots,p$ since there is nothing to prove otherwise. Further, we may assume that $\nu(\alpha_i^j,D)=0$ for all $i=1,\ldots,p$ and all $j>0$.

Fix a Kähler class $\gamma\in \mathrm{H}^{1,1}(X,\mathbb{R})$. Then for any $\epsilon>0$, we can find $j_0>0$ so that when $j\geq j_0$, the classes $\alpha_i+\epsilon \gamma-\alpha_i^j$'s are all Kähler, where $i=1,\ldots,p$. Then due to (3), we find
\[
\vol_{X|D}\left( \alpha_1^j,\ldots,\alpha_p^j \right)\leq_X \vol_{X|D}\left( \alpha_1+\epsilon\gamma,\ldots,\alpha_p+\epsilon\gamma \right).
\]
Therefore, 
\[
\varlimsup_{j\to\infty}\vol_{X|D}\left( \alpha_1^j,\ldots,\alpha_p^j \right) \cap \{F\} \leq \vol_{X|D}\left( \alpha_1+\epsilon\gamma,\ldots,\alpha_p+\epsilon\gamma \right) \cap \{F\}.
\]
Letting $\epsilon \to 0+$, we conclude \eqref{eq:mix_res_vol_class_usc}.

(5) We first assume that the $\alpha_i$'s are all big and $D$ is not in the non-Kähler loci of the $\alpha_i$'s.

In this case, $D$ is also not in the non-Kähler loci of the $\langle \alpha_i \rangle$'s, as shown in \cite[Lemma~6.1]{CT22}.
In this case, let $T_1,\ldots,T_{p}$ be currents with minimal singularities in $\alpha_1,\ldots,\alpha_p$. Then
\[
T_i-\sum_{E\subseteq X} \nu(\alpha_i,E)[E]
\]
has minimal singularities in $\langle \alpha_i \rangle$. Therefore, \eqref{eq:mix_res_vol_class_inv_dzd} translates into
\[
\left\{ \bigwedge_{i=1}^p T_i|_{\tilde{D}} \right\}=\left\{ \bigwedge_{i=1}^p \left.\left( T_i-\sum_{E\subseteq X} \nu(\alpha_i,E)[E]\right)\right|_{\tilde{D}} \right\},
\]
which is obvious.

Next we handle the general case. Fix a Kähler class $\gamma\in \mathrm{H}^{1,1}(X,\mathbb{R})$. We first observe that the $\geq_X$ direction in  \eqref{eq:mix_res_vol_class_inv_dzd} follows readily from (3). It suffices to prove the reverse inequality. From what we have proved, we know that for any $\epsilon>0$, we have
\begin{equation}\label{eq:mix_res_vol_class_temp1}
 \vol_{X|D}\left(\alpha_1+\epsilon\gamma,\ldots,\alpha_p+\epsilon\gamma \right)=\vol_{X|D}\Bigl(\langle \alpha_1+\epsilon\gamma \rangle,\ldots, \langle \alpha_p+\epsilon\gamma \rangle \Bigr).
\end{equation}
From \cref{prop:mov_int_prop}, for each $i=1,\ldots,p$, we have
\[
\lim_{\epsilon\to 0+}\langle \alpha_i+\epsilon\gamma \rangle=\langle \alpha_i\rangle,
\]
so letting $\epsilon\to 0+$ in \eqref{eq:mix_res_vol_class_temp1} and applying (4), the $\leq_X$ direction in \eqref{eq:mix_res_vol_class_inv_dzd} follows.
\end{proof}

\begin{example}\label{ex:resvol_DTT}
    Assume that $D$ is a prime divisor on $X$. Let $T$ be a closed positive $(1,1)$-current on $X$ with $\nu(T,D)=0$. Then
    \begin{equation}\label{eq:volres_DTT}
    \vol_{X|D} \left(\mathbb{D}(T)^{n-1}\right)=\vol_{X|D}(T).
    \end{equation}
    To prove this, we first assume that $T$ is a Kähler current. Take a representative $\Tr_D T\in \{T\}|_D$. Then our assertion amounts to 
    \[
    \mathbb{D}(\Tr_D T)^{n-1}=\vol(\Tr_D T).
    \]
    This is part of \cref{thm:main_paper1}.

    Now let us come back to the general case.
    Take a Kähler form $\omega$ on $X$. Then we know that
    \[
    \vol_{X|D} \biggl(\Bigl(\mathbb{D}(T)+\epsilon \mathbb{D}(\omega)\Bigr)^{n-1}\biggr)=\vol_{X|D}(T+\epsilon \omega).
    \]
    Thanks to \cref{prop:res_vol} and \eqref{eq:resvolcurr}, when $\epsilon\to 0+$, the limit gives \eqref{eq:volres_DTT}.
\end{example}

Motivated by the above example, we introduce the following:
\begin{definition}\label{def:mix_res_vol_curr}
    Suppose that $T_1,\ldots,T_p$ are closed positive $(1,1)$-currents on $X$, $D$ is a prime divisor on $X$. 
    Assume that $\nu(T_i,D)=0$ for all $i$, then we define
    \[
    \vol_{X|D}(T_1,\ldots,T_p)\coloneqq \vol_{X|D}\Bigl( \mathbb{D}(T_1),\ldots,\mathbb{D}(T_p)\Bigr)\in \mathrm{H}^{p+1,p+1}(X,\mathbb{R}).
    \]
    If $\nu(T_i,D)>0$ for some $i$, we simply set
    \[
    \vol_{X|D}(T_1,\ldots,T_p)\coloneqq 0.
    \]
\end{definition}

We need the following Brunn--Minkowski inequality for the sequel.
\begin{proposition}\label{prop:BM}
Assume that $n>1$.
    Let $\mathbb{D}_1,\ldots,\mathbb{D}_{n-1}$ be nef b-divisors over $X$ and $D$ is a prime divisor on $X$. Then
    \[
    \vol_{X|D}\left(\mathbb{D}_1,\ldots,\mathbb{D}_{n-1}\right)\geq \prod_{i=1}^{n-1} \left(\vol \Tr_D \mathbb{D}_i\right)^{1/(n-1)}.
    \]
\end{proposition}
\begin{proof}
    This follows from the Brunn--Minkowski inequality proved in \cite[Proposition~5.7]{Xiabdiv}.
\end{proof}

\section{Qualitative monotonicity theorem}\label{sec:qual_mono}
Let $X$ be a connected compact Kähler manifold of dimension $n$. Fix a non-negative integer $p$.

We shall derive a qualitative monotonicity theorem using the theory of b-divisors.

\begin{lemma}\label{lma:Tpwedgediff_analy}
    Let $T_1,\ldots,T_{p}$ be closed positive $(1,1)$-currents with analytic singularities on $X$. Suppose that $T,T'$ are $\mathcal{I}$-good closed positive $(1,1)$-currents on $X$ with positive volumes in the same cohomology class. Assume that $T\preceq T'$. 
    Then
    \begin{equation}\label{eq:T1toTp_diff_1}
    \begin{split}
    &\left\{ T_1\wedge \cdots \wedge T_{p}\wedge T'\right\} - \left\{ T_1\wedge \cdots \wedge T_p\wedge T \right\} \\
    \geq_X & \sum_{D\subseteq X}\Bigl( \nu(T,D)-\nu(T',D) \Bigr)\left\{ \bigwedge_{i=1}^{p}\left.\Bigl( T_i-\nu(T_i,D)[D]\Bigr)\right|_{\widetilde{D}} \right\}.    
    \end{split}
    \end{equation}
    Here $D$ runs over the set of prime divisors on $X$, and $\widetilde{D}\rightarrow D$ is the normalization of $D$.
\end{lemma}
Note that the sum is a countable sum. We have omitted the obvious pushforward maps from $\widetilde{D}$ to $X$.

\begin{proof}

\textbf{Step~1}. We first assume that $T_1,\ldots,T_p$, $T,T'$ all have log singularities. Then we may replace the $T_i$'s by their non-divisorial parts, and hence reduce to the case where $T_1,\ldots,T_p$ all have bounded local potentials. A further regularization then allows us to reduce to the case where $T_1,\ldots,T_p$ are all Kähler forms, say $\omega_1,\ldots,\omega_p$. Then our assertion \eqref{eq:T1toTp_diff_1} means
\[
\begin{split}
& \left\{ \omega_1\wedge \cdots \wedge\omega_p\wedge \Reg T' \right\} - \left\{ \omega_1\wedge \cdots \wedge\omega_p\wedge \Reg T \right\}\\
\geq_X & \sum_{D\subseteq X}\Bigl( \nu(T,D)-\nu(T',D) \Bigr)\left\{ \omega_1|_D\wedge \cdots \wedge \omega_p|_D \right\}. 
\end{split}
\]
This is obvious. We even have equality in this case.

\textbf{Step~2}. We assume that $T$ and $T'$ both have analytic singularities. Replacing $T_i$ by $\Reg T_i$, we may assume that each $T_i$ is non-divisorial.

Let $\pi\colon Y\rightarrow X$ be a modification which resolves the singularities of $T_1,\ldots,T_p$ and $T,T'$.
Then by Step~1 and \cref{prop:npp_strongpos}, we have
\[
\begin{aligned}
& \left\{ \pi^*T_1\wedge \cdots \wedge \pi^*T_{p}\wedge \pi^* T'\right\} - \left\{ \pi^*T_1\wedge \cdots \wedge \pi^* T_p\wedge \pi^*T\right\}\\
\geq_Y & \sum_{E\subseteq Y}\Bigl( \nu(T,E)-\nu(T',E) \Bigr)\left\{ \bigwedge_{i=1}^{p}\left.\Bigl( \pi^*T_i-\nu(T_i,E)[E]\Bigr)\right|_{\tilde{E}} \right\}\\
\geq_Y & \sum_{D\subseteq X}\Bigl( \nu(T,D)-\nu(T',D) \Bigr)\tilde{j}_* \left\{ \bigwedge_{i=1}^{p}\left( \pi^*T_i\right)|_{\widetilde{D'}} \right\}\\
=&\sum_{D\subseteq X}\Bigl( \nu(T,D)-\nu(T',D) \Bigr)\tilde{j}_* \left\{ \bigwedge_{i=1}^{p} \tilde{p}^*\left( T_i|_{\widetilde{D}}\right)\right\}. 
\end{aligned}
\]
Here $D'$ denotes the strict transform of $D$ on $Y$. The notations are summarized in the following commutative diagram:
\[
\begin{tikzcd}
\widetilde{D'} \arrow[r] \arrow[d, "\tilde{p}"] \arrow[rr, "\tilde{j}", bend left=49] & D' \arrow[r, "j", hook] \arrow[d, "p"] & Y \arrow[d, "\pi"] \\
\widetilde{D} \arrow[r] \arrow[rr, "\tilde{i}", bend right=49]                            & D \arrow[r, "i", hook]                 & X.                 
\end{tikzcd}
\]
Taking pushforward and applying \cref{prop:bime_npp}, we find
\[
\begin{aligned}
    & \{T_1\wedge \cdots \wedge T_{p}\wedge T'\}- \{T_1\wedge \cdots \wedge T_p\wedge T\} \\
    \geq_X & \sum_{D\subseteq X}\Bigl( \nu(T,D)-\nu(T',D) \Bigr)\pi_*\tilde{j}_*\left\{ \bigwedge_{i=1}^{p} \tilde{p}^*\left( T_i|_{\widetilde{D}}\right) \right\}\\
    =&\sum_{D\subseteq X}\Bigl( \nu(T,D)-\nu(T',D) \Bigr)\tilde{i}_*\tilde{p}_*\left\{ \bigwedge_{i=1}^{p} \tilde{p}^*\left( T_i|_{\widetilde{D}}\right)\right\}\\
    =&\sum_{D\subseteq X}\Bigl( \nu(T,D)-\nu(T',D) \Bigr)\tilde{i}_*\left\{ \bigwedge_{i=1}^{p} \left( T_i|_{\widetilde{D}}\right) \right\}.
\end{aligned}
\]
The desired inequality \eqref{eq:T1toTp_diff_1} follows.

\textbf{Step~3}. We handle the general case.
    
    Replacing $T$ and $T'$ by
    \[
    T-\sum_{D\subseteq X}\nu(T',D)[D],\quad T'-\sum_{D\subseteq X}\nu(T',D)[D],
    \]
    we may first assume that $T'$ is non-divisorial. We need to show that  
    \begin{equation}\label{eq:wedge_diff_temp2}
    \left\{ T_1\wedge \cdots \wedge T_{p}\wedge T' \right\}- \left\{ T_1\wedge \cdots \wedge T_p\wedge T \right\} \geq_X \sum_{D\subseteq X}\nu(T,D)\left\{ \bigwedge_{i=1}^{p}\left.\Bigl( T_i-\nu(T_i,D)[D]\Bigr)\right|_{\widetilde{D}} \right\}.
    \end{equation}
    Take a Kähler form $\omega$ on $X$.
    Replacing $T$ and $T'$ by $T+\omega$ and $T'+\omega$, we may assume that both are Kähler currents. Take quasi-equisingular approximations $(S_k)_k,(S'_k)_k$ of $T$ and $T'$ so that $S_k\preceq S'_k$. 
    It follows from Step~2 that for each $k>0$,
    \begin{equation}\label{eq:wedge_diff_temp1}
    \left\{ T_1\wedge \cdots \wedge T_{p}\wedge S'_k \right\}- \left\{ T_1\wedge \cdots \wedge T_p\wedge S_k \right\} \geq_X \sum_{D\subseteq X}\nu(S_k,D)\left\{ \bigwedge_{i=1}^{p}\left.\Bigl( T_i-\nu(T_i,D)[D]\Bigr)\right|_{\widetilde{D}} \right\}.
    \end{equation}
    Note that in all expressions like $\sum_{D\subseteq X}$ above, we only need to consider the countable set of prime divisors $D$ with $\nu(T,D)>0$. 
    
    Next, for each $k>0$, we have
    \[
    \sum_{D\subseteq X}\nu(S_k,D)\left\{ \bigwedge_{i=1}^{p}\left.\Bigl( T_i-\nu(T_i,D)[D]\Bigr)\right|_{\widetilde{D}} \right\}\leq_X \left\{ T_1\wedge \cdots \wedge T_{p}\wedge S'_1 \right\}
    \]
    by \eqref{eq:wedge_diff_temp1} and the monotonicity theorem \cref{thm:mono_thm_partial}. Therefore, the monotone convergence theorem
    \cref{prop:Levi} is applicable. Letting $k\to\infty$ in \eqref{eq:wedge_diff_temp1} and applying \cref{thm:dscontnpp1}, we conclude the desired inequality \eqref{eq:wedge_diff_temp2}.
\end{proof}
Now using the language of b-divisors, we derive our main theorem:
\begin{theorem}\label{thm:bdivint_diff_ineq}
     Let $\mathbb{D}_1,\ldots,\mathbb{D}_p$ be nef b-divisors over $X$, and $T,T'$ be closed positive $(1,1)$-currents on $X$ in the same cohomology class such that $T\preceq_{\mathcal{I}}T'$. Then
    \begin{equation}\label{eq:Ddiffineq1}
    \begin{split}
        &\Bigl(\mathbb{D}_1\cap \dots \cap \mathbb{D}_p \cap \mathbb{D}(T')\Bigr)_X-\Bigl(\mathbb{D}_1\cap \dots \cap \mathbb{D}_p \cap \mathbb{D}(T)\Bigr)_X\\
        \geq_X & \sum_{D/X} \Bigl(\nu(T,D)-\nu(T',D) \Bigr) \vol_{X|D}\left(\mathbb{D}_1,\dots, \mathbb{D}_p \right).
    \end{split}
    \end{equation}
    Furthermore, equality holds in \eqref{eq:Ddiffineq1} when $\mathbb{D}_1,\ldots,\mathbb{D}_p$ are Cartier.
\end{theorem}
The right-hand side of \eqref{eq:Ddiffineq1} should be interpreted as the following limit:
\[
\lim_{\pi\colon Y\rightarrow X} \sum_{D\subseteq Y}\Bigl(\nu(T,D)-\nu(T',D) \Bigr) \vol_{X|D}\left(\mathbb{D}_1,\dots, \mathbb{D}_p \right).
\]
In terms of b-divisors, one can also formulate \eqref{eq:Ddiffineq1} as 
\[
\begin{split}
        &\Bigl(\mathbb{D}_1\cap \dots \cap \mathbb{D}_p \cap \mathbb{D}(T')\Bigr)-\Bigl(\mathbb{D}_1\cap \dots \cap \mathbb{D}_p \cap \mathbb{D}(T)\Bigr)\\
        \geq & \sum_{D/X} \Bigl(\nu(T,D)-\nu(T',D) \Bigr) \vol_{|D}\left(\mathbb{D}_1,\dots, \mathbb{D}_p \right).
\end{split}
\]

\begin{proof}
    From the bimeromorphic invariance of our assumptions, it suffices to prove the following:
    \[
    \begin{split}
        &\Bigl(\mathbb{D}_1\cap \dots \cap \mathbb{D}_p \cap \mathbb{D}(T')\Bigr)_X-\Bigl(\mathbb{D}_1\cap \dots \cap \mathbb{D}_p \cap \mathbb{D}(T)\Bigr)_X \\
       \geq_X & \sum_{D\subseteq X} \Bigl(\nu(T,D)-\nu(T',D) \Bigr) \left(\Tr_D \mathbb{D}_1 \cap \dots \cap \Tr_D \mathbb{D}_p \right)_{X}.
    \end{split}
    \]
    Fix a Kähler form $\omega$ on $X$ and replace $T$ and $T'$ by $T+\omega$ and $T'+\omega$, we may assume that both are Kähler currents.
    
    Take a closed smooth real $(1,1)$-form $\theta\in \{T\}$ and write
    \[
    T=\theta+\ddc \varphi,\quad T'=\theta+\ddc\varphi'.
    \]
    We may then replace $\varphi$ and $\varphi'$ by $P_{\theta}[\varphi]_{\mathcal{I}}$ and $P_{\theta}[\varphi']_{\mathcal{I}}$ respectively and assume that both $T$ and $T'$ are $\mathcal{I}$-good and $T\preceq T'$.
    
    Replacing $T$ and $T'$ by
    \[
    T-\sum_{D\subseteq X}\nu(T',D)[D],\quad T'-\sum_{D\subseteq X}\nu(T',D)[D],
    \]
    we may further assume that $T'$ is non-divisorial and hence it remains to prove
    \begin{equation}\label{eq:DintDdiff_temp1}
    \begin{split}
        &\Bigl(\mathbb{D}_1\cap \dots \cap \mathbb{D}_p \cap \mathbb{D}(T')\Bigr)_X-\Bigl(\mathbb{D}_1\cap \dots \cap \mathbb{D}_p \cap \mathbb{D}(T)\Bigr)_X\\
        \geq_X & 
      \sum_{D\subseteq X} \nu(T,D)\left(\Tr_D \mathbb{D}_1 \cap \dots \cap \Tr_D \mathbb{D}_p \right)_{X} .
    \end{split}
    \end{equation}
    
    We first prove \eqref{eq:DintDdiff_temp1} when $\mathbb{D}_i=\mathbb{D}(T_i)$ for some non-divisorial $\mathcal{I}$-good Kähler current $T_i$ for all $i=1,\ldots,p$. In this case, thanks to \cref{thm:DTandnpp}, \eqref{eq:DintDdiff_temp1} reduces to
    \[
    \left\{ T_1\wedge \cdots \wedge T_p\wedge T' \right\}-\left\{ T_1\wedge \cdots \wedge T_p\wedge T \right\}\geq_X
    \sum_{D\subseteq X} \nu(T,D)\left\{  \Tr_D T_1 \wedge \cdots \wedge \Tr_D T_p \right\}.
    \]
    Here as usual, we omitted the obvious pushforward from $\widetilde{D}$ to $X$.
    
    Taking quasi-equisingular approximations $(T_i^j)_{j>0}$ of $T_i$ for each $i=1,\ldots,p$, we can then apply \cref{lma:Tpwedgediff_analy} to conclude that for each $j>0$,
    \[
    \left\{ T^j_1\wedge \cdots \wedge T^j_p\wedge T' \right\}-\left\{ T^j_1\wedge \cdots \wedge T^j_p\wedge T \right\}\geq_X
    \sum_{D\subseteq X} \nu(T,D)\left\{  \Tr_D T^j_1 \wedge \cdots \wedge \Tr_D T^j_p \right\}.
    \]
    Letting $j\to\infty$, we conclude using \cref{thm:bdiv_int_cont1} and \cref{prop:incdom}.

    Now let us come back to the general situation. For each $\epsilon>0$, we then have
    \[
    \begin{aligned}
        &\biggl(\Bigl( \mathbb{D}_1+\epsilon\mathbb{D}(\omega) \Bigr)\cap \dots \cap \Bigl( \mathbb{D}_p+\epsilon\mathbb{D}(\omega) \Bigr) \cap \mathbb{D}(T')\biggr)_X\\
        &-\biggl(\Bigl( \mathbb{D}_1+\epsilon\mathbb{D}(\omega) \Bigr)\cap \dots \cap \Bigl( \mathbb{D}_p+\epsilon\mathbb{D}(\omega) \Bigr)\cap \mathbb{D}(T)\biggr)_X \\
       \geq_X &\sum_{D\subseteq X} \nu(T,D) \biggl(\Tr_D \Bigl( \mathbb{D}_1+\epsilon\mathbb{D}(\omega) \Bigr) \cap \dots \cap \Tr_D \Bigl( \mathbb{D}_p+\epsilon\mathbb{D}(\omega) \Bigr) \biggr)_{X}\\
       \geq_X &  \sum_{D\subseteq X} \nu(T,D) \left(\Tr_D \mathbb{D}_1 \cap \dots \cap \Tr_D \mathbb{D}_p \right)_{X}.
    \end{aligned}
    \]
    Letting $\epsilon\to 0+$ we conclude \eqref{eq:DintDdiff_temp1}.

    Finally, it remains to prove the final assertion. Assume that $\mathbb{D}_1,\ldots,\mathbb{D}_p$ are all Cartier. Replacing $X$ by a modification, we may assume that $\mathbb{D}_i=\mathbb{D}(\alpha_i)$ for some nef class $\alpha_i\in \mathrm{H}^{1,1}(X,\mathbb{R})$ for each $i$. Then thanks to \cref{prop:bdiv_int_prop}(5) and \cref{ex:res_vol_nef}, the desired equality becomes
    \[
    \alpha_1\cap \cdots \cap \alpha_p \cap \{\Reg T'\}-\alpha_1\cap \cdots \cap \alpha_p \cap \{\Reg T\}=\sum_{D\subseteq X} \Bigl(\nu(T,D)-\nu(T',D) \Bigr) \alpha_1\cap \cdots \cap \alpha_p \cap \{D\},
    \]
    which is obvious.
\end{proof}
Although the final assertion in \cref{thm:bdivint_diff_ineq} is easy to prove after establishing the general theory of b-divisors, it is far from trivial when stated in terms of currents:
\begin{corollary}\label{cor:main_eq_analy}
    Let $T_1,\ldots,T_{n-1}$ be closed positive $(1,1)$-currents with analytic singularities on $X$, and $T,T'$ be closed positive $(1,1)$-currents in the same cohomology class such that $T\preceq_{\mathcal{I}} T'$, then
    \begin{equation}
    \begin{aligned}
        &\vol\left(T_1,\ldots,T_{n-1},T'\right)-\vol\left(T_1,\ldots,T_{n-1},T\right)\\
        =&\lim_{\pi\colon Y\rightarrow X}\sum_{D\subseteq Y} \Bigl(\nu(T,D)-\nu(T',D) \Bigr)\int_{\tilde{D}} \bigwedge_{i=1}^{n-1}\left.\Bigl(\pi^*T_i-\nu(T_i,D)[D]\Bigr)\right|_{\tilde{D}}.
    \end{aligned}
    \end{equation}
    Furthermore, the net on the right-hand side is eventually constant.
\end{corollary}

We get some interesting new inequalities even for the movable intersection theory:
\begin{corollary}\label{cor:mov_modnef_minus_ineqgen}
    Let $[E]$ be a divisorial closed positive $(1,1)$-current on $X$, say 
    \begin{equation}\label{eq:Edecomp}
    [E]=\sum_i c_i E_i,
    \end{equation}
    where the $E_i$'s are distinct prime divisors on $X$ and $c_i>0$. Consider pseudoeffective classes $\alpha_1,\ldots,\alpha_p,\beta\in \mathrm{H}^{1,1}(X,\mathbb{R})$. Assume that $\beta\geq \{E\}$, then
    \begin{equation}\label{eq:mov_diff1}
    \begin{split}
        &\langle \alpha_1\wedge \cdots \wedge \alpha_p \wedge \beta \rangle-\Bigl\langle \alpha_1\wedge \cdots \wedge \alpha_p \wedge (\beta-\{E\}) \Bigr\rangle\\
        \geq_X &\sum_{i} \Bigl(\nu\left(\beta-\{E\},E_i\right)+c_i-\nu(\beta,E_i) \Bigr) \vol_{X|E_i}\Bigl(\langle \alpha_1\rangle ,\ldots,\langle \alpha_p\rangle \Bigr).
    \end{split}
    \end{equation}
\end{corollary}
Here we used the mixed restricted volume defined in \cref{def:vol_mix_class}. This inequality is stated purely using the traditional language of movable intersection product and restricted volume, at least when $\alpha_1=\dots=\alpha_p$. But it is far from being obvious without the knowledge of b-divisors.
\begin{proof}
    We first make a simple observation: The coefficient in \eqref{eq:mov_diff1} is non-negative:
    \[
    \nu\left(\beta-\{E\},E_i\right)+c_i-\nu(\beta,E_i)\geq 0,
    \]
    as a simple consequence of the convexity of $\nu(\bullet,E_i)$ proved in \cite[Proposition~2.1.8]{Bou02}. Since both sides of \eqref{eq:mov_diff1} remain invariant if we replace $\alpha_j$ by $\langle \alpha_j \rangle$, we may assume that each $\alpha_j$ is modified nef. 

    We may assume that the index set in \eqref{eq:Edecomp} is a set of the form $\{1,2,\ldots,N\}$, where $N$ is possibly $\infty$. 
    
    \textbf{Step~1}.
    We first reduce to the case where $N$ is finite. Assume that this case is known, let us handle the case where $N=\infty$. For each positive integer $M$, we introduce 
    \[
    E^M=\sum_{i=1}^M c_i E_i.
    \]
    Then using \cref{prop:mov_int_prop}(4), we have
    \[
    \begin{aligned}
     &\langle \alpha_1\wedge \cdots \wedge \alpha_p \wedge \beta \rangle-\Bigl\langle \alpha_1\wedge \cdots \wedge \alpha_p \wedge \left(\beta-\{E\}\right) \Bigr\rangle\\
        \geq_X &\langle \alpha_1\wedge \cdots \wedge \alpha_p \wedge \beta \rangle-\Bigl\langle \alpha_1\wedge \cdots \wedge \alpha_p \wedge \left(\beta-\left\{E^M\right\}\right) \Bigr\rangle\\
        \geq_X &\sum_{i=1}^M \Bigl(\nu\left(\beta-\left\{E^M\right\},E_i\right)+c_i-\nu(\beta,E_i) \Bigr) \vol_{X|E_i}\Bigl(\alpha_1,\ldots, \alpha_p\Bigr).
    \end{aligned}
    \]
    Observe that $\nu\left(\beta-\left\{E^M\right\},E_i\right)$ is an increasing function of $M$ as long as $M\geq i$, therefore, we can apply the monotone convergence theorem \cref{prop:Levi} to conclude that
    \[
    \begin{aligned}
     &\langle \alpha_1\wedge \cdots \wedge \alpha_p \wedge \beta \rangle-\Bigl\langle \alpha_1\wedge \cdots \wedge \alpha_p \wedge \left(\beta-\{E\}\right) \Bigr\rangle\\
        \geq_X &\sum_{i=1}^{\infty} \Bigl(\lim_{M\to\infty}\nu\left(\beta-\left\{E^M\right\},E_i\right)+c_i-\nu(\beta,E_i) \Bigr) \vol_{X|E_i}\Bigl(\alpha_1,\ldots,\alpha_p\Bigr).
    \end{aligned}
    \]
    But the function $\nu(\bullet,E_i)$ is lower semi-continuous as proved in \cite[Proposition~2.1.8]{Bou02}, \eqref{eq:mov_diff1} then follows.

    \textbf{Step~2}. We assume that $N$ is finite and prove \eqref{eq:mov_diff1}. Fix a Kähler class $\gamma\in \mathrm{H}^{1,1}(X,\mathbb{R})$
    
    We can replace $\alpha_j$ by $\alpha_j+\epsilon\gamma$ for all $\epsilon>0$. In this way, we reduce to the case where the $\alpha_i$'s are all big and none of the $E_i$'s lies in the non-Kähler loci of the $\alpha_i$'s, and hence  \cref{prop:res_vol_class} is applicable.
    
    Observe that
    \[
    \lim_{\epsilon\to 0+}\nu\left(\beta+\epsilon\gamma-\{E\},E_i\right)=\nu\left(\beta-\{E\},E_i\right),\quad \lim_{\epsilon\to 0+}\nu\left(\beta+\epsilon\gamma,E_i\right)=\nu\left(\beta,E_i\right),\quad \forall i.
    \]
    Therefore, in view of \cref{prop:mov_int_prop}(6), we may therefore replace $\beta$ by $\beta+\epsilon\gamma$ for some $\epsilon>0$ and assume that $\beta$ and $\beta-\{E\}$ are both big.

    We apply \cref{thm:bdivint_diff_ineq} to the following situation: $\mathbb{D}_i=\mathbb{D}(\alpha_i)$, and $T'$ is a current with minimal singularities in $\beta$ and $T-[E]$ is a current with minimal singularities in $\beta-\{E\}$. Then we find
    \[
    \begin{aligned}
        &\Bigl(\mathbb{D}(\alpha_1)\cap \dots \cap \mathbb{D}(\alpha_p) \cap \mathbb{D}(\beta)\Bigr)_X-\Bigl(\mathbb{D}(\alpha_1)\cap \dots \cap \mathbb{D}(\alpha_p) \cap \mathbb{D}\left(\beta-\{E\}\right)\Bigr)_X\\
        \geq_X &\sum_{D\subseteq X} \Bigl(\nu(\beta-\{E\},D)-\nu(\beta,D)+\nu([E],D) \Bigr)\vol_{X|D}\Bigl(\mathbb{D}(\alpha_1),\dots, \mathbb{D}(\alpha_p) \Bigr)\\
        \geq_X &\sum_{i} \Bigl(\nu(\beta-\{E\},E_i)-\nu(\beta,E_i) +c_i\Bigr)\vol_{X|E_i}\Bigl(\mathbb{D}(\alpha_1),\dots, \mathbb{D}(\alpha_p) \Bigr).
    \end{aligned}
    \]
    Thanks to \cref{cor:mov_int_generalized} and \cref{prop:res_vol_class}(3), this inequality translates immediately to \eqref{eq:mov_diff1}.
\end{proof}
The following special case is already very non-trivial:
\begin{corollary}\label{cor:mov_modnef_minus_ineq}
    Let $[E]$ be a divisorial closed positive $(1,1)$-current on $X$, say 
    \[
    [E]=\sum_i c_i E_i,
    \]
    where the $E_i$'s are distinct prime divisors on $X$ and $c_i>0$. Consider modified nef classes $\alpha_1,\ldots,\alpha_p,\beta\in \mathrm{H}^{1,1}(X,\mathbb{R})$. Assume that $\beta\geq \{E\}$, then
    \[
    \langle \alpha_1\wedge \cdots \wedge \alpha_p \wedge \beta \rangle-\Bigl\langle \alpha_1\wedge \cdots \wedge \alpha_p \wedge (\beta-\{E\}) \Bigr\rangle
        \geq_X \sum_{i} \Bigl(\nu(\beta-\{E\},E_i)+c_i\Bigr)\vol_{X|E_i}\left(\alpha_1,\ldots,\alpha_p\right).
    \]
\end{corollary}
A weaker result with $c_i$ in place of $\nu(\beta-\{E\},E_i)+c_i$ can also be proved via Vu's theory of relative non-pluripolar products. The weaker inequality has already played a significant role in the works of Su and Vu.

Although \eqref{eq:Ddiffineq1} is an equality when the $\mathbb{D}_1,\ldots,\mathbb{D}_p$ are all Cartier, in general, it fails to be an equality as can be easily seen in the toric setup. So we shall formulate a variant of \eqref{eq:Ddiffineq1}.

In order to proceed further, we shall need the celebrated conjectural transcendental Morse inequality \cite{BDPP13}:
\begin{equation}\label{eq:BDPP}
    \left.\frac{\mathrm{d}}{\mathrm{d}t}\right|_{t=0}\vol (\alpha+t\beta)=n\langle \alpha^{n-1} \rangle\cap \beta
\end{equation}
for all classes $\alpha,\beta\in \mathrm{H}^{1,1}(X,\mathbb{R})$ with $\alpha$ big. This conjecture is known when $X$ is projective, as proved by Witt Nyström in \cite{WN19dual}. When $n=1,2$, it is also known, see \cite{Deng17}.

We shall need the following consequence:
\begin{theorem}\label{thm:BDPP1}
    Assume that the transcendental Morse inequality holds on $X$. Then for any prime divisor $D$ on $X$ and any big class $\alpha\in \mathrm{H}^{1,1}(X,\mathbb{R})$, we have
    \begin{equation}\label{eq:resvol_BDPP}
    \vol_{X|D}(\alpha)=\langle \alpha^{n-1} \rangle\cap \{D\}.
    \end{equation}
\end{theorem}
We note that the $\leq$ direction in \eqref{eq:resvol_BDPP} is trivial. 
\begin{remark}
    It is not clear to the author whether the mixed version holds as well: 
    \[
    \vol_{X|D}\left(\alpha_1,\ldots,\alpha_{n-1}\right)=\left\langle \alpha_1\wedge \cdots \wedge \alpha_{n-1}\right\rangle \cap \{D\}.
    \]
    A seemingly weaker (but in fact equivalent) statement is that $\vol_{X|D}\left(\alpha_1,\ldots,\alpha_{n-1}\right)$ depends only on the linear equivalence class of $D$ when $D$ is ample.
\end{remark}

\begin{proof}
    This is a consequence of \eqref{eq:BDPP} and the main theorems of  \cite{WN21, Vu23der}.
\end{proof}

\begin{theorem}\label{thm:main_thm}
    Let $\mathbb{D}$ be a big and nef b-divisors over $X$, $T,T'$ be closed positive $(1,1)$-currents on $X$ in the same cohomology class such that $T\preceq_{\mathcal{I}}T'$. Then
    \begin{equation}\label{eq:Ddiffprecise}
     \Bigl(\mathbb{D}^{n-1} \cap \mathbb{D}(T')\Bigr)-\Bigl(\mathbb{D}^{n-1} \cap \mathbb{D}(T)\Bigr)
        \geq  \varlimsup_{\pi\colon Y\rightarrow X}\sum_{D\subseteq Y}\Bigl(\nu(T,D)-\nu(T',D) \Bigr) \vol_{Y|D}\left(\mathbb{D}_{Y}\right).
    \end{equation}    
    If furthermore either $X$ is projective or $n=1,2$, then the limit exists and equality holds in \eqref{eq:Ddiffprecise}.
\end{theorem}

The projectivity assumption (or $n=1,2$ assumption) is only to guarantee that we can apply \cref{thm:BDPP1}. If the transcendental Morse inequality is known (or if the much weaker assertion \eqref{eq:resvol_BDPP} is known unconditionally), the same proof works without this assumption as well.
\begin{proof}
    \textbf{Step~1}. We make some preliminary reductions.
    
    After adding a Kähler form to $T$ and $T'$, neither side of \eqref{eq:Ddiffprecise} changes, so we may assume that $T$ and $T'$ are both Kähler currents. Take a smooth closed real $(1,1)$-form $\theta\in \{T\}$, then we can represent $T=\theta+\ddc \varphi$ and $T'=\theta+\ddc \varphi'$. After replacing $T$ and $T'$ by $\theta+\ddc P_{\theta}[\varphi]_{\mathcal{I}}$ and $\theta+\ddc P_{\theta}[\varphi']_{\mathcal{I}}$, we may assume that they are both $\mathcal{I}$-good. Finally, after adding a Kähler form to them again, we reduce to the case where $T$ and $T'$ are both $\mathcal{I}$-good Kähler currents.     

    Next observe that the $\geq_X$ direction in \eqref{eq:Ddiffprecise} follows from \cref{thm:bdivint_diff_ineq}. In fact, for each fixed modification $\pi\colon Y\rightarrow X$, we have
    \[
    \begin{aligned}
        &\Bigl(\mathbb{D}\left(\mathbb{D}_Y\right)^{n-1} \cap \mathbb{D}(T')\Bigr)_X-\Bigl(\mathbb{D}\left(\mathbb{D}_Y\right)^{n-1} \cap \mathbb{D}(T)\Bigr)_X\\
        \geq & \sum_{D\subseteq Y}  \Bigl(\nu(T,D)-\nu(T',D) \Bigr) \vol_{Y|D}\Bigl(\mathbb{D}\left(\mathbb{D}_Y\right)^{n-1}\Bigr)\\
        = & \sum_{D\subseteq Y}\Bigl(\nu(T,D)-\nu(T',D) \Bigr) \vol_{Y|D}\Bigl(\mathbb{D}_{Y}\Bigr).
    \end{aligned}
    \]
    The proof of \cite[Theorem~4.10]{Xiabdiv} shows that  $( \mathbb{D}(\mathbb{D}_{Y}) )_{\pi\colon Y\rightarrow X}$ is a decreasing net with limit $\mathbb{D}$. Therefore, taking limit with respect to $\pi\colon Y \rightarrow X$ and applying \cref{thm:bdiv_int_cont1}, we conclude the $\geq$ direction in \eqref{eq:Ddiffprecise}.
    Our assertion becomes
    \begin{equation}\label{eq:diff_toprove_temp1}
    \left(\mathbb{D}^{n-1} \cap \mathbb{D}(T')\right)_X-\left(\mathbb{D}^{n-1} \cap \mathbb{D}(T)\right)_X\leq \varliminf_{\pi\colon Y\rightarrow X}\sum_{D\subseteq Y} \Bigl(\nu(T,D)-\nu(T',D) \Bigr) \vol_{Y|D}\left(\mathbb{D}_Y^{n-1} \right).
    \end{equation}

    \textbf{Step~2}. Represent $\mathbb{D}$ as $\mathbb{D}(S)$ for some non-divisorial $\mathcal{I}$-good closed positive $(1,1)$-current $S$ on $X$. This is possible thanks to \cref{thm:main_paper1}.
    In view of \cref{thm:DTandnpp}, it remains to show that
    \begin{equation}\label{eq:intSnm1Tp_temp}
        \int_X S^{n-1}\wedge T'-\int_X S^{n-1}\wedge T\leq \varliminf_{\pi\colon Y\rightarrow X}\sum_{D\subseteq Y} \Bigl(\nu(T,D)-\nu(T',D) \Bigr) \vol_{Y|D}\left(\mathbb{D}(S)_Y\right).
    \end{equation}
    We make a further reduction, we may add freely a Kähler form to $T$ and $T'$, so that we may assume that $\{T\}$ is a Kähler class, say represented by a Kähler form $\omega$. Due to \cref{thm:bdivint_diff_ineq}, we have
    \[
    \begin{aligned}
    &\int_X S^{n-1}\wedge \omega-\int_X S^{n-1}\wedge T \\
    =&\int_X S^{n-1}\wedge \omega-\int_X S^{n-1}\wedge T'+\int_X S^{n-1}\wedge T'-\int_X S^{n-1}\wedge T\\
    \geq &\varlimsup_{\pi\colon Y\rightarrow X}\sum_{D\subseteq Y} \Bigl(\nu(T,D)-\nu(T',D) \Bigr) \vol_{Y|D}\left(\mathbb{D}(S)_Y\right)+ \varlimsup_{\pi\colon Y\rightarrow X}\sum_{D\subseteq Y} \nu(T',D) \vol_{Y|D}\left(\mathbb{D}(S)_Y\right)\\
    =& \varlimsup_{\pi\colon Y\rightarrow X}\sum_{D\subseteq Y} \nu(T,D) \vol_{Y|D}\left(\mathbb{D}(S)_Y\right)\\
    \geq & \varliminf_{\pi\colon Y\rightarrow X}\sum_{D\subseteq Y} \nu(T,D) \vol_{Y|D}\left(\mathbb{D}(S)_Y\right).
    \end{aligned}
    \]
    In order to establish \eqref{eq:intSnm1Tp_temp}, it suffices to prove the outer equality, namely we may assume that $T'=\omega$, and it remains only to prove that
    \begin{equation}\label{eq:reduc_int_diff_temp2}
         \int_X S^{n-1}\wedge \omega-\int_X S^{n-1}\wedge T\leq  \varliminf_{\pi\colon Y\rightarrow X}\sum_{D\subseteq Y} \nu(T,D)\vol_{Y|D}\left(\mathbb{D}(S)_Y\right).
    \end{equation}
    Next take a quasi-equisingular approximations $(T_j)_{j>0}$ of $T$. Suppose that we can prove \eqref{eq:reduc_int_diff_temp2} with $T_j$ in place of $T$, namely
        \[
        \int_X S^{n-1}\wedge \omega-\int_X S^{n-1}\wedge T_j\leq  \varliminf_{\pi\colon Y\rightarrow X}\sum_{D\subseteq Y} \nu(T_j,D)\vol_{Y|D}\left(\mathbb{D}(S)_Y\right).
        \]
        Letting $j\to\infty$ and applying  \cref{prop:Levi} and \cref{thm:dscontnpp1}, we conclude \eqref{eq:reduc_int_diff_temp2}. In particular, we have reduced to the case where $T$ has analytic singularities. 

        Take a modification $\pi\colon Y\rightarrow X$ resolving the singularities of $T$. Then it suffices to prove \eqref{eq:reduc_int_diff_temp2} when $T$ has log singularities.
        
    After all these reductions, we are finally reduced to the following assertion: Suppose that $E$ is a prime divisor on $X$, then
    \begin{equation}\label{eq:Snminus1capE_temp3}
    \{S^{n-1}\}\cap \{E\}\leq \varliminf_{\pi\colon Y\rightarrow X}\sum_{D\subseteq Y}  \nu([E],D) \vol_{Y|D}\left(\mathbb{D}(S)_Y\right).
    \end{equation}
    Here $S$ is an $\mathcal{I}$-good current with positive volume. Note that the reverse inequality is known, namely follows from \cref{thm:bdivint_diff_ineq}.
    
    \textbf{Step~4}. We finally prove \eqref{eq:Snminus1capE_temp3} using the transcendental Morse inequality. Only in this step are we in need of the assumption that $X$ is projective or $n=1,2$.

    If $\pi\colon Y\rightarrow X$ is a modification, then \cref{thm:BDPP1} shows that 
    \begin{equation}\label{eq:cons_BDPP_temp1}
    \pi_*\Bigl\langle \mathbb{D}(S)_{Y}^{n-1} \Bigr\rangle \cap \{E\}=\Bigl\langle \mathbb{D}(S)_{Y}^{n-1} \Bigr\rangle \cap \pi^*\{E\}=\sum_{D\subseteq Y} \nu([E],D)\vol_{Y|D}\Bigl( \mathbb{D}(S)_{Y}\Bigr).
    \end{equation}
    By \cref{thm:nefbint_aslimit} and \cref{thm:DTandnpp}, the limit of the left-hand side of \eqref{eq:cons_BDPP_temp1} is nothing but
    \[
        \left( \mathbb{D}(S)^{n-1} \right)_{X}\cap \{E\}=\{S^{n-1}\}\cap \{E\}.
    \]
    Therefore, \eqref{eq:Snminus1capE_temp3} follows.
\end{proof}

When $n=1$, \cref{thm:main_thm} reduces to the following elegant formula:
\begin{corollary}\label{cor:RS}
    Assume that $X$ is a compact Riemann surface. Let $T,T'$ be closed positive $(1,1)$-currents on $X$ in the same cohomology class such that $T\preceq_{\mathcal{I}}T'$. Then
    \begin{equation}
        \vol T'-\vol T=\sum_{x\in X} \Bigl( \nu(T,x)-\nu(T',x) \Bigr).
    \end{equation}
    In particular,
    \[
    \vol T=\vol\{T\}-\sum_{x\in X} \nu(T,x).
    \]
\end{corollary}

\section{Loss of masses in terms of Lelong numbers}\label{sec:lom}
Let $X$ be a connected compact Kähler manifold of dimension $n$.

In a series of papers, Vu and Su \cite{Vu23, Su25, Suthesis} established the following type of estimates: Consider a closed positive $(1,1)$-current $T$ in a big cohomology class $\alpha\in \mathrm{H}^{1,1}(X,\mathbb{R})$, then
\begin{equation}\label{eq:volalphaT_Vu}
\vol \alpha-\vol T\geq c \Bigl( \nu(T,x)-\nu(\alpha,x) \Bigr)^n,\quad \forall x\in X
\end{equation}
and similar estimates bounding the difference of the volumes of two currents from below. In none of these works is the constant $c$ explicit. In this section, based on the monotonicity theorems proved in \cref{sec:qual_mono}, we shall provide an explicit constant in a much more general setup.

Since the Lelong number at a point is the same as the generic Lelong number at the exceptional divisor after blowing-up this point, we shall consider Lelong numbers along divisors instead.

\subsection{The width of a big class}
Let $X$ be a connected compact Kähler manifold of dimension $n$, and $\alpha \in \mathrm{H}^{1,1}(X,\mathbb{R})$ be a big class.

\begin{definition}\label{def:numax}
    Let $D$ be a prime divisor over $X$. We define
    \[
        \nu_{\max}(\alpha,D)\coloneqq \sup \left\{\nu(T,D):T\textup{ is a closed positive }(1,1)\textup{-current}\in \alpha\right\}.
    \]
    Similarly, define the \emph{$D$-width} of $\alpha$ as 
    \[
    \wid(\alpha,D)\coloneqq \nu_{\max}(\alpha,D)-\nu(\alpha,D).
    \]
\end{definition}

The basic properties of these quantities are summarized as follows:
\begin{proposition}\label{prop:numaxalpha}
    Let $D$ be a prime divisor over $X$. Then we have the following:
    \begin{enumerate}
        \item For each modification $\pi\colon Y\rightarrow X$, we have
        \[
        \nu_{\max}(\alpha,D)=\nu_{\max}(\pi^*\alpha,D),\quad \wid(\alpha,D)=\wid(\pi^*\alpha,D).
        \]
        \item We have $\wid(\alpha,D)>0$.
        \item Suppose that $D\subseteq X$, then for any $t\geq 0$, we have
        \begin{equation}\label{eq:numax_addcurrenttoalpha}
            \nu_{\max}(\alpha,D)=\nu_{\max}\left(\alpha+t\{D\},D\right)-t.
        \end{equation}
        \item Suppose that $D\subseteq X$, then
        \[
        \wid(\alpha,D)=\wid\left( \alpha-\nu(\alpha,D)\{D\},D \right).
        \]
        \item $\nu_{\max}(\bullet,D)$ is homogeneous and concave in the big cone: If $\beta\in \mathrm{H}^{1,1}(X,\mathbb{R})$ is another big class, and $\lambda> 0$, then 
        \[
        \nu_{\max}(\lambda\alpha,D)=\lambda \nu_{\max}(\alpha,D),\quad \nu_{\max}(\alpha+\beta,D)\geq \nu_{\max}(\alpha,D)+\nu_{\max}(\beta,D).
        \]
        \item $\wid(\bullet,D)$ is continuous in the big cone.
    \end{enumerate}
\end{proposition}
\begin{proof}
    (1) It suffices to observe that $\pi^*$ is a bijection between the set of closed positive $(1,1)$-currents in $\alpha$ and the set of  closed positive $(1,1)$-currents in $\pi^*\alpha$.

    (2) Thanks to (1), we may assume that $D$ is a prime divisor on $X$. Since $\langle \alpha \rangle$ is big,  we can find $\epsilon>0$ small enough with $\langle\alpha\rangle -\epsilon \{D\}$ big. 
It follows that
\[
\alpha-\left(\nu(\alpha,D)+\epsilon\right)\{D\}
\]
is big.
Adding $(\nu(\alpha,D)+\epsilon)[D]$ to a current with minimal singularities in this class, we obtain a closed positive $(1,1)$-current $T\in \alpha$ with $\nu(T,D)>\nu(\alpha,D)$. 

(3) Fix $t\geq 0$.
    Take a closed positive $(1,1)$-current $T\in\alpha$, then $T+t[D]$ is a current in $\alpha+t\{D\}$. Since
    \[
    \nu(T+t[D],D)=\nu(T,D)+t,
    \]
    we conclude
    \[
    \nu_{\max}(\alpha,D)\leq \nu_{\max}\left(\alpha+t\{D\},D\right)-t.
    \]
    In particular, in view of (2), the right-hand side of \eqref{eq:numax_addcurrenttoalpha} is always positive.

    Next fix $\epsilon\in (0,\nu_{\max}(\alpha+t\{D\},D)-t)$, take a closed positive $(1,1)$-current $T\in \alpha+t\{D\}$ with 
    \[
    \nu(T,D)\geq \nu_{\max}(\alpha+t\{D\},D)-\epsilon.
    \]
    By our choice of $\epsilon$, $\nu(T,D)> t$. Then $T-t[D]\in \alpha$ is a closed positive $(1,1)$-current, and
    \[
    \nu_{\max}(\alpha,D)\geq \nu(T-t[D],D)=\nu(T,D)-t\geq \nu_{\max}(\alpha+t\{D\},D)-t-\epsilon.
    \]
    Letting $\epsilon\to 0+$, the reverse inequality follows.

(4) By (3), we have
\[
\nu_{\max}\left(\alpha-\nu(\alpha,D)\{D\},D\right)=\nu_{\max}(\alpha,D)-\nu(\alpha,D)=\wid(\alpha,D).
\]

(5) This is trivial.

(6) This follows from (5).
\end{proof}

\subsection{The toric setting}\label{subsec:toric}
In order to get a feeling of what kind of inequality we should expect, let us first consider the toric situation. For the details of the setup, we refer to \cite[Chapter~12]{Xiabook}.

Let $T$ be a complex torus of dimension $n$. Let $N$ (resp. $M$) be the cocharacter lattice (resp. character lattice) of $T$. Fix a fan $\Sigma$ in $N_{\mathbb{R}}$ corresponding to a smooth projective toric variety $X$. Fix a toric-invariant big divisor 
\[
H=\sum_{\rho\in \Sigma(1)} a_{\rho} D_{\rho},
\]
where $\Sigma(1)$ is the set of rays in $\Sigma$, $D_{\rho}$ is the toric-invariant divisor on $X$ corresponding to $\rho$ via the orbit-cone correspondence, and $a_{\rho}\in \mathbb{Z}$. Let $P_H$ denote the Newton polytope of $H$:
\[
P_H=\Bigl\{m\in M_{\mathbb{R}} : \langle m,u_{\rho}\rangle\geq -a_{\rho}\quad \forall \rho\in \Sigma(1) \Bigr\},
\]
where $u_{\rho}\in N$ is the ray generator of $\rho$. 

Recall the following key theorem proved in \cite[Chapter~12]{Xiabook}.
\begin{theorem}
    The $\mathcal{I}$-equivalence classes of toric-invariant closed positive $(1,1)$-currents in $\alpha$ are in natural bijection with the convex bodies contained in $P_H$. 
\end{theorem}
Here a convex body refers to a compact non-empty convex set.

Given such a current $T$, the corresponding convex body (known as the Newton body of $T$) is denoted by $\Delta(T)$. We have
\[
\vol T=n! \vol \Delta(T).
\]
Two toric-invariant currents $T, S$ in $\alpha$ satisfies $T\preceq_{\mathcal{I}} S$ if and only if $\Delta(T)\subseteq \Delta(S)$. Now fix a toric-invariant prime divisor $D$. In the ample case, $D$ corresponds to a facet of $P_H$, but this fails in the big case. However, if $\rho\in \Sigma(1)$ is the ray corresponding to $D$, then we have
\[
\nu(T,D)=\inf \left\{ \langle m,u_{\rho}\rangle+a_{\rho}: m\in \Delta(T) \right\}.
\]
So eventually, we are looking for such inequalities:
\begin{equation}\label{eq:volDeltadiff_lower_guess}
    \vol \Delta(S)-\vol \Delta(T)\geq c \biggl( \inf \Bigl\{ \langle m,u_{\rho}\rangle: m\in \Delta(T) \Bigr\}-\inf \Bigl\{ \langle m,u_{\rho}\rangle: m\in \Delta(S) \Bigr\}\biggr)^n.
\end{equation}

Note that it is impossible to find uniform $c$ for all $S$ and $T$, contrary to the assertions in the literature.
\begin{example}\label{ex:count_Su}
    Consider the simplex $Q$ with the following vertices in $\mathbb{R}^{n}$:
    \[
    (0,\ldots,0), (1,0,\ldots,0), (1,\epsilon,0,\ldots,0),(1,0,\epsilon,0,\ldots,0),\ldots, (1,0,\ldots,0,\epsilon)
    \]
    for some $\epsilon>0$. It is a polytope contained in $[0,1]^n$, the Newton polytope of $\mathcal{O}(1,\dots,1)$ on $(\mathbb{P}^1)^n$.

    Now let 
    \[
    Q_t=\{x\in Q:x_1\geq t\},
    \]
    where $t\in (0,1)$. Then
    \[
    \vol Q-\vol Q_t=\frac{t^n}{n!}\epsilon^{n-1}.
    \]
    If we take $u_{\rho}=(1,0,\ldots,0)$ corresponding to the facet $\{x_1=0\}$, we find that the constant $c$ in \eqref{eq:volDeltadiff_lower_guess} corresponds to 
    \[
    \frac{\epsilon^{n-1}}{n!},
    \]
    which can be arbitrarily small.
\end{example}

Now we look for the optimal constant $c$ in \eqref{eq:volDeltadiff_lower_guess}. We first make a few simplifications. Since $u_{\rho}$ is part of a basis of $N$, up to a linear transform in $\mathrm{SL}(n,\mathbb{Z})$, we may assume that $N=\mathbb{Z}^n$, $u_{\rho}=(1,0,\ldots,0)$. We have a given rational polytope $P_H$ with positive volume contained in $\{x_1\geq 0\}$, two subconvex bodies $Q_1\subseteq Q_2$. We wish to find $c$ so that
\begin{equation}\label{eq:volDeltadiff_lower_guess2}
    \vol Q_2-\vol Q_1\geq c \left(\inf_{x\in Q_1}x_1-\inf_{x\in Q_2}x_1 \right)^n.
\end{equation}
Without loss of generality, we may assume that $Q_2$ touches $\{x_1=0\}$. Namely, $\inf_{x\in Q_2}x_1=0$.

Now we fix $Q_2$ and $t\coloneqq \inf_{x\in Q_1}x_1$ and consider the optimal $Q_1$, namely when the left-hand side takes the minimal value. It is clearly given by
\[
Q_1\coloneqq \{(x_1,y)\in Q:x_1\geq t\}.
\]
See \cref{fig:Q1Q1} for the optimal situation.
\begin{figure}[ht]
\caption{The optimal situation}\label{fig:Q1Q1}
\centering
\includegraphics[width=0.5\textwidth]{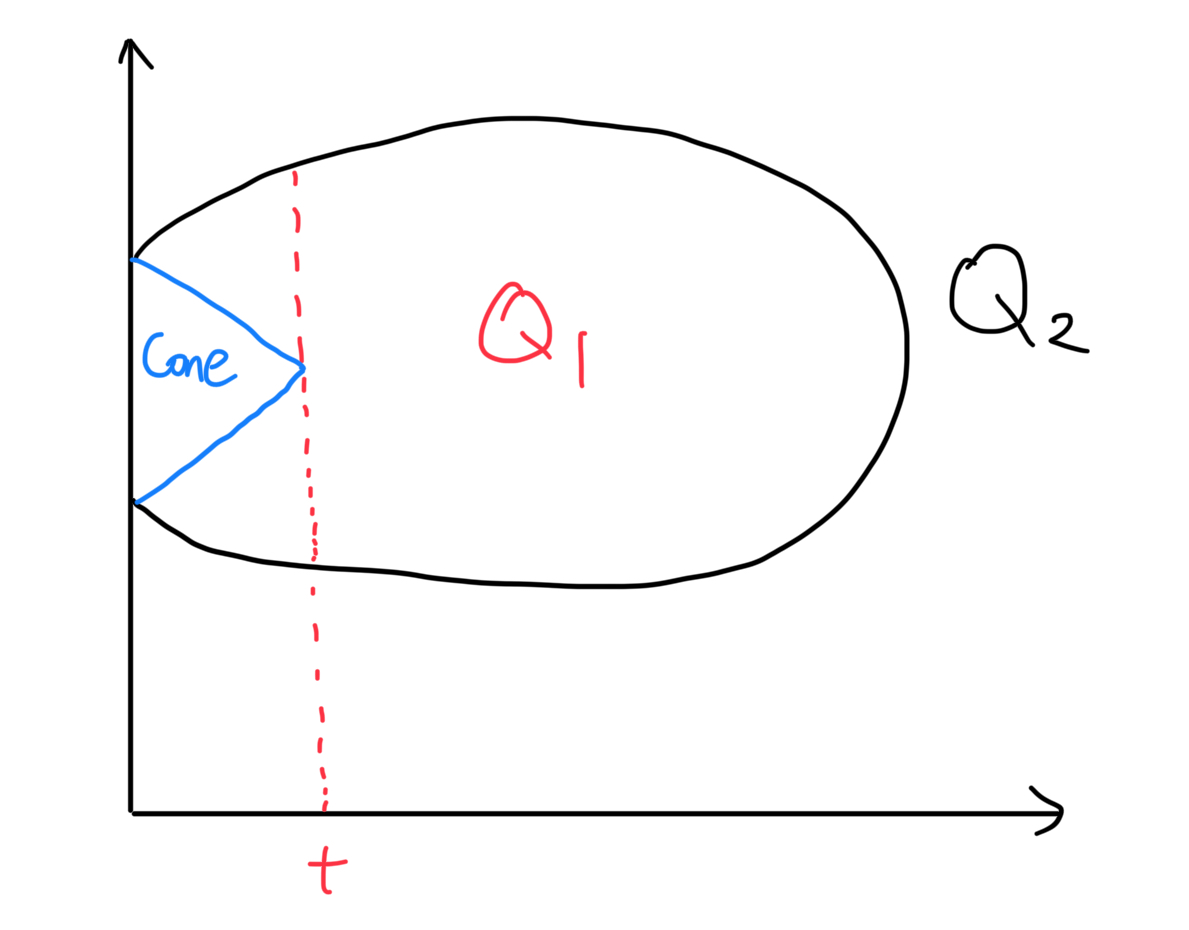}
\end{figure}

The volume of the remaining part is then bounded from below by the volume of a cone with  base $Q_2\cap \{x_1=0\}$ and a vertex with $x_1=t$:
\[
\frac{t}{n}\vol \{y\in \mathbb{R}^{n-1}:(0,y)\in Q_2\}. 
\]
But we wish to obtain something to the order of $t^n$, for this purpose, it suffices to observe that
\[
t\leq \max_{x\in P_H}x_1.
\]
Putting everything together, we find
\[
\begin{split}
\vol Q_2-\vol Q_1\geq \frac{1}{n\left(\max_{x\in P_H}x_1-\min_{x\in P_H}x_1\right)^{n-1}}\\
\left(\inf_{x\in Q_1}x_1-\inf_{x\in Q_2}x_1 \right)^n \vol \{y\in \mathbb{R}^{n-1}:(0,y)\in Q_2\}.
\end{split}
\]
Translating everything back to the language of currents, we get
\begin{proposition}\label{prop:ST_toric_volumediff}
    Let $S,T$ be toric invariant closed positive $(1,1)$-currents in the same cohomology class $\alpha$ with $S\succeq_{\mathcal{I}}T$, then for any toric-invariant prime divisor $D$ on $X$, we have
    \[
    \vol S-\vol T\geq \frac{1}{\wid(\alpha,D)^{n-1}} \cdot \Bigl( \nu(T,D)-\nu(S,D) \Bigr)^{n}\cdot \vol_{X|D}\Bigl(S-\nu(S,D)[D]\Bigr).
    \]
\end{proposition}
See \cref{def:numax} for the definition of $\nu_{\max}$.

This proof only works in the toric setting. In general, we shall rely on the non-toric generalization of Newton bodies, namely the Okounkov bodies.

\subsection{Transcendental Okounkov bodies}
Let $X$ be a connected compact Kähler manifold of dimension $n$.

Recall that a smooth flag on $X$ is a chain of smooth analytic subspaces $Y_{\bullet}=(Y_1\supseteq \cdots \supseteq Y_n)$ so that each $Y_i$ is irreducible of codimension $i$ in $X$. Given a big class $\alpha\in \mathrm{H}^{1,1}(X,\mathbb{R})$, there is a natural way of constructing a convex body $\Delta_{Y_{\bullet}}(\alpha)\subseteq \mathbb{R}^n$ associated with $Y_{\bullet}$ and $\alpha$, known as the Okounkov body. The Okounkov body is contained in the first quadrant. Instead of recalling the lengthy definition, we refer to \cite{Deng17, DRWNXZ} for the details.

We shall need a few basic properties of the Okounkov bodies, as proved in \cite{DRWNXZ} and \cite[Chapter~10]{Xiabook}. 
\begin{theorem}\label{thm:Oko_prop}
Let $\alpha\in \mathrm{H}^{1,1}(X,\mathbb{R})$ be a big class and $Y_{\bullet}$ be a smooth flag on $X$, then we have the following:
\begin{enumerate}
    \item The volume of the Okounkov body is proportional to the volume of the class:
        \begin{equation}\label{eq:vol_Oko}
            \vol \alpha=n!\vol \Delta_{Y_{\bullet}}(\alpha).
        \end{equation}
    \item The Okounkov body $\Delta_{Y_{\bullet}}(\alpha)$ is continuous with respect to $\alpha$ if we consider the topology induced by the Hausdorff metric.
    \item  The two ends have pluripotential-theoretic interpretations:
        \[
            \min_{x\in \Delta_{Y_{\bullet}}(\alpha)}x_1=\nu(\alpha,Y_1),\quad \max_{x\in \Delta_{Y_{\bullet}}(\alpha)}x_1=\nu_{\max}(\alpha,Y_1).
        \]
    \item For all $t\in (\nu(\alpha,Y_1),\nu_{\max}(\alpha,Y_1))$, we have
        \[
            \vol_{X|Y_1}\left(\alpha-t\{Y_1\}\right)=(n-1)!\vol \left\{y\in \mathbb{R}^{n-1}:(t,y)\in \Delta_{Y_{\bullet}}(\alpha)\right\}.
        \]
\end{enumerate}    
\end{theorem}
In fact, the slice in (4) is an example of the partial Okounkov bodies studied in \cite{Xia21} and \cite[Chapter~10]{Xiabook}.

We shall need a few auxiliary lemmata.
\begin{lemma}\label{lma:conc_int}
    Let $A>0$.
    Let $f\colon [0,A]\rightarrow \mathbb{R}_{\geq 0}$ be a continuous concave function. Then for any $n\in \mathbb{Z}_{>0}$, and any $t_0\in (0,A)$, we have
    \[
    \int_0^A f(t)^n\,\mathrm{d}t\leq \frac{f(t_0)^nA^{n+1}}{n+1} \min \{t_0,A-t_0\}^{-n}.
    \]
\end{lemma}
\begin{proof}
Fix $t_0\in (0,A)$, assume that $f(t_0)=C\geq 0$, let us compute the maximum of $\int_0^A f(t)^n\mathrm{d}t$. Since $f$ is concave, the maximization problem reduces immediately to the case where $f$ is affine. So we need to compute the maximum of
\[
\int_0^A \left( B(t-t_0)+C\right)^n\,\mathrm{d}t=\frac{(B(A-t_0)+C)^{n+1}-(-Bt_0+C)^{n+1}}{(n+1)B}
\]
under the constraints that
\[
-Bt_0+C\geq 0,\quad B(A-t_0)+C\geq 0.
\]
Observe that the integral is convex in $B$, and hence the maximum is obtained at the boundary of the interval, in other words, the maximum of the integral is 
\[
\max\left\{ \frac{C^nA^{n+1}}{(n+1)t_0^n}, \frac{C^nA^{n+1}}{(n+1)(A-t_0)^n}\right\}=\frac{C^nA^{n+1}}{n+1} \min \{t_0,A-t_0\}^{-n}.
\]
\end{proof}

\begin{lemma}\label{lma:slice_convexbody}
Assume that $n>1$.
    Let $P\subseteq \mathbb{R}^n$ be a convex body. Assume that
    \[
    \min_{x\in P}x_1=0,\quad \max_{x\in P}x_1=A\geq 0.
    \]
    Then for any $t_0\in [0,A]$, we have
    \[
    \vol \left\{y\in \mathbb{R}^{n-1}: (t_0,y)\in P \right\}\geq \frac{n}{A^n}\vol P\cdot \min\{t_0,A-t_0\}^{n-1}. 
    \]
\end{lemma}
The proof resulted from a discussion with Yangyang Li.
\begin{proof}
We may assume that $t_0\in (0,A)$ since otherwise there is nothing to prove.

We write 
\[
g(t)\coloneqq \vol \left\{y\in \mathbb{R}^{n-1}: (t,y)\in P \right\}.
\]
    By Brunn--Minkowski inequality, the function $g^{1/(n-1)}$ is concave.
    Therefore, thanks to \cref{lma:conc_int}, we have
    \[
    \vol P=\int_0^A g(t)\,\mathrm{d}t\leq \frac{g(t_0)A^n}{n}\min\{t_0,A-t_0\}^{-n+1}.
    \]
\end{proof}

The following lower bound of the restricted volume seems new.
\begin{corollary}\label{cor:res_vol_vol_ineq}
Let $\alpha\in \mathrm{H}^{1,1}(X,\mathbb{R})$ be a big class. Then for any prime divisor $D$ on $X$, we have
\begin{equation}\label{eq:res_vol_vol_ineq}
\vol_{X|D}\left(\alpha-t\{D\}\right)\geq \frac{\vol \alpha }{\wid(\alpha,D)^n}\cdot \min\{t-\nu(\alpha,D),\nu_{\max}(\alpha,D)-t\}^{n-1}
\end{equation}
as long as $\nu(\alpha,D)\leq t\leq \nu_{\max}(\alpha,D)$.
\end{corollary}
In particular, when $\nu(\alpha,D)< t < \nu_{\max}(\alpha,D)$, we have $\nu(\alpha-t\{D\},D)=0$.
\begin{proof}
The inequality \eqref{eq:res_vol_vol_ineq} is trivial if $t$ takes the boundary value. We may assume that 
\[
\nu(\alpha,D)< t < \nu_{\max}(\alpha,D).
\]
Since the problem is invariant after replacing $X$ by a modification and $D$ by its strict transform, we may assume that $D$ is smooth and is the first component in a smooth flag $Y_{\bullet}$ on $X$. Then our assertion follows from \cref{lma:slice_convexbody} and \cref{thm:Oko_prop}. More precisely, we have
    \[
    \begin{aligned}
        &\vol_{X|D}\left(\alpha-t\{D\}\right)\\
        \geq & (n-1)!\vol \left\{y\in \mathbb{R}^{n-1}:(t,y)\in \Delta_{Y_{\bullet}}(\alpha)\right\}\\
        \geq & (n-1)!\cdot \frac{n}{\Bigl(\nu_{\max}(\alpha,D)-\nu(\alpha,D)\Bigr)^n}\vol \Delta_{Y_{\bullet}}(\alpha)\cdot \min\{t-\nu(\alpha,D),\nu_{\max}(\alpha,D)-t\}^{n-1}\\
        =& \frac{\vol \alpha}{\Bigl(\nu_{\max}(\alpha,D)-\nu(\alpha,D)\Bigr)^n}\cdot \min\{t-\nu(\alpha,D),\nu_{\max}(\alpha,D)-t\}^{n-1}.
    \end{aligned}
    \]

\end{proof}

\subsection{Loss of mass problem}
Now fix a connected compact Kähler manifold $X$ of dimension $n$.

\begin{theorem}\label{thm:loss_mass_Lelongprod}
    Let $\alpha_1,\dots,\alpha_n\in \mathrm{H}^{1,1}(X,\mathbb{R})$ be big classes. Consider closed positive $(1,1)$-currents $S_i,T_i\in \alpha_i$ with $S_i\succeq_{\mathcal{I}} T_i$. Fix a prime divisor $D$ over $X$. Then
    \[
    \begin{aligned}
       & \vol(S_1,\ldots,S_n)-\vol(T_1,\ldots,T_n)\\
        \geq & \max_{i=1,\ldots,n}\left( \frac{1}{\prod_{j\neq i}\wid(\alpha_j,D)}\cdot \vol_{X|D}\left(\mathbb{D}(S_1),\ldots,\widehat{\mathbb{D}(S_i)},\ldots,\mathbb{D}(S_n)\right)\right) \\
        &\cdot \prod_{i=1}^n \Bigl(\nu(T_i,D)-\nu(S_i,D) \Bigr).
    \end{aligned}
    \]
\end{theorem}
\begin{proof}
    Observe that
    \[
    \vol(S_1,\ldots,S_n)-\vol(T_1,\ldots,T_n)\geq \vol(S_1,S_2,\ldots,S_n)-\vol(T_1,S_2,\ldots,S_n).
    \]
    We apply \cref{thm:main_thm} to get
    \[
    \vol(S_1,\ldots,S_n)-\vol(T_1,\ldots,T_n)\geq \Bigl(\nu(T_1,D)-\nu(S_1,D)\Bigr)\vol_{X|D}\Bigl(\mathbb{D}(S_2),\ldots,\mathbb{D}(S_n)\Bigr).
    \]
    Next observe that for $j>1$,
    \[
        \nu(T_j,D)-\nu(S_j,D)\leq \nu_{\max}(\alpha_j,D)-\nu(\alpha_j,D),
    \]
    and hence
    \[
    \begin{aligned}
    &\vol(S_1,\ldots,S_n)-\vol(T_1,\ldots,T_n)\\
    \geq &\prod_{j=2}^n \frac{\nu(T_j,D)-\nu(S_j,D)}{\nu_{\max}(\alpha_j,D)-\nu(\alpha_j,D)}\cdot \Bigl(\nu(T_1,D)-\nu(S_1,D)\Bigr)\cdot \vol_{X|D}\Bigl(\mathbb{D}(S_2),\ldots,\mathbb{D}(S_n)\Bigr),
    \end{aligned}
    \]
    and our assertion follows by symmetry.
\end{proof}

In order to appreciate our result, let us consider the following special cases.

\begin{corollary}
    Let $\alpha\in \mathrm{H}^{1,1}(X,\mathbb{R})$ be a big class. Consider closed positive $(1,1)$-currents $T,S\in \alpha$ with $T\preceq_{\mathcal{I}} S$. Then
    \[
    \vol S-\vol T\geq \frac{1}{\wid(\alpha,D)^{n-1}} \cdot \Bigl(\nu(T,D)-\nu(S,D)\Bigr)^n\cdot \vol_{X|D}\Bigl(S-\nu(S,D)[D]\Bigr).
    \]
\end{corollary}
We have obtain what we could expect from the toric situation \cref{prop:ST_toric_volumediff}.
\begin{proof}
    It suffices to apply \cref{thm:loss_mass_Lelongprod}. Note that the restricted volume of b-divisors is just $\vol_{X|D}(S-\nu(S,D)[D])$ thanks to \cref{ex:resvol_DTT}.
\end{proof}

\begin{corollary}\label{cor:volalphamvolT1Tn}
    Let $\alpha\in \mathrm{H}^{1,1}(X,\mathbb{R})$ be a big class, and $T_1,\ldots,T_n\in \alpha$ be closed positive $(1,1)$-currents. Suppose that $D$ is a prime divisor on $X$.
    Then
    \begin{equation}\label{eq:loss_mixed_wid1}
    \vol\alpha-\vol(T_1,\ldots,T_n)\geq \frac{1}{\wid(\alpha,D)^{n-1}} \vol_{X|D} \Bigl(\langle \alpha \rangle\Bigr) \cdot \prod_{i=1}^n \Bigl( \nu(T_i,D)-\nu(\alpha,D) \Bigr).
    \end{equation}
\end{corollary}
\begin{proof}
Thanks to \cref{prop:numaxalpha}(4), replacing $\alpha$ by $\alpha-\nu(\alpha,D)\{D\}$ and $T_i$ by $T_i-\nu(\alpha,D)[D]$ does not change either side of \eqref{eq:loss_mixed_wid1}, so we may assume that $\nu(\alpha,D)=0$.
    We apply \cref{thm:loss_mass_Lelongprod}, with all $S_i$'s being currents with minimal singularities in $\alpha$. The restricted volume of the b-divisors is just  $\vol_{X|D} \left(\alpha\right)=\vol_{X|D} \left(\langle \alpha \rangle\right)$ due to \cref{ex:volrestDalphavol}.
\end{proof}

Note that in the situation of \cref{cor:volalphamvolT1Tn}, it could happen that $\vol_{X|D}(\langle \alpha \rangle)=0$. In order to get a meaningful lower bound in this case, we shall derive a different estimate following the same idea.

\begin{theorem}\label{thm:voldiff_prod1}
Assume that $n>1$.
    Let $\alpha_1,\dots,\alpha_n\in \mathrm{H}^{1,1}(X,\mathbb{R})$ be big classes. Consider closed positive $(1,1)$-currents $T_i\in \alpha_i$ for all $i=1,\ldots,n$. Fix a prime divisor $D$ over $X$. Then
    \begin{equation}\label{eq:alpha_T_mixedvol}
        \begin{aligned}
        &\langle \alpha_1,\ldots,\alpha_n\rangle -\vol(T_1,\ldots,T_n)\\
            \geq & \frac{1}{2^{n-1}}\prod_{i=1}^n \Bigl( \nu(T_i,D)-\nu(\alpha_i,D) \Bigr) \max_{k=1,\ldots,n}\prod_{j\neq k} \left(  \frac{\vol \alpha_j }{\wid(\alpha_j,D)^n}\right)^{1/(n-1)}.
        \end{aligned}
    \end{equation}
\end{theorem}
In the proof below, by choosing better values of $c$, we can slightly improve the inequality, we leave the details to the interested readers.
\begin{proof}
    Since the problem is invariant after replacing $X$ by a modification, we may assume that $D$ is a prime divisor on $X$.

     After replacing $\alpha_i$ by $\alpha_i-\nu(\alpha_i,D)\{D\}$ and $T_i$ by $T_i-\nu(\alpha_i,D)[D]$, we may assume that $\nu(\alpha_i,D)=0$ for all $i=1,\ldots,n$.

    We may assume that $\nu(T_i,D)>0$ for all $i$ since there is nothing to prove otherwise. 
    Fix a constant $c\in (0,1)$ for the moment.
    
    Now let $S_i'$ be a current with minimal singularities in $\alpha_i-c\nu(T_i,D)\{D\}$. Then $\nu(S'_i,D)=0$ thanks to \cref{cor:res_vol_vol_ineq}. We write $S_i=S'_i+c\nu(T_i,D)[D]\in \alpha_i$. Note that $T_i\preceq_{\mathcal{I}} S_i$.
    Let $S_{\min}\in \alpha_1$  be a current with minimal singularities.
    Then
    \[
    \begin{aligned}
         &\langle \alpha_1,\ldots,\alpha_n\rangle -\vol\left(T_1,\ldots,T_n\right)\\
         \geq & \vol\left(S_{\min},S_2,\ldots,S_n\right)-\vol\left(T_1,S_2,\ldots,S_n\right)\\
         \geq & \nu(T_1,D)\vol_{X|D}\Bigl(\mathbb{D}(S_2),\ldots,\mathbb{D}(S_n)\Bigr)\\
         \geq & \nu(T_1,D) \prod_{j=2}^n \vol_{X|D}\Bigl(\alpha_j-c\nu(T_j,D)\{D\}\Bigr)^{1/(n-1)}\\
         \geq &  \nu(T_1,D) \prod_{j=2}^n \left(  \frac{\vol \alpha_j }{\nu_{\max}\left(\alpha_j,D\right)^n}\cdot \min\Bigl\{c\nu(T_j,D),\nu_{\max}(\alpha_j,D)-c\nu(T_j,D)\Bigr\}^{n-1}\right)^{1/(n-1)}\\
         =& \nu(T_1,D) \prod_{j=2}^n \left( \left(  \frac{\vol \alpha_j }{\nu_{\max}\left(\alpha_j,D\right)^n}\right)^{1/(n-1)}\cdot \min\Bigl\{c\nu(T_j,D),\nu_{\max}(\alpha_j,D)-c\nu(T_j,D)\Bigr\}\right).
    \end{aligned}
    \]
    where the first inequality follows from \cite[Proposition~3.6]{Xiabdiv}, the second follows from \cref{thm:bdivint_diff_ineq}, the third follows from \cref{prop:BM} and \cref{ex:volrestDalphavol}, the fourth follows from \cref{cor:res_vol_vol_ineq}. 
    
    Next take $c=1/2$, then since for each $j=2,\ldots,n$, we have
    \[
    \frac{1}{2}\nu(T_j,D)\leq \nu_{\max}(\alpha_j,D)-\frac{1}{2}\nu(T_j,D),
    \]
    we can continue the estimate
    \[
    \begin{aligned}
         &\langle \alpha_1,\ldots,\alpha_n\rangle -\vol\left(T_1,\ldots,T_n\right)\\
         \geq &\nu(T_1,D) \prod_{j=2}^n \left( \left(  \frac{\vol \alpha_j }{\nu_{\max}(\alpha_j,D)^n}\right)^{1/(n-1)}\cdot \frac{1}{2}\nu(T_j,D)\right)\\
         = & \frac{1}{2^{n-1}}\prod_{i=1}^n\nu(T_i,D) \cdot \prod_{j=2}^n \left(  \frac{\vol \alpha_j }{\nu_{\max}(\alpha_j,D)^n}\right)^{1/(n-1)}.
    \end{aligned}
    \]
    The desired inequality follows by symmetry.
\end{proof}
\begin{remark}\label{rmk:alphai_equal}
    When the $T_i$'s are all the same, then we can in fact remove the factor $2^{1-n}$ in \eqref{eq:alpha_T_mixedvol}. In fact, we may assume that $D$ is a prime divisor on $X$ and is the first component in a smooth flag. By considering the Okounkov bodies $\Delta(T)\subseteq \Delta(\alpha)$, the desired inequality translates immediately to the following: Suppose that $Q\subseteq P$ are two convex bodies in $\mathbb{R}^n$. Then
    \begin{equation}\label{eq:volPvolQ_diff_rmk}
    \vol P-\vol Q\geq \vol P\cdot \left( \frac{\min_{x\in P}x_1-\min_{x\in Q}x_1}{\max_{x\in P}x_1-\min_{x\in P}x_1} \right)^n.
    \end{equation}
    In order to show \eqref{eq:volPvolQ_diff_rmk}, we may assume without loss of generality that 
    \[
    \min_{x\in P}x_1=0,\quad \max_{x\in P}x_1=1.
    \]
    Let 
    \[
    f(t)=\vol\left\{y\in \mathbb{R}^{n-1}:(t,y)\in P\right\}^{1/(n-1)},\quad t\in [0,1].
    \]
    Then $f$ is concave and non-negative, and \eqref{eq:volPvolQ_diff_rmk} follows from the following elementary fact: The function
    $\frac{\int_0^s f^{n-1}}{s^n}$
    is decreasing with $s\in (0,1]$.

    Note that this argument is not valid when the $T_i$'s are different.
\end{remark}

\begin{corollary}\label{cor:mixed_loss_mass}
Assume that $n>1$.
    Let $\alpha_1,\dots,\alpha_n\in \mathrm{H}^{1,1}(X,\mathbb{R})$ be big classes. Consider closed positive $(1,1)$-currents $S_i,T_i\in \alpha_i$. Assume that $T_i\preceq_{\mathcal{I}}S_i$.
    Fix a prime divisor $D$ over $X$. Assume that $\vol S_j>0$ for all $j=2,\ldots,n$.
    Then
    \begin{equation}\label{eq:volSandT_diff_final}
        \begin{aligned}
        &\vol(S_1,\ldots,S_n) -\vol(T_1,\ldots,T_n)\\
            \geq & \frac{1}{2^{n-1}}\prod_{i=1}^n \Bigl( \nu(T_i,D)-\nu(S_i,D) \Bigr) \prod_{j=2}^n \left(  \frac{\vol S_j }{\Bigl(\nu_{\max}(\alpha_j,D)-\nu(S_j,D)\Bigr)^n}\right)^{1/(n-1)}.
        \end{aligned}
    \end{equation}
\end{corollary}
Observe that under our assumptions, for $j=2,\ldots,n$, we have
\[
\nu_{\max}(\alpha_j,D)-\nu(S_j,D)>0.
\]
In fact, since 
\[
\vol \Bigl( S_j-\nu(S_j,D)[D] \Bigr)=\vol S_j>0,
\]
the class $\alpha_j-\nu(S_j,D)\{D\}$ is big. Therefore,
\[
\nu_{\max}\Bigl( \alpha_j-\nu(S_j,D)\{D\},D \Bigr)=\nu_{\max}\left(\alpha_j,D\right)-\nu(S_j,D)>0
\]
by \cref{prop:numaxalpha}(2) and (3).

\begin{proof}
    Since the problem is invariant after replacing $X$ by a modification and $D$, we may assume that $D$ is a prime divisor on $X$.

    When $S_1,\ldots,S_n$ have analytic singularities, from the bimeromorphic invariance of both sides of \eqref{eq:volSandT_diff_final}, we may reduce to the case where $S_1,\ldots,S_n$ have log singularities: Write
    \[
    S_i=[D_i]+R_i,
    \]
    where $D_i$ is an effective divisor and $R_i$ is a closed positive $(1,1)$-current with locally bounded potentials. Note that for $i>0$, $\vol \{R_i\}=\vol S_i>0$, and hence $\wid(\{R_i\},D)>0$ thanks to \cref{prop:numaxalpha}(2).

    By \cref{thm:voldiff_prod1}, we have
    \[
    \begin{aligned}
    & \vol\left(S_1,\ldots,S_n\right)-\vol\left(T_1,\ldots,T_n\right)\\
    = & \{R_1\}\cap \dots \cap \{R_n\}-\vol\Bigl(T_1-[D_1],\ldots,T_n-[D_n]\Bigr)\\
    \geq & \frac{1}{2^{n-1}}\prod_{i=1}^n \Bigl( \nu(T_i,D)-\nu(S_i,D) \Bigr) \cdot \prod_{j=2}^n \left( \frac{\vol \{R_j\}}{\wid\left(\{R_j\},D\right)} \right)^{1/(n-1)}
    \end{aligned}
    \]
    Observe that for $j=2,\ldots,n$,
    \[
    \vol \{R_j\}=\vol S_j,\quad \nu_{\max}\left( \{R_j\},D \right)+\nu\left(S_j,D\right)\leq \nu_{\max}\left( \{S_j\},D \right),
    \]
    our assertion \eqref{eq:volSandT_diff_final} follows.

    Next we handle the case where $S_1,\ldots,S_n$ are Kähler currents. Take quasi-equisingular approximations $(S_i^k)_{k>0}$ of $S_i$ for each $i=1,\ldots,n$, the the desired inequality holds with $S_i^k$ in place of $S_i$. Let $k\to\infty$, the assertion for the $S_i$'s follows. Here we have applied the continuity of Lelong numbers, as proved in \cite[Theorem~6.2.4]{Xiabook}, and the continuity of the volumes as proved in \cite[Proposition~3.10]{Xiabdiv}.

    Finally, consider the general case. Fix a Kähler form $\omega$ on $X$. Then for any $\epsilon>0$, we have
     \[
        \begin{aligned}
        &\vol\left(S_1+\epsilon\omega,\ldots,S_n+\epsilon\omega\right) -\vol\left(T_1+\epsilon\omega,\ldots,T_n+\epsilon\omega\right)\\
            \geq &  \frac{1}{2^{n-1}}\prod_{i=1}^n \Bigl( \nu(T_i,D)-\nu(S_i,D) \Bigr) \prod_{j=2}^n \left(  \frac{\vol (S_j+\epsilon\omega) }{\Bigl(\nu_{\max}(\alpha_j+\epsilon\{\omega\},D)-\nu(S_j,D)\Bigr)^n}\right)^{1/(n-1)}.
        \end{aligned}
    \]
    Letting $\epsilon\to 0+$ and applying \cref{prop:numaxalpha}(5), we conclude \eqref{eq:volSandT_diff_final}.
\end{proof}

The following case is probably the most important for applications:
\begin{corollary}
    Let $\alpha\in \mathrm{H}^{1,1}(X,\mathbb{R})$ be a big class. Consider closed positive $(1,1)$-currents $S,T\in \alpha$ with $T\preceq_{\mathcal{I}}S$. Fix a prime divisor $D$ over $X$, then
    \begin{equation}\label{eq:volSvolTdifffinal}
    \vol S-\vol T\geq \Bigl( \nu(T,D)-\nu(S,D) \Bigr)^n\cdot  \frac{\vol S }{2^{n-1}\cdot \wid(\alpha,D)^n}.
    \end{equation}
\end{corollary}
If we use \cref{rmk:alphai_equal}, then we can remove the factor $2^{n-1}$ on the right-hand side.

\begin{proof}
We first assume that $n>1$. When $\vol S>0$, \eqref{eq:volSvolTdifffinal} follows immediate from \cref{cor:mixed_loss_mass}. When $\vol S=0$, there is nothing to prove.

Finally when $n=1$, we consider the partial Okounkov bodies $\Delta(S)$ and $\Delta(T)$ with respect to the flag $D$. See the algebraic approach in \cite{Xia21} or the transcendental approach in  \cite[Chapter~10]{Xiabook}. Since we are in dimension $1$, the cohomology class $\alpha$ is necessarily algebraic, so the algebraic theory in \cite{Xia21} actually suffices, at least when $\alpha$ is a rational class.

Let $\Delta(\alpha)$ be the corresponding Okounkov body of $\alpha$.
Then $\Delta(\alpha)\supseteq \Delta(S)\supseteq \Delta(T)$ and they are (possibly degenerate) closed intervals. The situation is summarized in \cref{fig:Oko}.
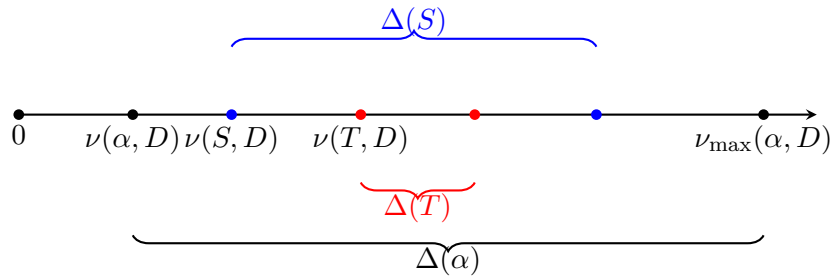
\begin{figure}[htbp]
\centering
\begin{tikzpicture}[>=stealth, thick]

\draw[->] (0,0) -- (10.5,0);

\fill (0,0) circle (2pt);
\fill (1.5,0) circle (2pt);

\fill[blue] (2.8,0) circle (2pt);   
\fill[red]  (4.5,0) circle (2pt);   

\fill[red]  (6.0,0) circle (2pt);   
\fill[blue] (7.6,0) circle (2pt);   

\fill       (9.8,0) circle (2pt);   

\node[below] at (0,0) {$0$};
\node[below] at (1.5,0) {$\nu(\alpha,D)$};
\node[below] at (2.8,0) {$\nu(S,D)$};
\node[below] at (4.5,0) {$\nu(T,D)$};
\node[below] at (9.8,0) {$\nu_{\max}(\alpha,D)$};

\draw[blue, decorate, decoration={brace, amplitude=6pt}]
  (2.8,0.9) -- (7.6,0.9);
\node[blue, above] at (5.2,0.9) {$\Delta(S)$};

\draw[red, decorate, decoration={brace, amplitude=6pt}]
  (6.0,-0.9) -- (4.5,-0.9);
\node[red, below] at (5.25,-0.9) {$\Delta(T)$};

\draw[decorate, decoration={brace, amplitude=6pt}]
  (9.8,-1.6) -- (1.5,-1.6);
\node[below] at (5.65,-1.6) {$\Delta(\alpha)$};

\end{tikzpicture}
\caption{The Okounkov bodies}
\label{fig:Oko}
\end{figure}

Then using the theory of Okounkov bodies, our assertion \eqref{eq:volSvolTdifffinal} translates into
\[
\vol \Delta(S)-\vol \Delta(T)\geq \Bigl( \nu(T,D)-\nu(S,D) \Bigr) \cdot \frac{\vol \Delta(S)}{\nu_{\max}(\alpha,D)-\nu(\alpha,D)}.
\]
From the picture, it is clear that
\[
\vol \Delta(S)\leq \nu_{\max}(\alpha,D)-\nu(\alpha,D),\quad  \vol \Delta(S)-\vol \Delta(T)\geq \nu(T,D)-\nu(S,D),
\]
and our assertion follows.
\end{proof}

Note that the constant 
\[
\frac{1}{\wid(\alpha,D)^n}
\]
appearing in \eqref{eq:volSvolTdifffinal} depends continuously on $\alpha$, as follows from \cref{prop:numaxalpha}(6). Our theorem therefore implies the main theorem of \cite{Vu23} when $D$ is taken as the exceptional divisor of the blow-up at a point.

The dependence on $\vol S$ in the right-hand of \eqref{eq:volSvolTdifffinal} seems optimal, in view of \cref{ex:count_Su}.

\section{Toric monotonicity theorem}\label{sec:toric_mono}
In this section, we will explain \cref{thm:main_thm} in the toric setting. As we shall see, our main theorem recovers the well-known Minkowski volume formula of convex bodies due to Aleksandrov \cite{Ale39} and Fenchel--Jessen \cite{FJ38}. See \cite{Sch14} for a modern treatment. The readers are assumed to be familiar with the toric pluripotential developed in \cite[Chapter~12]{Xiabook}.

We shall continue to use the notations of \cref{subsec:toric}. In particular, $N$ is a $n$-dimensional lattice and $M$ is its dual, $X$ is a smooth projective toric variety corresponding to a fan $\Sigma$ in $N_{\mathbb{R}}$. Let $T_c$ be the compact toric contained in the dense torus in $X$.

We recall a few standard terminologies from convex geometry. Given a convex body $P\subseteq M_{\mathbb{R}}$, the support function $\Supp_P\colon N_{\mathbb{R}}\rightarrow \mathbb{R}$ is defined by
\[
\Supp_P(n)\coloneqq \max_{m\in P} \langle m,n\rangle.
\]
Given $u\in N_{\mathbb{R}}\setminus \{0\}$, we write $F(P,u)$ for the support hyperplane of $P$ in the direction of $u$, namely
\[
F(P,u)\coloneqq \left\{m\in M_{\mathbb{R}}: \langle m,u \rangle=\Supp_P(u) \right\}.
\]

We first observe that \cref{thm:main_thm} requires only toric-invariant divisors.
\begin{theorem}\label{thm:main_toric}
    Let $S$ be a $T_c$-invariant closed positive $(1,1)$-currents on $X$ with positive volume, $T,T'$ be $T_c$-invariant closed positive $(1,1)$-currents on $X$ in the same cohomology class such that $T\preceq_{\mathcal{I}}T'$. Then
        \begin{equation}\label{eq:Ddiffprecise_toric}
    \begin{split}
        &\Bigl(\mathbb{D}(S)^{n-1} \cap \mathbb{D}(T')\Bigr)-\Bigl(\mathbb{D}(S)^{n-1} \cap \mathbb{D}(T)\Bigr)\\
        =&\lim_{\pi\colon Y\rightarrow X}{}' \sum_{D\subseteq Y}{}' \Bigl(\nu(T,D)-\nu(T',D) \Bigr) \vol_{Y|D}\Bigl(\mathbb{D}(S)_Y \Bigr),
    \end{split}
    \end{equation}
    where the primes refer to the fact that we only consider toric-invariant modifications and toric-invariant prime divisors.
\end{theorem}
\begin{proof}
    The $\geq$ direction is a direct consequence of \cref{thm:bdivint_diff_ineq}. The reverse inequality follows almost \emph{verbatim} from the argument of \cref{thm:main_thm}.
\end{proof}

We are now in a position to translate everything into the language of convex bodies. Just for clarity, we assume that $\{S\}$ is a big integral class in the Néron--Severi group, so that \cite[Chapter~12]{Xiabook} can be applied without modification. The general case follows by simple approximations.

Fix a basis of $N$ and define a norm on $N_{\mathbb{R}}$ by making this basis orthonormal.
As shown in \cite[Theorem~12.3.3]{Xiabook}, the mixed volumes on the left-hand side of \eqref{eq:Ddiffprecise_toric} can be expressed as the mixed volumes of the corresponding Newton bodies:
\begin{equation}\label{eq:voldiff_toric_gen1}
\left(\mathbb{D}(S)^{n-1}\cap \mathbb{D}(T') \right)-\left(\mathbb{D}(S)^{n-1}\cap \mathbb{D}(T) \right)=n!\Biggl( \vol\left( \Delta(S)^{n-1},\Delta(T') \right)-\vol\left( \Delta(S)^{n-1},\Delta(T) \right)\Biggr).
\end{equation}

As for the right-hand side, consider a toric-invariant prime divisor $D\subseteq Y$.
Let $u_D\in N$ be the primitive vector corresponding to $D$. Then $\nu(T,D)-\nu(T',D)$ is nothing but the difference between the support functions of the Newton bodies:
\begin{equation}\label{eq:Leldiff_toric_gen1}
\nu(T,D)-\nu(T',D)=\Supp_{\Delta(T')}(-u_D)-\Supp_{\Delta(T)}(-u_D),
\end{equation}
as a consequence of \cite[Corollary~12.2.1]{Xiabook}. 

Thanks to \cite[Corollary~12.3.8]{Xiabook}, we have
\begin{equation}\label{eq:resvol_toric_gen1}
\vol_{|D}\left(\mathbb{D}(S)_Y^{n-1}\right)=(n-1)! \vol F\Bigl( \Delta\left(\mathbb{D}(S)_Y\right),-u_D \Bigr),
\end{equation}
where the volume on the right-hand side is the $(n-1)$-dimensional volume on the normal hyperplane to $u_D$ so that the smallest cube of the restriction of $M$ has volume $1$.  Note that the $(n-1)$-dimensional Hausdorff measure of the latter cube is exactly the norm of $u_D$. Therefore,  we could rescale the variable $-u_D$ in the support functions in \eqref{eq:Leldiff_toric_gen1} to a unit vector and replace $\vol$ in \eqref{eq:resvol_toric_gen1} by the $(n-1)$-dimensional Hausdorff measure at the same time, without changing their product as on the right-hand side of \eqref{eq:Ddiffprecise_toric}.

Therefore, combining \eqref{eq:resvol_toric_gen1} with \eqref{eq:Leldiff_toric_gen1}, we can rewrite 
\[
\begin{aligned}
& \sum_{D\subseteq Y}{}' \Bigl(\nu(T,D)-\nu(T',D) \Bigr) \vol_{Y|D}\Bigl(\mathbb{D}(T_1)_Y\cap \dots \cap \mathbb{D}(T_{n-1})_Y \Bigr)\\
= &(n-1)!\int_{\mathbb{S}^{n-1}} \left( \Supp_{\Delta(T')}(v)-\Supp_{\Delta(T)}(v) \right)\,\mathrm{d}\mathcal{S}\Bigl(\Delta\left(\mathbb{D}(S)_Y\right)\Bigr)(v),
\end{aligned}
\]
where $\mathcal{S}(\Delta\left(\mathbb{D}(S)_Y\right))$ is the area measure of the polytope $\Delta\left(\mathbb{D}(S)_Y\right)$ (the Newton polytope of an arbitrary toric invariant $\mathbb{Q}$-divisor on $Y$ representing the class $\mathbb{D}(S)_Y$, well-defined up to translation), as defined in \cite[Section~4.2]{Sch14}. For those who are not familiar with the area formula $\mathcal{S}$, it suffices to use the simple formula \cite[{(4.24)}]{Sch14} as the definition.

Taking the limit with respect to $Y$ and using the weak continuity of the area measure, we find
\[
\begin{aligned}
&\lim_{\pi\colon Y\rightarrow X}{}' \sum_{D\subseteq Y}{}' \Bigl(\nu(T,D)-\nu(T',D) \Bigr) \vol_{|D}\Bigl(\mathbb{D}(T_1)_Y\cap \dots \cap \mathbb{D}(T_{n-1})_Y \Bigr)\\
=& (n-1)!\int_{\mathbb{S}^{n-1}} \left( \Supp_{\Delta(T')}(v)-\Supp_{\Delta(T)}(v) \right)\,\mathrm{d}\mathcal{S}\Bigl(\Delta(S)\Bigr)(v),
\end{aligned}
\]
Therefore, in view of \eqref{eq:voldiff_toric_gen1}, \eqref{eq:Ddiffprecise_toric} translates into
\begin{equation}
\begin{aligned}
     &\vol\left( \Delta(S)^{n-1},\Delta(T') \right)-\vol\left( \Delta(S)^{n-1},\Delta(T) \right)\\
     =&\frac{1}{n}\int_{\mathbb{S}^{n-1}}\left( \Supp_{\Delta(T')}(v)-\Supp_{\Delta(T)}(v) \right)\,\mathrm{d}\mathcal{S}\Bigl(\Delta(S)\Bigr)(v).
\end{aligned}
\end{equation}
By polarization, we get
\begin{equation}
\begin{aligned}
     &\vol\Bigl( \Delta(T_1),\ldots,\Delta(T_{n-1}),\Delta(T') \Bigr)-\vol\Bigl(  \Delta(T_1),\ldots,\Delta(T_{n-1}),\Delta(T) \Bigr)\\
     =&\frac{1}{n}\int_{\mathbb{S}^{n-1}}\left( \Supp_{\Delta(T')}(v)-\Supp_{\Delta(T)}(v) \right)\,\mathrm{d}\mathcal{S}\Bigl( \Delta(T_1),\ldots,\Delta(T_{n-1})\Bigr)(v).
\end{aligned}
\end{equation}
Here we have used the mixed area measure $\mathcal{S}$ as defined in \cite[Theorem~5.1.7]{Sch14}.

Note that arbitrary convex bodies $P_1,\ldots,P_{n-1}$ can be realized as $\Delta(T_1),\ldots,\Delta(T_{n-1})$, and similarly arbitrary pairs of convex bodies $Q'\supseteq Q$ can be realized as $\Delta(T')\supseteq \Delta(T)$ as in \cref{thm:main_toric}. In fact, it suffices to take a big toric-invariant section $H$ of $\mathcal{O}(k)$ for some large enough $k$, so that the corresponding Newton body $\Delta(H)$ contains $P_1,\ldots,P_{n-1},Q',Q$ in the interior. Then the existence of the currents in $c_1(\mathcal{O}(k))$ follow immediately from \cite[Corollary~12.3.3]{Xiabook}. We conclude the following  consequence:
\begin{corollary}\label{cor:Minkow_volume_formula}
    Consider convex bodies $P_1,\ldots,P_{n-1},Q',Q\subseteq \mathbb{R}^n$. Assume that $Q'\supseteq Q$, then
    \begin{equation}\label{eq:toric_mixarea}
\begin{aligned}
     &\vol\left( P_1,\ldots,P_{n-1},Q'\right)-\vol\left(  P_1,\ldots,P_{n-1}, Q\right)\\
     =&\frac{1}{n}\int_{\mathbb{S}^{n-1}}\left( \Supp_{Q'}-\Supp_{Q} \right)\,\mathrm{d}\mathcal{S}\left(P_1,\ldots,P_{n-1}\right).
\end{aligned}
\end{equation}
\end{corollary}
The argument sketched above only works when $P_1,\ldots,P_{n-1}$ have full dimensions. The general case follows by a simple approximation.

We have therefore completely recovered the Minkowski volume formula on Page~282 of \cite{Sch14}. In fact, our main theorem gives a new construction of the mixed area measures $\mathcal{S}\left(P_1,\ldots,P_{n-1}\right)$ of convex bodies themselves!

In the non-toric setting, we expect that there be a similar limit measure between nef b-divisors defined on the Berkovich analytification of $X$. These measures should play similar roles in pluripotential theory as the mixed area measures play in convex geometry.
\appendix

\section{Positive forms}\label{sec:posform}

We briefly recall the notion of positive forms due to Lelong \cite{Lel68} and Harvey--Knapp \cite{HK74}. The latter reference, on which a large part of Demailly's textbook \cite[Chapter~III]{Dembook2} was based, was unfortunately omitted from the reference list.

Let $V$ be a finite dimensional complex vector space of dimension $n$. Let $V^{\vee}$ denote the complex linear space consisting of $\mathbb{R}$-linear functional $V\rightarrow \mathbb{R}$ on $V$. 
For each pair of non-negative integers $p,q$, we write $\Lambda^{p,q}V^{\vee}$ for the space of $(p,q)$-forms in $\Lambda^{p+q}(V^{\vee}\otimes_{\mathbb{R}}\mathbb{C})$. Note that the complex dual of $V$ is $\Lambda^{1,0}V^{\vee}$.

Fix an integer $p\in \{0,1,\ldots,n\}$ in the sequel.
\begin{definition}
A form $\alpha\in \Lambda^{p,p}V^{\vee}$ is 
\begin{enumerate}
    \item \emph{weakly positive} if for each $\beta_1,\ldots,\beta_{n-p}\in V^{\vee}$, we have
    \[
    \alpha\wedge \mathrm{i}\beta_1\wedge \overline{\beta_1}\wedge \cdots \wedge \mathrm{i}\beta_{n-p}\wedge \overline{\beta_{n-p}}
    \]
    gives the positive orientation of $V$;
    \item \emph{strongly positive} if $\alpha$ is a finite linear combination with $\mathbb{R}_{\geq 0}$ coefficients of terms of the form
    \begin{equation}\label{eq:basic_form}
    \mathrm{i}\beta_1\wedge \overline{\beta_1}\wedge \cdots \wedge \mathrm{i}\beta_{p}\wedge \overline{\beta_{p}},
    \end{equation}
    where $\beta_1,\ldots,\beta_{p}\in V^{\vee}$. A form like \eqref{eq:basic_form} is called a \emph{basic form}.
    \item \emph{positive} if for one (hence all) basis $w_1,\ldots,w_n$ of $\Lambda^{1,0} V^{\vee}$, when we expand $\alpha$ as
    \begin{equation}\label{eq:pos_form_mat}
    \alpha=(-1)^{p(p-1)/2}\sum_{|I|=|J|=p} \alpha_{I,J} \mathrm{i}^p w_I\wedge \overline{w_J},
    \end{equation}
    the matrix $(\alpha_{I,J})_{I,J}$ is Hermitian and positive semidefinite. Here $w_I=w_{i_1}\wedge \cdots \wedge w_{i_p}$, where $i_1<i_2<\cdots<i_p$ are the elements in $I$, and $\overline{\beta_J}$ is similar.
    \end{enumerate}
\end{definition}
\begin{remark}
    The terminology needs some clarification. In \cite{Lel68}, Lelong only introduced the notions of weakly positive forms and strongly positive forms. He called the weakly positive forms \emph{positive forms}. Lelong's terminology is largely followed by the modern schools in complex geometry.

    The notion of positive forms as above is introduced by Harvey--Knapp. The terminology of Harvey--Knapp is, by contrast, largely followed by people in non-Archimedean geometry and tropical geometry.
\end{remark}

When we represent a weakly positive form in the form of matrices as in \eqref{eq:pos_form_mat}, the matrix is always Hermitian. See \cite[Page~167, Corollary~1.5]{Dembook2}.

It is immediate from the definitions that a strongly positive form is positive, and a positive form is weakly positive. Furthermore, the cone of weakly positive forms is the dual cone of that of strongly positive forms. 

When $p=0,1,n-1,n$, all three notions are equivalent. For all other value of $p$, namely $p=2,3,\ldots,n-2$, all three notions are different. This was originally a question of Lelong, and proved in \cite{HK74}.

The wedge product preserves the various positivities as \cref{table:wed_product_pos}. In general, the product between a positive form and a weakly positive form is no longer weakly positive, as follows from the arguments of \cite[Page~49]{HK74}.
\begin{table}[ht]
\centering
\begin{tabular}{|c|c|c|}
\hline
$\alpha$ 
& $\beta$ 
& $\alpha \wedge \beta$ \\
\hline
Strongly positive & Strongly positive & Strongly positive \\
\hline
Positive & Positive & Positive \\
\hline
Strongly positive & Weakly positive & Weakly positive \\
\hline
\end{tabular}
\caption{Positivity properties of wedge products}
\label{table:wed_product_pos}
\end{table}

It is easy to see that weakly positive and positive forms form closed  convex cones. 
As for the case of strongly positive forms, this has always been a folklore result. Harvey--Knapp \cite{HK74} mentioned this on Page~49 without giving the details.
We take this opportunity to give a complete proof along the lines of Harvey--Knapp, following the answer on MathOverflow \cite{Strongpos_form}.
\begin{lemma}\label{lma:strongpos_closed}
    The set of strongly positive $(p,p)$-forms on $V$ form a closed convex cone.
\end{lemma}
Therefore, the cone of strongly positive forms is the dual of that of weakly positive forms.

We need the following straightforward consequence in the main body of the paper: A $(p,p)$-form on a complex manifold is strongly positive as a form if and only if it is strongly positive as a current, namely the pairing with any compactly supported weakly positive form with bidimension $(p,p)$ is non-negative.

This consequence has already been widely applied in the related literature without rigorous justification.
\begin{remark}
    A very similar argument works in the non-Archimedean setting, and gives a rigorous proof to the assertion about strongly positive Lagerberg forms (as defined in \cite{Lag12}) in the second version of \cite[Section~1.2.4]{CLD12}.
\end{remark}
\begin{proof}
    The non-trivial point is to show that the cone of strongly positive $(p,p)$-forms is closed.
    
    Fix an identification $V\cong \mathbb{C}^n$ so that $V$ gets a Hermitian norm. 

    \textbf{Step~1}. We show that
    \[
    C\coloneqq \left\{ \alpha\in \Lambda^{p,p}V^{\vee} : \alpha=\mathrm{i}\beta_1\wedge \overline{\beta_1}\wedge \cdots \wedge \mathrm{i}\beta_{p}\wedge \overline{\beta_{p}}\textup{ for some }\beta_1,\ldots,\beta_p\in  V^{\vee}\right\}
    \]
    is closed.

    Let 
    \[
    C'\coloneqq \left\{ \alpha\in \Lambda^{p,p}V^{\vee} : \alpha=\mathrm{i}^{p^2} \gamma\wedge \overline{\gamma} \textup{ for some }\gamma\in  \Lambda^{p,0}V^{\vee}\right\}.
    \]
    Note that $C\subseteq C'$.

    We first show that $C'$ is closed. Consider a Cauchy sequence $(\mathrm{i}^{p^2}\gamma_j\wedge \bar{\gamma_j})_j$ in $C$. Fix a basis $w_1,\ldots,w_n$ of $V^{\vee}$. Then we can write
    \[
    \gamma_j=\sum_{|I|=p} \gamma_{j,I} w_I.
    \]
    The boundedness of $\mathrm{i}^{p^2}\gamma_j\wedge \bar{\gamma_j}$ implies the boundedness of $\gamma_{j,I}$ for each fixed $I$. Therefore, after subtracting a subsequence, we may assume that $\gamma_{j,I}\to \gamma_{I}$ for some $\gamma_I\in \mathbb{C}$. Then
    \[
    \mathrm{i}^{p^2}\gamma_j\wedge \bar{\gamma_j}\to \mathrm{i}^{p^2}\gamma\wedge \bar{\gamma}.
    \]
    Next we show that $C$ is closed. This means, suppose that $\gamma_j$ is decomposable then we want to show that $\gamma$ is also decomposable. But the condition of being decomposable is equivalent to finitely many polynomial relations between the coefficients, classically known as the Plücker relations, which pass through limits. See \cite[Section~3.1.E]{GKZ08} for details.\footnote{More conceptually, the decomposability of $\gamma$ also follows from the fact that the Plücker embedding of the Grasmannian is a closed immersion.} Our assertion follows.

    \textbf{Step~2}. Let $S$ be the unit sphere in $V$. Then from Step~1, the intersection $S\cap C$ is compact.

    We claim that the convex hull of $S\cap C$ does not contain the origin. Suppose that this fails, then we can find $v_i\in S\cap C$ and $\lambda_i>0$ ($i=1,\ldots,m$) so that
    \[
    \sum_{i=1}^m \lambda_i v_i=0.
    \]
    But under the matrix representation \eqref{eq:pos_form_mat}, the $v_i$'s can be regarded as strictly positive matrices, so this cannot happen.

    Our assertion now follows from \cref{lma:convexcone_closed}.
\end{proof}

\begin{lemma}\label{lma:convexcone_closed}
    Let $C$ be a convex body in $\mathbb{R}^n$, $0\not\in C$, then the convex cone generated by $C$ union with $0$ is closed.
\end{lemma}
\begin{proof}
    Let $\lambda_i>0$ and $c_i\in C$. Assume that $\lambda_ic_i$ converges to $c\in \mathbb{R}^n$. Then we need to show that either $c=0$ or $c$ is in the convex cone generated by $c$. 

    We may assume that $c\neq 0$. Then since $C$ is compact and $0\not\in C$, then $|c_i|$ is a bounded sequence, bounded away from $0$ as well. But $\lambda_i c_i$ has a non-zero limit, this means $\lambda_i$ is also bounded and bounded away from $0$. After replacing everything by a subsequence, we may assume that $\lambda_i\to \lambda>0$, $c_i\to c'\in C$. Then $c=\lambda c'$, and our assertion follows.
\end{proof}

A symmetric form can always be decomposed as the difference of two strongly positive forms. More generally, we have the following:
\begin{theorem}\label{thm:HK}
    Fix a basis $w_1,\ldots,w_n$, then $\Lambda^{p,p}V^{\vee}$ has a basis consisting of forms of the form
    \[
    \mathrm{i}^{p^2} \beta_{1}\wedge \cdots \wedge \beta_p\wedge \overline{\beta_{1}}\wedge \cdots \wedge \overline{\beta_p},
    \]
    where $\beta_i$ is one of $w_i\pm w_j$ or $w_i\pm \mathrm{i}w_j$ with $i,j=1,\ldots,n$.
\end{theorem}
This result was first proved by Harvey--Knapp \cite[Proposition~1.9]{HK74}. See also Demailly \cite[Page~167, Lemma~1.4]{Dembook2}. 
The corresponding statement fails in the setup of Lagerberg forms, as previously asserted in the first version of \cite{CLD12}. See \cite{Ber25} for the details.

\clearpage

\printbibliography

Mingchen Xia, \textsc{Institute of Geometry and Physics, USTC}\par\nopagebreak
  \textit{Email address}, \texttt{xiamingchen2008@gmail.com}\par\nopagebreak
  \textit{Homepage}, \url{https://mingchenxia.github.io/}.

\end{document}